\documentclass[twoside,11pt]{article}

\usepackage{geometry}
\usepackage{graphicx}
\usepackage{indentfirst}                                
\usepackage[colorlinks]{hyperref}
\hypersetup{linkcolor=blue,filecolor=black,urlcolor=blue, citecolor=black}   
\usepackage[numbers,sort&compress]{natbib}         
\usepackage[perpage,symbol*,marginal]{footmisc}   

\usepackage{amsmath}                                
\usepackage{amsfonts}
\usepackage{amssymb}                                
\usepackage{bm}                                            
\usepackage{mathrsfs}                                 
\usepackage{amsthm}                                   
\usepackage{accents}                                    
\usepackage{amsxtra}

\newcommand\w[1]{\makebox[1em]{$#1$}}

\ifxetex
\usepackage{letltxmacro}
\LetLtxMacro\SavedIncludeGraphics\includegraphics
\def\includegraphics#1#{
 \IncludeGraphicsAux{#1}%
}%
\newcommand*{\IncludeGraphicsAux}[2]{%
 \XeTeXLinkBox{%
  \SavedIncludeGraphics#1{#2}%
}}
\fi

\newcommand\orcidicon[1]{\href{https://orcid.org/#1}{\includegraphics[scale=0.02]{orcid.pdf}}}

\def\d{\,\mathrm{d}}
\def\p{\partial}

\makeatletter
\def\wideubar{\underaccent{{\cc@style\underline{\mskip10mu}}}}
\def\Wideubar{\underaccent{{\cc@style\underline{\mskip8mu}}}}
\makeatother
\makeatletter
\def\widebar{\accentset{{\cc@style\underline{\mskip10mu}}}}
\def\Widebar{\accentset{{\cc@style\underline{\mskip8mu}}}}
\makeatother

\newcommand*\xbar[1]{%
\hbox{%
\vbox{%
\hrule height 0.4pt 
\kern0.3ex
\hbox{%
\kern-0.1em
\ensuremath{#1}%
\kern-0.1em
}}}}

\usepackage[titletoc,title]{appendix}
\usepackage{titlesec}
\titleformat{\section}{\large\bfseries\centering}{\thesection}{1em}{}
\titleformat{\subsection}{\it\centering\bfseries}{\thesubsection}{0.5em}{}
\titleformat{\subsubsection}[runin]{\bfseries}{\thesubsubsection.}{0.4em}{}[.]

\numberwithin{equation}{section}
\newtheorem{lemma}{Lemma}[section]
\newtheorem{proposition}[lemma]{Proposition}
\newtheorem{theorem}{Theorem}[section]
\newtheorem{corollary}[lemma]{Corollary}
\newtheorem{definition}{Definition}[section]
\newtheorem{remark}{Remark}[section]

\usepackage{color}
\definecolor{DarkRed}{RGB}{139,0,0}
\definecolor{Purple}{RGB}{128,0,128}
\usepackage{ulem}

\begin{document}

\title{\bf Well-Posedness of Free Boundary Problem in Non-relativistic and Relativistic Ideal Compressible Magnetohydrodynamics\let\thefootnote\relax\footnotetext{
The research of {\sc Yuri Trakhinin} 
was partially supported by RFBR (Russian Foundation for Basic Research) under Grant 19-01-00261-a and by Mathematical Center in Akademgorodok under agreement No.~075-15-2019-1675 with the Ministry of Science and Higher Education of the Russian Federation.
The research of {\sc Tao Wang}
was partially supported by the National Natural Science Foundation of China under Grants 11971359 and 11731008.}
}

\author{
{\sc Yuri Trakhinin}\thanks{e-mail: trakhin@math.nsc.ru}\\
{\footnotesize Sobolev Institute of Mathematics, Koptyug av.~4, 630090 Novosibirsk, Russia}\\
{\footnotesize Novosibirsk State University, Pirogova str.~1, 630090 Novosibirsk, Russia}\\[2mm]
{\sc Tao Wang}\thanks{e-mail: tao.wang@whu.edu.cn}\\
{\footnotesize School of Mathematics and Statistics, Wuhan University, Wuhan 430072, China}
}

\date{\today}
\maketitle

\begin{abstract}

We consider the free boundary problem for non-relativistic and relativistic ideal compressible magnetohydrodynamics in two and three spatial dimensions with the total pressure vanishing on the plasma--vacuum interface. We establish the local-in-time existence and uniqueness of solutions to this nonlinear characteristic hyperbolic problem under the Rayleigh--Taylor sign condition on the total pressure. The proof is based on certain tame estimates in anisotropic Sobolev spaces for the linearized problem and a modification of the Nash--Moser iteration scheme. Our result is uniform in the light speed and appears to be the first well-posedness result for the free boundary problem in ideal compressible magnetohydrodynamics with zero total pressure on the moving boundary.

\vspace*{4mm}


\vspace*{2mm}

\end{abstract}

\setcounter{tocdepth}{3} 
\tableofcontents

\section{Introduction}\label{sec:intro}

This paper concerns the well-posedness of the free boundary problem for a plasma--vacuum interface in  non-relativistic and relativistic, ideal ({\it i.e.}, neglecting the effect of viscosity and electrical resistivity), compressible magnetohydrodynamics (MHD) for the space dimension $d=2,3$.

Let $\varOmega(t)\subset \mathbb{R}^d$ denote the moving domain occupied by the plasma.
We first consider the following equations of non-relativistic ideal compressible MHD (see {\sc Landau--Lifshitz} \cite[\S 65]{LL84MR766230}):
\begin{subequations} \label{MHD1}
\begin{alignat}{3}
	&\p_t \rho+\nabla\cdot (\rho v)=0
	&& \textrm{in } \varOmega(t), \label{MHD1a} \\
	&\p_t (\rho v)+\nabla\cdot (\rho v\otimes v-H\otimes H)+\nabla q=0
	&& \textrm{in } \varOmega(t), \label{MHD1b} \\
	&\p_t H-\nabla\times(v\times H)=0
	&& \textrm{in } \varOmega(t), \label{MHD1c} \\
	&\p_t(\rho E+\tfrac{1}{2}|H|^2) +\nabla\cdot ( v  (\rho E+p)+H\times(v\times H) )=0
	&\quad& \textrm{in } \varOmega(t), \label{MHD1d}
\end{alignat}
\end{subequations}
together with
\begin{align} \label{divH1}
	\nabla\cdot H=0 \quad \textrm{in } \varOmega(t),
\end{align}
which describe the motion of a perfectly conducting fluid in a magnetic field.
Here the density $\rho$, velocity $v\in\mathbb{R}^d$, magnetic field $H\in\mathbb{R}^d$, and pressure $p$ are unknown functions of the time $t$ and spatial variables $x=(x_1,\ldots,x_d)$.
The symbols $q=p+\frac{1}{2}|H|^2$ and $E=e+\frac{1}{2}|v|^2$ denote the total pressure and specific total energy, respectively,
where $e$ is the specific internal energy.
Equations \eqref{MHD1} constitute a closed system of $2d+2$ conservation laws through the smooth constitutive relations $\rho=\rho(p, S)$ and $e=e(p, S)$,
where $S$ is the specific entropy.
The thermodynamic variables $\rho$, $p$, $e$, and $S$ satisfy the Gibbs relation
\begin{align} \label{Gibbs}
	\vartheta \d S=\d e-\frac{p}{\rho^2} \d \rho,
\end{align}
where $\vartheta>0$ is the absolute temperature.
Identity \eqref{divH1} is preserved by the evolution, provided it holds at the initial time.
Throughout this paper, we denote by $\p_t:=\p/\p t$ the time derivative and by $\nabla:=(\p_1,\ldots,\p_d)^{\mathsf{T}}$ the gradient with $\p_i:=\p/\p x_i$.
See Appendix \ref{sec:appA} for the conventional notation in the vector calculus.

System \eqref{MHD1} is supplemented with the initial conditions
\begin{align} \label{IC1}
	\varOmega(0)=\varOmega_0,\quad
	(\rho,v,H,S)|_{t=0}=(\rho_0,v_0,H_0,S_0) \textrm{ on } \varOmega_0,
\end{align}
where the domain $\varOmega_0\subset\mathbb{R}^d$ and the initial data $(\rho_0,v_0,H_0,S_0)$ are already given and satisfy the constraint $\nabla\cdot H_0=0$.
On the free vacuum boundary $\varSigma(t)$, we require the following boundary conditions:
\begin{subequations} \label{BC1}
\begin{alignat}{3}
	&v\cdot \bm{n}=\mathcal{V} (\varSigma(t))
	&& \textrm{on } \varSigma(t), \label{BC1a} \\
	&q=0
	&& \textrm{on } \varSigma(t), \label{BC1b} \\
	&H\cdot \bm{n}=0
	&\quad& \textrm{on } \varSigma(t), \label{BC1c}
\end{alignat}
\end{subequations}
where $\bm{n}$ is the exterior unit normal to $\varOmega(t)$ on the boundary $\varSigma(t)$ and $\mathcal{V}(\varSigma(t))$ is the normal velocity of $\varSigma(t)$.
Condition \eqref{BC1a} means that the interface $\varSigma(t)$ moves with the fluid.
Condition \eqref{BC1b} comes from the vanishing vacuum magnetic field;
see {\sc Goedbloed et al.}~\cite[\S 4.6]{GKP19} for physical models of  plasma--vacuum interfaces.
Condition \eqref{BC1c} corresponds to the constraint \eqref{divH1} and should also be regarded as a constraint on the initial data.

We aim to prove the local-in-time existence and uniqueness of solutions to the free boundary problem \eqref{MHD1}, \eqref{IC1}--\eqref{BC1b}, 
provided the hyperbolicity assumption $\rho|_{\varOmega(t)}>\rho_*>0$ and the following Rayleigh--Taylor sign condition on the total pressure hold initially:
\begin{align} \label{RT1}
	\nabla_{\bm{n}} q \leq -\kappa_{0} <0
	\quad \textrm{on }\varSigma(t),
\end{align}
where $\rho_*$ and $\kappa_{0}$ are positive constants, and $\nabla_{\bm{n}} :=\bm{n}\cdot \nabla$.
This problem is a  nonlinear hyperbolic problem with a free {\it characteristic} boundary due to the condition \eqref{BC1a}.
Assumption \eqref{RT1} is also the natural physical condition for the incompressible MHD;
see {\sc Hao--Luo} \cite{HL14MR3187678} for {\it a priori} estimates through a geometrical point of view and {\sc Gu--Wang} \cite{GW19MR3980843} for local well-posedness.
We also refer to {\sc Hao--Luo} \cite{HL19} for a recent ill-posedness result of the 2D incompressible MHD when the condition \eqref{RT1} is violated.
For the vacuum free boundary problem of incompressible and compressible liquids, 
without magnetic fields, the natural physical assumption reads
\begin{align} \label{RT.p}
	\nabla_{\bm{n}} p \leq -\kappa_{0} <0
	\quad \textrm{on }\varSigma(t).
\end{align}
See, for instance, the works \cite{L05AnnMR2178961,CS07MR2291920, ZZ08MR2410409} and \cite{L05CMPMR2177323,T09CPAMMR2560044} respectively for incompressible and compressible liquids with a free surface under the assumption \eqref{RT.p}.
It is important to point out that under the initial constraint $H|_{\varSigma(0)}= 0$ (implying $H|_{\varSigma(t)}=0$ for $t>0$),
the stability condition \eqref{RT1} can be reduced to \eqref{RT.p}.

Our construction of solutions is motivated by that of the first author in \cite{T09ARMAMR2481071,T09CPAMMR2560044} for compressible current-vortex sheets and compressible Euler equations in vacuum.
The approach involves, in particular, the reduction to a fixed domain, the application of the ``good unknown'' of {\sc Alinhac} \cite{A89MR976971}, suitable tame estimates in certain function spaces for the linearized problem, and an appropriate Nash--Moser iteration scheme.
This approach has also been applied to the study of rarefaction waves \cite{A89MR976971}, compressible vortex sheets \cite{CS08MR2423311}, compressible current-vortex sheets \cite{CW08MR2372810,T05MR2187618,T09ARMAMR2481071}, MHD contact discontinuities \cite{MTT18MR3766987}, relativistic vortex sheets \cite{CSW19MR3925528}, among others.
More precisely, we consider the liquids, for which the density $\rho$ is supposed to be uniformly bounded from below by a positive constant, 
so that the equations \eqref{MHD1} can be rewritten equivalently as a symmetric hyperbolic system for sufficiently smooth solutions.
Furthermore, we suppose that the moving boundary $\varSigma(t)$ has the form of a graph, 
which enables us to reduce the problem \eqref{MHD1}, \eqref{IC1}--\eqref{BC1b} to a fixed domain by the standard partial hodograph transformation.
We first study the well-posedness for the effective linear problem,
resulting from the variable coefficient linearized problem by use of the ``good unknown'' and neglect of some zero-th order terms.
In the basic $L^2$ estimate, the sign condition \eqref{RT1} provides a good term for the interface function. 
In general, there is a loss of normal derivatives in higher-order energy estimates for hyperbolic problems with a characteristic boundary.
Especially for ideal compressible MHD, {\sc Ohno--Shirota} \cite{OS98MR1658400} have proved that a characteristic fixed-boundary linearized problem is always ill-posed in the standard Sobolev spaces $H^m$ for $m\geq 2$.
Rather than in the standard Sobolev spaces $H^m$, we shall work in the anisotropic Sobolev spaces $H_*^m$, first introduced by {\sc Chen} \cite{C07MR2289911}.
These function spaces, taking into account the loss of normal derivatives to the characteristic boundary, turn out to be appropriate for investigating symmetric hyperbolic, characteristic problems.  
See {\sc Secchi} \cite{S96MR1405665} for a general theory and \cite{YM91MR1092572,S95MR1314493,T09ARMAMR2481071,CW08MR2372810} for the study of other characteristic problems in ideal compressible MHD.
Having the well-posedness and tame estimates for the effective linear problem in hand, 
we deduce the local existence and uniqueness of solutions for our nonlinear problem ({\it cf.}~Theorem \ref{thm:1}) by a modification of the Nash--Moser iteration scheme.

Moreover, we can employ the approach outlined above to show a counterpart of Theorem \ref{thm:1} for the relativistic version of the problem \eqref{MHD1}, \eqref{IC1}--\eqref{BC1b} in the Minkowski spacetime $\mathbb{R}^{1+d}$ ({\it cf.}~Theorem \ref{thm:2}).
The main ingredient is that one can symmetrize the following equations of ideal relativistic magnetohydrodynamics (RMHD)
(see {\sc Lichnerowicz} \cite[\S\S 30--34]{Lichnerowicz67}):
\begin{subequations} \label{RMHD1}
	\begin{alignat}{3}
	&\nabla_{\alpha} (\rho u^{\alpha})=0
	\quad &&\textnormal{(conservation of matter)}, \label{RMHD1a} \\
	&\nabla_{\alpha}  T^{\alpha\beta}=0
	\quad &&\textnormal{(conservation of energy--momentum)}, \label{RMHD1b} \\
	&\nabla_{\alpha}(u^{\alpha}b^{\beta}-u^{\beta}b^{\alpha})=0
	\quad &&\textnormal{(relevant Maxwell's equations)}. \label{RMHD1c}
	\end{alignat}
\end{subequations}
The symmetrization has been derived in {\sc Freist\"{u}hler--Trakhinin} \cite{FT13MR3044369} by properly applying the Lorentz transformation;
also see Appendix \ref{sec:appB} for a direct verification.
Here $\nabla_{\alpha}$ is the covariant derivative with respect to the Minkowski metric $(g_{\alpha \beta})$ for
\begin{align*}
(g_{\alpha \beta}):=
	\left\{
	\begin{aligned}
	&\mathrm{diag}\,(-1,1,1)\quad &&\textrm{if } d=2,\\
	&\mathrm{diag}\,(-1,1,1,1)\quad &&\textrm{if }d=3,
	\end{aligned}
	\right.
\end{align*}
the symbol $\rho$ is the particle number density, $u^{\alpha}$ and $b^{\alpha}$ are, respectively, the components of the $d$-velocity and the magnetic field $d$-vector with respect to the plasma velocity such that
\begin{align} \label{u.relation}
	g_{\alpha \beta} u^{\alpha} u^{\beta}=-1,
	\quad g_{\alpha \beta} u^{\alpha} b^{\beta}=0.
\end{align}
The total energy--momentum tensor $T^{\alpha\beta}$ has the form
\begin{align} \nonumber 
	T^{\alpha\beta}=(\rho h+|b|^2)u^{\alpha}u^{\beta}
	+\epsilon^2 q g^{\alpha\beta} -b^{\alpha}b^{\beta},
\end{align}
where $h:=1+\epsilon^{2}(e+ p/\rho)$ is the index of the fluid, 
$\epsilon^{-1}$ is the speed of light, 
$e$ is the specific internal energy, 
$p$ is the pressure, 
$|b|^2:=g_{\alpha \beta} b^{\alpha} b^{\beta}$, 
and $q:=p+\frac{1}{2} \epsilon^{-2}|b|^2$ is the total pressure.
The density $\rho=\rho(p,S)$ and internal energy $e=e(p,S)$ are given smooth functions of $p$ and $S$ that satisfy the Gibbs relation \eqref{Gibbs}.
Throughout this paper, we adopt the Einstein summation convention, for which Greek and Latin indices range from $0$ to $d$ and from $1$ to $d$ respectively.

It is worth mentioning that our result is uniform in the light speed and appears to be the first well-posedness result for the free boundary problem in ideal compressible magnetohydrodynamics with zero total pressure on the moving vacuum boundary.
However, our approach relies heavily on the hyperbolicity condition $\rho|_{\varOmega(t)}> \rho_*>0$,
and it is an open problem to extend our result to the motion of gases, namely $\rho|_{\varSigma(t)}=0$.
For ideal gases without magnetic fields, 
the local well-posedness for the 3D compressible Euler equations with a physical vacuum boundary has been independently established by {\sc Coutand--Shkoller} \cite{CS12MR2980528} and {\sc Jang--Masmoudi} \cite{JM15MR3280249} via the Lagrangian reformulation and nonlinear energy estimates;
also see {\sc Luo et al.}~\cite{LXZ14MR3218831} for a general uniqueness result.

The case with nonvanishing total pressure on the interface corresponds to that with nonzero vacuum magnetic field $\mathcal{H}$.
Considering the pre-Maxwell equations for $\mathcal{H}$ in the vacuum region, 
{\sc Secchi--Trakhinin} have proved in \cite{ST14MR3151094} the nonlinear local well-posedness for the full plasma--vacuum interface problem in 3D ideal compressible MHD based on their linear well-posedness results in \cite{T10MR2718711,ST13MR3148595}.
Notice that the results in \cite{ST14MR3151094,ST13MR3148595} require the non-collinearity condition $|H\times \mathcal{H}|\geq \kappa_0>0$ on the interface, which stems from the current-vortex sheet problems ({\it cf.}~\cite{T05MR2187618,T09ARMAMR2481071}) and enhances the regularity of the free interface. 
Clearly, the assumption $\mathcal{H}=0$ under our consideration violates the non-collinearity condition.
Nevertheless, this assumption makes some difficult boundary terms arising in the energy estimates for the full plasma--vacuum interface model (such as the term $\widetilde{\mathcal{Q}}$ in {\sc Trakhinin} \cite[(62)--(63)]{T16MR3503661}) disappear, which enables us to solve the problem under the Rayleigh--Taylor sign condition only.
We point out that the local well-posedness of the full plasma--vacuum interface problem in ideal compressible MHD is still unknown without the non-collinearity condition; see \cite{T16MR3503661} for a thorough discussion of this issue.
As for the incompressible case, 
we refer to {\sc Morando et al.}~\cite{MTT14MR3237563} for well-posedness of the linearized problem, {\sc Hao} \cite{H17MR3614754} for  nonlinear {\it a priori} estimates, and {\sc Sun et al.}~\cite{SWZ19MR3981394} for nonlinear well-posedness.
For the plasma--vacuum interface problem in RMHD, where the vacuum electric and magnetic fields satisfy Maxwell's equations,  an {\it a priori} estimate for the linearized problem in the anisotropic Sobolev space $H_*^1$ was provided by the first author in \cite{T12MR2974767}.

The plan of this paper is as follows.
In Section \ref{sec:nonlinear}, we reformulate the problem \eqref{MHD1}, \eqref{IC1}--\eqref{BC1b} to that in a fixed domain and state the main theorems of this paper, 
namely Theorems \ref{thm:1}--\ref{thm:2}.
Section \ref{sec:linear} is devoted to the proof of Theorem \ref{thm:3},
that is, the well-posedness of the effective linear problem in anisotropic Sobolev spaces $H_*^m$.
In Section \ref{sec:Nash}, we employ a modification of the Nash--Moser iteration scheme to construct the solutions of our nonlinear problem and prove Theorem \ref{thm:1}.
In Section \ref{sec:rela}, we sketch the proof of Theorem \ref{thm:2} for ideal RMHD in vacuum.
For the self-containedness of the paper,
we collect the conventional notation of the vector calculus for two and three spatial dimensions in Appendix \ref{sec:appA} and present a direct calculation for the symmetrization of ideal RMHD in Appendix \ref{sec:appB}.

\section{Nonlinear Problems and Main Theorems} \label{sec:nonlinear}

This section is dedicated to reducing the free boundary problem \eqref{MHD1}, \eqref{IC1}--\eqref{BC1b} to a fixed domain and stating the main theorems of this paper.

Let us first introduce the symmetric hyperbolic form of the MHD equations \eqref{MHD1}.
For this purpose, we suppose that the sound speed $a:=a(\rho,S)$ is smooth and satisfies
\begin{align} \label{a.def}
	a(\rho,S):=\sqrt{\frac{\p p}{\p \rho} (\rho, S)}>0
	\quad \textrm{for all \ } \rho \in (\rho_*,\rho^*) ,
\end{align}
where $\rho_*$ and $\rho^*$ are positive constants with $\rho_*<\rho^*$.
By virtue of \eqref{divH1}--\eqref{Gibbs},
we have the following equivalent system for smooth solutions of \eqref{MHD1}:
\begin{align} \label{MHD2}
\left\{
	\begin{aligned}
	&(\p_t +v\cdot\nabla)p+\rho a^2\nabla\cdot v=0,\\
	&\rho (\p_t +v\cdot\nabla)v -(H\cdot \nabla) H+\nabla q=0,\\
	&(\p_t +v\cdot\nabla)H-(H\cdot \nabla)v+H \nabla\cdot v=0,\\
	&(\p_t +v\cdot\nabla)S=0.
	\end{aligned}
\right.
\end{align}
Thanks to \eqref{a.def}, the system \eqref{MHD2} is symmetrizable hyperbolic when
\begin{align} \label{hyperbolicity}
\rho_*< \rho < \rho^*.
\end{align}
In light of \eqref{BC1b}, as in {\sc Secchi--Trakhinin} \cite{ST13MR3148595},
we take $U:=(q,v,H,S)^{\mathsf{T}}$ as the primary unknowns and obtain from \eqref{MHD2} the symmetric system
\begin{align} \label{MHD.vec}
	A_0(U)\p_t U+A_i(U)\p_i U=0\quad \textrm{in } \varOmega(t),
\end{align}
where
\setlength{\arraycolsep}{4pt}
\begin{align}
	&A_0(U):=
	\begin{pmatrix}
	\dfrac{1}{\rho a^2} & 0 & -\dfrac{1}{\rho a^2} H^{\mathsf{T}} & 0\\[2.5mm]
	0 & \rho I_d & O_d & 0\\[1.5mm]
	-\dfrac{1}{\rho a^2} H & O_d & I_d+\dfrac{1}{\rho a^2}H\otimes H & 0\\[1.5mm]
	\w{0}  & \w{0}  & \w{0} & \w{1}
	\end{pmatrix}, 
\label{A0.def} \\
	&A_i(U):=
	\begin{pmatrix}
	\dfrac{v_i}{\rho a^2} & \bm{e}_i^{\mathsf{T}} & -\dfrac{v_i}{\rho a^2} H^{\mathsf{T}} & 0\\[2.5mm]
	\bm{e}_i & \rho v_i I_d & -H_i I_d & 0\\[1.5mm]
	-\dfrac{v_i}{\rho a^2} H & -H_i I_d & v_iI_d+\dfrac{v_i}{\rho a^2}H\otimes H & 0\\[1.5mm]
	\w{0}  & \w{0}  & \w{0} & \w{v_i}
	\end{pmatrix}
\label{Ai.def}
\end{align}
for $i=1,\ldots,d$. 
Here and below, $O_m$ and $I_m$ denote the zero and identity matrices of order $m$, respectively, 
$\bm{e}_i:=(\delta_{i1},\ldots,\delta_{id})^{\mathsf{T}}$, and $\delta_{ij}$ denotes the Kronecker delta.

For technical simplicity, 
we assume that the space domain $\varOmega(t)$ occupied by the plasma takes the form
\begin{align*}
	\varOmega(t):=
	\left\{x\in\mathbb{R}^d: x_1>\varphi(t,x'),  x':=(x_2,\ldots,x_d)\in \mathbb{T}^{d-1} \right\},
\end{align*}
where $\mathbb{T}^{d-1}$ denotes the $(d-1)$--torus and the interface function $\varphi(t,x')$ is to be determined.
Then the free interface $\varSigma(t)$ is given by
\begin{align*}
	\varSigma(t):=
	\left\{x\in\mathbb{R}^d:  x_1=\varphi(t,x'),   x'\in \mathbb{T}^{d-1} \right\}.
\end{align*}
Denoting by $N:=(1,-\p_2\varphi,\ldots,-\p_d\varphi)^{\mathsf{T}}$ the normal to $\varSigma(t)$,
we can rewrite the boundary conditions \eqref{BC1} as
\begin{alignat}{3}
	&\p_t \varphi=v_N,\quad q=0
	\quad &&\textrm{on }\varSigma(t),
\label{BC2} \\
	&H_N=0
	\quad &&\textrm{on }\varSigma(t),
\label{HN2}
\end{alignat}
where $v_N:=v\cdot N$ and $H_N:=H\cdot N$.
Moreover, the Rayleigh--Taylor sign condition \eqref{RT1} becomes
\begin{align} \label{RT2}
	N\cdot \nabla q\geq \kappa_{0} |N|>0
	\quad \textrm{on }\varSigma(t).
\end{align}
In view of \eqref{BC2}--\eqref{HN2}, 
the boundary matrix for the problem \eqref{MHD1}, \eqref{IC1}--\eqref{BC1b} on $\varSigma(t)$ reads
\begin{align*}
(\p_t\varphi  A_0(U) -N_{i}A_i(U) )|_{\varSigma(t)}
	=\left.
	\begin{pmatrix}
	0& -N^{\mathsf{T}} & 0 & 0\\
	-N & O_d &O_d &0\\
	0 & O_d &O_d &0\\
	\w{0} &\w{0} &\w{0}  &\w{0}
	\end{pmatrix}
	\right|_{\varSigma(t)},
\end{align*}
which is singular, meaning that the free boundary $\varSigma(t)$ is {\it characteristic}.
Furthermore, the boundary matrix on $\varSigma(t)$ has one positive, one negative, and $2d$ zero eigenvalues.
Since one boundary condition is necessary for determining the unknown function $\varphi$, 
we infer that the correct number of the boundary conditions is two, 
according to the well-posedness theory for hyperbolic problems.
Therefore, the condition \eqref{HN2} will be taken as a constraint on the initial data rather than a real boundary condition.

It is a standard first step to reduce the free boundary problem \eqref{MHD1}, \eqref{IC1}--\eqref{BC1b} to an equivalent problem in a fixed domain.
To this end, we replace the unknown $U$ by
\begin{align} \label{U.sharp}
	U_{\sharp}(t,x):=U(t,\varPhi(t,x),x'),
\end{align}
where as in {\sc M\'{e}tivier} \cite[p.\;70]{M01MR1842775}, we choose the function $\varPhi$ to satisfy
\begin{align} \label{varPhi.def}
	\varPhi(t,x):=x_1+\kappa_{\sharp} \chi( x_1)\varphi(t,x'),
\end{align}
with positive constant $\kappa_{\sharp}$ and $C^{\infty}_0(\mathbb{R})$--function $\chi$ satisfying
\begin{align} \label{chi}
\kappa_{\sharp} \|\varphi_0\|_{L^{\infty}(\mathbb{T}^{d-1}) }\leq  \frac{1}{4},
\quad \|\chi'\|_{L^{\infty}(\mathbb{R})} <1,
\quad \chi\equiv 1\  \textrm{on } [0,1].
\end{align}
This change of variables is admissible on the time interval $[0,T]$,
provided $T>0$ is sufficiently small so that  $\kappa_{\sharp}\|\varphi\|_{L^{\infty}([0,T]\times\mathbb{T}^{d-1}) }\leq 1/2$.
Without loss of generality we set $\kappa_{\sharp}=1$.
The cut-off function $\chi$ is introduced as in \cite{M01MR1842775,T09CPAMMR2560044} to avoid assumptions about the compact support of the initial data (shifted to a constant state).

Dropping the subscript $``\sharp"$ for convenience,
we reformulate the vacuum free boundary problem \eqref{MHD1}, \eqref{IC1}--\eqref{BC1b} as the following initial boundary value problem in a fixed domain
$\varOmega:=\{x\in \mathbb{R}^d: x_1>0, x'\in \mathbb{T}^{d-1} \}$:
\begin{subequations} \label{NP1}
\begin{alignat}{2}
	&\mathbb{L}(U,\varPhi) :=L(U,\varPhi)U =0
	&\quad &\textrm{in  }  [0,T]\times \varOmega,
\label{NP1a} \\
	&\mathbb{B}(U,\varphi):=
	\begin{pmatrix}\p_t \varphi -v_N\\ q\end{pmatrix}
	=0
	& \quad  &\textrm{on  } [0,T]\times \varSigma,
\label{NP1b} \\
	&(U,\varphi)|_{t=0}=(U_0,\varphi_0),
	& \quad &
\label{NP1c}
\end{alignat}
\end{subequations}
where $\varSigma:=\{x\in \mathbb{R}^d: x_1=0,  x'\in \mathbb{T}^{d-1} \}$ denotes the boundary and the operator $L(U,\varPhi)$ is defined by
\begin{align}
	L(U,\varPhi)&:=
	A_0(U)\partial_t+\widetilde{A}_1(U,\varPhi)\partial_1+\sum_{i=2}^d A_i(U)\partial_i, 
\nonumber 
\end{align}
with
\begin{align}
   \widetilde{A}_1(U,\varPhi)&:=
	\frac{1}{\partial_1\varPhi}\bigg(A_1(U)-\partial_t\varPhi A_0(U)-\sum_{i=2}^d\partial_i\varPhi A_i(U)\bigg).
\label{A1t.def}
\end{align}
Here, the coefficient matrices $A_{i}(U)$, for $i=0,\ldots, d$, are given in \eqref{A0.def}--\eqref{Ai.def}.
But, in the relativistic case, $A_{i}(U)$, for $i=0,\ldots, d$,  are defined by \eqref{Ai.def.r}.
In the new variables, the identities \eqref{divH1} and \eqref{HN2} are reduced to
\begin{alignat}{3}
	&H_N=0\quad &&\textrm{if  }  x_1=0,
\label{HN}\\
	&\p_{i}^{\varPhi} H_i=0\quad &&\textrm{if  }  x_1>0,
\label{divH}
\end{alignat}
where for notational simplicity we denote the partial differentials with respect to the lifting function $\varPhi$  by
\begin{align} \label{differential}
	\partial_t^{\varPhi}:=\partial_t-\frac{\partial_t\varPhi}{\partial_1\varPhi}\p_1,\   \
	\partial_1^{\varPhi}:=\frac{1}{\partial_1\varPhi}\partial_1,\   \
	\partial_i^{\varPhi}:=\partial_i-\frac{\partial_i\varPhi}{\partial_1\varPhi}\partial_1
\end{align}
for  $i=2,\ldots,d.$
Thanks to the change of variables \eqref{U.sharp}--\eqref{varPhi.def} and the second condition in \eqref{BC2}, 
the assumption \eqref{RT2} can be recovered from
\begin{align} \label{RT}
	\p_1 q\geq \kappa_{0}>0 \quad \textrm{if  }  x_1=0.
\end{align}
The following proposition indicates that the identities \eqref{HN}--\eqref{divH} can be regarded as constraints on the initial data (see \cite[Appendix A]{T09ARMAMR2481071} for the proof).
\begin{proposition} \label{pro:1}
For sufficiently smooth solutions of the problem \eqref{NP1} on the time interval $[0,T]$, the constraints \eqref{HN}--\eqref{divH} are satisfied for all $t\in [0,T]$ as long as they hold initially.
\end{proposition}

Let $\lfloor s \rfloor$ denote the floor function of $s\in\mathbb{R}$ that maps $s$ to the greatest integer less than or equal to $s$.
We are now ready to state the first main theorem of this paper.

\begin{theorem}[Non-relativistic case]
\label{thm:1}
Let $m\geq {20}$ be an integer. 
Assume that the initial data \eqref{NP1c} satisfy the hyperbolicity condition \eqref{hyperbolicity}, the constraints \eqref{HN}--\eqref{divH}, the sign condition \eqref{RT}, and the compatibility conditions up to order $m$ (see Definition \ref{def:1}).
Assume further that $(U_0-\widebar{U},\varphi_0)$ belongs to $H^{m+3/2}(\varOmega)\times H^{m+2}(\mathbb{T}^{d-1})$ for some constant state $\widebar{U}$ with $ \rho_*<\rho(\widebar{U})<\rho^*$.
Then there exists a sufficiently small constant $T>0$, 
such that the problem \eqref{NP1} with $A_i(U)$ defined by \eqref{A0.def}--\eqref{Ai.def} has a unique solution $(U,\varphi)$ on the time interval $[0,T]$ satisfying
\begin{align*}
(U-\widebar{U},\varphi)
\in H^{{\lfloor (m-9)/2\rfloor}}([0,T]\times\varOmega)\times H^{m-{9}}([0,T]\times\mathbb{T}^{d-1}).
\end{align*}
\end{theorem}

For the relativistic case, as in {\sc Freist\"{u}hler--Trakhinin} \cite{FT13MR3044369},
we impose the physical assumption that the relativistic sound speed $c_{\rm s}=c_{\rm s}(\rho,S)$ is positive and smaller than the light speed, {\it i.e.},
\begin{align} \label{cs.def}
	0<c_{\rm s}(\rho, S):=\frac{a(\rho,S) }{\sqrt{h}}<\epsilon^{-1}
	\quad \textrm{for all }  \rho \in (\rho_*,\rho^*) ,
\end{align}
where $\rho_*$ and $\rho^*$ are positive constants with $ \rho_*<\rho^*$, 
and $a(\rho,S)$ is defined by \eqref{a.def}.
The second main theorem of this paper stated below is the relativistic counterpart of Theorem \ref{thm:1}.

\begin{theorem}[Relativistic case] \label{thm:2}
Let $m\geq {20}$ be an integer.
Assume that the initial data \eqref{NP1c} satisfy the hyperbolicity condition \eqref{hyperbolicity}, the constraints \eqref{HN}--\eqref{divH}, the sign condition \eqref{RT}, the physical constraint $|v_0|<\epsilon^{-1}$, and the compatibility conditions up to order $m$ (see Definition \ref{def:1}).
Assume further that $(U_0-\widebar{U},\varphi_0)$ belongs to $H^{m+3/2}(\varOmega)\times H^{m+2}(\mathbb{T}^{d-1})$ for some constant state $\widebar{U}$ with $\rho_*<\rho(\widebar{U})<\rho^*$ and $|\bar{v}|<\epsilon^{-1}$.
Then there exists a sufficiently small constant $T>0$, 
such that the problem \eqref{NP1} with $A_i(U)$ defined by \eqref{Ai.def.r} has a unique solution $(U,\varphi)$ on the time interval $[0,T]$ satisfying
\begin{align*}
(U-\widebar{U},\varphi)\in H^{{\lfloor (m-9)/2\rfloor}}([0,T]\times\varOmega)\times H^{m-{9}}([0,T]\times\mathbb{T}^{d-1}).
\end{align*}
\end{theorem}

\begin{remark}
As a matter of fact, the constructed solution $U$ satisfies $U-\widebar{U}\in H_*^{m-{9}}([0,T]\times\varOmega)$,
and hence $U-\widebar{U}$ belongs to $H^{{\lfloor (m-9)/2\rfloor}}([0,T]\times\varOmega) $ in virtue of the embedding $H_*^m\hookrightarrow H^{\lfloor m/2 \rfloor}$
(see \S\S \ref{sec:linear.1}--\ref{sec:linear.2} for the definition and properties of anisotropic Sobolev spaces $H_*^m$).
\end{remark}

\begin{remark}
We emphasize that Theorems \ref{thm:1} and \ref{thm:2} imply corresponding results for the free boundary problem \eqref{MHD1}, \eqref{IC1}--\eqref{BC1b} and its relativistic counterpart respectively, 
because the constraint \eqref{divH} and {relation} $\p_1\varPhi\geq 1/2$ hold in $[0,T]\times \varOmega$ for a sufficiently small $T>0$.
\end{remark}

\section{Well-posedness of the Linearized Problem} \label{sec:linear}

In this section, we perform the linearization of the problem \eqref{NP1} and establish the well-posedness of the linearized problem in anisotropic Sobolev spaces $H_*^m$, that is, Theorem \ref{thm:3}.

\subsection{Main Result for the Linearized Problem} \label{sec:linear.1}
We denote $\varOmega_T:=(-\infty,T)\times \varOmega$ and $\varSigma_T:=(-\infty,T)\times\varSigma$ for $T>0$.
Let the basic state $(\mathring{U},\mathring{\varphi})$ with $\mathring{U}:=(\mathring{q},\mathring{v},\mathring{H},\mathring{S})^{\mathsf{T}}$ be sufficiently smooth and satisfy
\begin{alignat}{3}
	&\rho_*<\rho(\mathring{U})<\rho^*
	\quad &&\textrm{in } \varOmega_T,
\label{bas1a}\\
	&\p_t \mathring{\varphi}=\mathring{v}_N,\quad \mathring{H}_N=0
	\quad&&\textrm{on }\varSigma_T,
\label{bas1b}\\
	&\p_1\mathring{q}\geq   \frac{\kappa_{0}}{2}>0
	\quad&&\textrm{on }\varSigma_T,
\label{bas1c}
\end{alignat}
where
\begin{gather*}
\mathring{v}_N:=\mathring{v}\cdot \mathring{N},\quad
\mathring{H}_N:=\mathring{H}\cdot \mathring{N},\quad
\mathring{N}:=(1,-\p_2\mathring{\varphi},\ldots,-\p_d\mathring{\varphi})^{\mathsf{T}}.
\end{gather*}
We also denote $\mathring{\varPsi}:=\chi(x_1)\mathring{\varphi}(t,x')$ and $\mathring{\varPhi}:=x_1+\mathring{\varPsi},$
where $\chi\in C_0^{\infty}(\mathbb{R})$ satisfies \eqref{chi}.
Then $\p_1\mathring{\varPhi}\geq 1/2$ on $\varOmega_T$, provided
we without loss of generality assume that $\|\mathring{\varphi}\|_{L^{\infty}(\varSigma_T)}\leq 1/2$.
Moreover, we assume that
\begin{align} \label{bas1d}
\|\mathring{U}\|_{W^{3,\infty}(\varOmega_T)}+\|\mathring{\varphi}\|_{W^{4,\infty}(\varSigma_T)}\leq K
\end{align}
for some constant $K>0$.
Note that the physical constraint $|\mathring{v}|<\epsilon^{-1}$ should be additionally imposed for the relativistic case.

The linearized operators for \eqref{NP1a}--\eqref{NP1b} are defined by
\begin{align}
	&\mathbb{L}'(\mathring{U},\mathring{\varPhi})(V,\varPsi)
	:=L(\mathring{U},\mathring{\varPhi})V +\mathcal{C}(\mathring{U},\mathring{\varPhi})V
	-\frac{1}{\p_1 \mathring{\varPhi}}(L(\mathring{U},\mathring{\varPhi})\varPsi)\p_1 \mathring{U},
\label{L'.bb}\\
	&\mathbb{B}'(\mathring{U} ,\mathring{\varphi})(V,\psi)
	:=
	\begin{pmatrix}
	 \p_t \psi+\sum_{i=2}^{d}\mathring{v}_i \p_i   \psi-v\cdot\mathring{N}\\ q
	\end{pmatrix},
\label{B'.bb}
\end{align}
where  $\varPsi:=\chi(x_1)\psi(t,x')$ and
\begin{align} \nonumber 
\mathcal{C}(\mathring{U},\mathring{\varPhi})V:=
	\sum_{k=1}^{2d+2}V_k
	\bigg(
	\frac{\p A_0}{\p {U_k}}(\mathring{U}) \partial_t \mathring{U}
	+ \frac{\p \widetilde{A}_1}{\p {U_k}}(\mathring{U},\mathring{\varPhi}) \partial_1 \mathring{U}
	+\sum_{i=2}^{d}\frac{\p A_i}{\p {U_k}}(\mathring{U}) \partial_i \mathring{U}
	\bigg).
\end{align}
Following {\sc Alinhac} \cite{A89MR976971} and introducing the good unknown
\begin{align} \label{good}
	\dot{V}:=V-\frac{\partial_1\mathring{U} }{\partial_1 \mathring{\varPhi} }\varPsi,
\end{align}
we have ({\it cf.}~\cite[Proposition 1.3.1]{M01MR1842775})
\begin{align}
\mathbb{L}'(\mathring{U}, \mathring{\varPhi})(V,\varPsi)
 =L(\mathring{U}, \mathring{\varPhi})\dot{V}
+\mathcal{C}( \mathring{U},\mathring{\varPhi})\dot{V}
+\frac{\varPsi}{\partial_1\mathring{\varPhi}}
\partial_1\big(L(\mathring{U},\mathring{\varPhi} )\mathring{U}\big).
\label{Alinhac}
\end{align}
As in \cite{A89MR976971, CS08MR2423311,CSW19MR3925528,T09ARMAMR2481071},
we drop the last term in \eqref{Alinhac} and consider the {\it effective linear problem}
\begin{subequations} \label{ELP1}
\begin{alignat}{3}
	&\mathbb{L}'_{e}(\mathring{U}, \mathring{\varPhi}) \dot{V}
	:=L(\mathring{U}, \mathring{\varPhi})\dot{V}
	+\mathcal{C}( \mathring{U},\mathring{\varPhi})\dot{V}=f
	&\ \  &\textrm{if } x_1>0,
\label{ELP1a}\\
	&\mathbb{B}'_e(\mathring{U}, \mathring{\varphi}) (\dot{V},\psi)
	=g
	&&\textrm{if } x_1=0,
\label{ELP1b}\\
	&(\dot{V},\psi)=0   &&\textrm{if } t<0,
\label{ELP1c}
	\end{alignat}
\end{subequations}
where
\begin{align*}
\mathbb{B}'_e(\mathring{U}, \mathring{\varphi}) (\dot{V},\psi)
:=
\begin{pmatrix}
 \p_t \psi +\sum_{i=2}^{d}\mathring{v}_i \p_i \psi-\p_1 \mathring{v}_N \psi-\dot{v}\cdot\mathring{N}
\\[0.5mm]
\dot{q}+\p_1 \mathring{q}\psi
\end{pmatrix}.
\end{align*}

As in \cite{S95MR1314493,T09ARMAMR2481071,YM91MR1092572,CW08MR2372810} for other characteristic problems in ideal compressible MHD, 
we shall work in anisotropic Sobolev spaces $H_*^m$.
Throughout this paper, symbol ${\mathrm{D}_*^{\alpha}}$ means that $\alpha:=(\alpha_0,\ldots,\alpha_{d+1})\in\mathbb{N}^{d+2}$, and
\begin{align} \label{D*}
\mathrm{D}_*^{\alpha}:=
\p_t^{\alpha_0} (\sigma \p_1)^{\alpha_1}\p_2^{\alpha_2} \cdots \p_d^{\alpha_d} \p_1^{\alpha_{d+1}},\ 
\langle \alpha \rangle :=|\alpha|+\alpha_{d+1},\ 
|\alpha|:=\sum_{i=0}^{d+1} \alpha_{i},
\end{align}
where $\sigma=\sigma(x_1)$ is an increasing smooth function on $[0,+\infty)$ and satisfies $\sigma(x_1)=x_1$ for $0\leq x_1\leq 1/2$ and $\sigma(x_1)=1$ for $x_1\geq 1$.
For any integer $m\in\mathbb{N}$ and interval $I\subset \mathbb{R}$, 
function space $H_*^{m}(I\times \varOmega)$ is defined by
\begin{align*}
&  H_*^m(I\times\varOmega):=
\{ u\in L^2(I\times\varOmega):\, \mathrm{D}_*^{\alpha} u\in L^2(I\times\varOmega) \textrm{ for } \langle \alpha \rangle\leq m  \},
\end{align*}
and equipped with the norm ${\|}\cdot{\|}_{H^m_*(I\times \varOmega)}$,  
where
\begin{align} 
& {\|}u{\|}_{H^m_*(I\times \varOmega)}^2 :=
\sum_{\langle \alpha\rangle\leq m} \|\mathrm{D}_*^{\alpha} u\|_{L^2(I \times\varOmega)}^2.
\label{norm1b}
\end{align}
We will write $\|{u}\|_{m,*,t}:={\|}u{\|}_{H^m_*( \varOmega_t)}$ for short.
By definition, we have
\begin{alignat*}{3}
&H^m(I\times \varOmega)\hookrightarrow H_*^m(I\times \varOmega) \hookrightarrow H^{\lfloor m/2\rfloor}(I\times \varOmega)\quad &&\textrm{for all } m\in \mathbb{N},\ I\subset \mathbb{R}.
\end{alignat*}

We are going to prove the following result in this section.
\begin{theorem} \label{thm:3}
Let $K_0>0$ and $m\in\mathbb{N}$ with $m\geq 6$.
Then there exist constants $T_0>0$ and $C(K_0)>0$ such that if the basic state $(\mathring{U}, \mathring{\varphi})$ satisfies \eqref{bas1a}--\eqref{bas1d},
$(\mathring{V}, \mathring{\varphi})\in H_*^{{m+4}}(\varOmega_T)\times H^{{m+4}}(\varSigma_T)$,
and
\begin{align} \label{thm3.H1}
{\|}\mathring{V}{\|}_{{10},*,T}+\|\mathring{\varphi}\|_{H^{{10}}(\varSigma_T)}\leq {K}_0 \quad
\textrm{for } \mathring{V}:=\mathring{U}-\widebar{U},
\end{align}
and the source terms $f\in H_*^{m}(\varOmega_T)$, $g\in H^{m+1}(\varSigma_T)$ vanish in the past,
for some $0<T\leq T_0$,
then the problem \eqref{ELP1} has a unique solution $(\dot{V},\psi)\in H_*^{m}(\varOmega_T)\times H^{m}(\varSigma_T)$ satisfying the tame estimate
\begin{align}
\nonumber
 	{\|}(\dot{V},\varPsi){\|}_{m,*,T}&+\|\psi\|_{H^{m}(\varSigma_T)}\\
\nonumber 
 	\leq   C(K_0)\Big\{ &  {\|}f{\|}_{m,*,T} +\|g\|_{H^{m+1}(\varSigma_T)} \\ 
&  +{\|}(\mathring{V},\mathring{\varPsi}){\|}_{{m+4},*,T}\big({\|}f{\|}_{6,*,T}+\|g\|_{H^{7}(\varSigma_T)}\big) \Big\}.
\label{tame}
\end{align}
\end{theorem}
In the above theorem, 
the condition that the source terms $f$ and $g$ vanish in the past corresponds to the case of zero initial data. 
The case of general initial data is postponed to the subsequent nonlinear analysis.
In the rest of this section, we will focus on the three-dimensional case, 
because the 2D case ($d=2$) can be analyzed in the same way.

\subsection{Properties of Anisotropic Sobolev Spaces} \label{sec:linear.2}

We shall deduce the tame estimate in $H_*^m(\varOmega_T)$ for the problem \eqref{ELP1} with $m$ large enough. 
For this purpose, in this subsection we collect the Moser-type calculus inequalities for the spaces $H^m $ and $H_*^m $, and the embedding and trace theorems for $H_*^m$.
We will employ the symbol $A\lesssim B$ (or $B\gtrsim A$) to mean that $A \leq CB$ holds uniformly for some universal positive constant $C$.

\begin{lemma}[Moser-type calculus inequalities for $H^m$] \label{lem:Moser1}
	Let $m\in \mathbb{N}_+$. 
	Let $\mathcal{O}$ be an open subset of $\mathbb{R}^n$ with Lipschitz boundary.
	Assume that $F$ is $C^{\infty}$ in a neighborhood of the origin with $F(0)=0$,
	and that $u,w\in H^m(\mathcal{O}) $ with $\|u\|_{L^{\infty}(\mathcal{O})}\leq M$ for some constant $M>0$. 
	Then
	\begin{alignat}{3}
	\nonumber 
	\|\p^{\alpha} u\p^{\beta} w\|_{L^{2}}+ \| uw\|_{H^m}
	&\lesssim \|u\|_{H^{m}}\|w\|_{L^{\infty}}+\|u\|_{L^{\infty}}\|w\|_{H^{m}},\\
	\nonumber 
	\|F(u)\|_{H^{m}}&\leq C(M)\|u\|_{H^{m}},
	\end{alignat}
	for any multi-indices $\alpha,\beta\in\mathbb{N}^n$ with $|\alpha|+|\beta|\leq m$.
\end{lemma}

For the proof of the last lemma, we refer to, for instance, \cite[pp.\;84--89]{AG07MR2304160}.

\begin{lemma}[Moser-type calculus inequalities for $H_*^m$] \label{lem:Moser2}
	Let $m\in \mathbb{N}_+$.
	Assume that $F$ is $C^{\infty}$ in a neighborhood of the origin with $F(0)=0$, 
	and the functions $u,w$ belong to $H_*^m(\varOmega_T) $  and satisfy
	\begin{align} \label{norm.W*}
\|u\|_{W_*^{1,\infty}(\varOmega_T) }
:=\sum_{\langle \alpha\rangle\leq 1}\| \mathrm{D}_*^{\alpha} u\|_{L^{\infty}(\varOmega_T) } 
\leq M_*
	\end{align}
	for some constant $M_*>0$, 
	where $\mathrm{D}_*^{\alpha}$ and $\langle\alpha\rangle$ are defined in \eqref{D*}.
	Then
	\begin{align}
	\|\mathrm{D}_*^{\alpha} u  \mathrm{D}_*^{\beta} w\|_{L^2(\varOmega_T)}
	&\lesssim  {\|}u{\|}_{m,*,T}\|w\|_{W_*^{1,\infty}(\varOmega_T) } +{\|}w{\|}_{m,*,T}\|u\|_{W_*^{1,\infty}(\varOmega_T) },
\label{Moser1} \\
	{\|}u w{\|}_{m,*,T}
	&\lesssim  {\|}u{\|}_{m,*,T}\|w\|_{W_*^{1,\infty}(\varOmega_T) } +{\|}w{\|}_{m,*,T}\|u\|_{W_*^{1,\infty}(\varOmega_T) },
\label{Moser2} \\
	{\|}F(u){\|}_{m,*,T}&\leq C(M_*){\|}u{\|}_{m,*,T},
\label{Moser3}
	\end{align}
	for any multi-indices $\alpha,\beta\in\mathbb{N}^{{5}}$ with $\langle \alpha\rangle+\langle\beta\rangle\leq m$.
\end{lemma}
One can find the proof of the inequalities \eqref{Moser1}--\eqref{Moser2} in \cite[Theorem B.3]{MST09MR2604255} and the proof of \eqref{Moser3} in \cite[Appendix B]{T09ARMAMR2481071}.

\begin{lemma}[Embedding theorem for $H_*^m$] \label{lem:embed}
	The following inequalities hold:
	\begin{alignat}{3}
	\|u\|_{L^{\infty}(\varOmega_T)}&\,\lesssim {\|}u{\|}_{3,*,T},
	\quad \|u\|_{W_*^{1,\infty}(\varOmega_T)}&\,\lesssim {\|}u{\|}_{4,*,T},
\label{embed1} \\
	\|u\|_{W^{1,\infty}(\varOmega_T)}&\,\lesssim {\|}u{\|}_{5,*,T},
	\quad \|u\|_{W_*^{2,\infty}(\varOmega_T)}&\,\lesssim {\|}u{\|}_{6,*,T},
\label{embed2}
	\end{alignat}
	where $\|u\|_{W_*^{1,\infty}(\varOmega_T) }$ is defined by \eqref{norm.W*}, and
	$$
	\|u\|_{W_*^{2,\infty}(\varOmega_T) }:=\sum_{\langle \alpha\rangle\leq 1}\| \mathrm{D}_*^{\alpha} u\|_{W^{1,\infty}(\varOmega_T) }.
	$$
\end{lemma}
Thanks to \cite[Theorem B.4]{MST09MR2604255} and $\varOmega_T\subset \mathbb{R}^4$,
we obtain the first inequality in \eqref{embed1}, which implies the second one by definition.
Observing that
\begin{align*}
\|u\|_{W^{1,\infty}(\varOmega_T)}\leq \sum_{\langle \alpha\rangle\leq 2}\| \mathrm{D}_*^{\alpha} u\|_{L^{\infty}(\varOmega_T) },
\end{align*}
we derive \eqref{embed2} from the first inequality in \eqref{embed1}.

For deriving higher-order energy estimates, we also need to use the following trace theorem for the anisotropic Sobolev spaces $H_*^m$.
\begin{lemma}[{\cite[Theorem 1]{OSY94MR1289186}}] \label{lem:trace}
	Let $m\geq 1$ be an integer.
	\begin{itemize}
		\item[a)]  If $u\in H_*^{m+1}(\varOmega_T)$, then its trace $u|_{x_1=0}$ belongs to $H^{m}(\varSigma_T)$ and  satisfies $\|u|_{x_1=0}\|_{H^{m}(\varSigma_T)}\lesssim \|u\|_{m+1,*,T}$.
		\item[b)]  There exists a continuous operator $\mathfrak{R}_T: H^{m}(\varSigma_T) \to H_*^{m+1}(\varOmega_T)$ such that $(\mathfrak{R}_T w)|_{x_1=0}=w $ and $\|\mathfrak{R}_T w\|_{m+1,*,T}\lesssim \|w\|_{H^{m}(\varSigma_T)}$.
	\end{itemize}
\end{lemma}

\subsection{Well-posedness in $L^2$}

It is more convenient to reduce the problem \eqref{ELP1} to the case with homogeneous boundary conditions. 
More precisely, for the source term $g=(g_1,g_2)^{\mathsf{T}}\in H^{m+1}(\varSigma_T)$, 
we employ Lemma~\ref{lem:trace} to find a function $V_{\natural}:=(q_{\natural},v_{\natural}, 0,0,0,0)^{\mathsf{T}}$ in $H_*^{m+2}(\varOmega_T)$ with $q_{\natural}:=\mathfrak{R}_Tg_2$ and $v_{\natural}:=(\mathfrak{R}_T(-g_1),0,0)^{\mathsf{T}}$, vanishing in the past and satisfying
\begin{align} \label{V.natural}
\mathbb{B}'_e(\mathring{U}, \mathring{\varphi}) (V_{\natural},0)\big|_{\varSigma_T}
=\begin{pmatrix}
-v_{\natural}\cdot \mathring{N}\\
q_{\natural}
\end{pmatrix}\bigg|_{\varSigma_T}
=g,
\ \ 
{\|}V_{\natural}{\|}_{s+2,*,T}\lesssim \|g\|_{H^{s+1}(\varSigma_T)}
\end{align}
for all $s\in\{0,\ldots,m\}.$
Then the vector $V_{\flat}:=\dot{V}-V_{\natural}$ solves the problem \eqref{ELP1} with zero boundary source term and new internal source term $\tilde{f}$, that is,
\begin{subequations} \label{ELP2}
	\begin{alignat}{3}
		\label{ELP2a}
		&\mathbb{L}'_{e}(\mathring{U}, \mathring{\varPhi}) {V}=\tilde{f}
		:= f-\mathbb{L}_{e}'(\mathring{U},\mathring{\varPhi})V_{\natural}
		&\quad &\textrm{if } x_1>0,\\
		\label{ELP2b}
		&\mathbb{B}'_e(\mathring{U}, \mathring{\varphi}) ({V},\psi)=0
		&&\textrm{if } x_1=0,\\
		\label{ELP2c}
		&({V},\psi)=0   &&\textrm{if } t<0,
	\end{alignat}
\end{subequations}
where we have dropped subscript ``$\flat$'' for notational simplicity.
Furthermore, we shall introduce a new unknown $W$ in order to separate the noncharacteristic variables from others for the problem \eqref{ELP2}. 
To be more precise, we set
\begin{align} \nonumber
	W:=(q, v_1-\p_2\mathring{\varPhi}v_2-\p_3\mathring{\varPhi}v_3, v_2, v_3, H_1,H_2,H_3, S)^{\mathsf{T}},
\end{align}
or equivalently,
\begin{align} \label{J.ring}
	W:=\mathring{J}^{-1}V\quad \textrm{with }\mathring{J}:=
	\begin{pmatrix}1 & 0 & 0 & 0 & 0 & 0 & 0 & 0  \\
		0 & 1 & \partial_2\mathring{\varPhi}  & \partial_3\mathring{\varPhi}  & 0 & 0 & 0 & 0  \\
		0 & 0 & 1 & 0 & 0 & 0 & 0 & 0  \\
		0 & 0 & 0 & 1 & 0 & 0 & 0 & 0  \\
		0 & 0 & 0 & 0 & 1& 0 & 0 & 0  \\
		0& 0 & 0 & 0 &0  & 1& 0 & 0  \\
		0& 0 & 0 & 0 & 0  & 0 & 1& 0  \\
		\w0 & \w0 & \w0 & \w0 & \w0 & \w0 & \w0 & 1
	\end{pmatrix}.
\end{align}
Then the problem \eqref{ELP2} can be reduced to
\begin{subequations}
	\label{ELP3}
	\begin{alignat}{3}
		\label{ELP3a}
		&{\bm{L}}W:={\bm{A}}_0\p_tW+{\bm{A}}_i\p_i W +{\bm{A}}_4 W =\mathring{J}^{\mathsf{T}}\tilde{f}&\quad &\textnormal{in }\varOmega_T,\\
		\label{ELP3b}
		&W_1=-\p_1\mathring{q} \psi &\quad &\textnormal{on }\varSigma_T,\\
		\label{ELP3c}
		&W_2=(\p_t+\mathring{v}_2\p_2+\mathring{v}_3\p_3)\psi-\p_1\mathring{v}_N \psi
		&\quad &\textnormal{on }\varSigma_T,\\	
		\label{ELP3d}
		&(W,\psi)=0 &\quad &\textnormal{if }t<0,			
	\end{alignat}
\end{subequations}
where
${\bm{A}}_1:=\mathring{J}^{\mathsf{T}}\widetilde{A}_1(\mathring{U},\mathring{\varPhi})\mathring{J}$,
${\bm{A}}_4:=\mathring{J}^{\mathsf{T}}\mathbb{L}_e'(\mathring{U},\mathring{\varPhi})\mathring{J}$,
and
${\bm{A}}_{i}:=\mathring{J}^{\mathsf{T}}{A}_{i}(\mathring{U})\mathring{J}$
for $i=0,2,3$.
Notice that the system \eqref{ELP3a} is still symmetric hyperbolic.
By virtue of \eqref{bas1b}, we have
\begin{align} \nonumber 
	\widetilde{A}_1(\mathring{U},\mathring{\varPhi})=
	\begin{pmatrix}
		0 &  \mathring{N}^{\mathsf{T}} & 0 &0 \\
		\mathring{N} &  O_3 &O_3  &0\\
		0 & O_3  & O_3 & 0\\
		\w{0} & \w{0} & \w{0} & \w{0}
	\end{pmatrix} \quad
	\textrm{on } \varSigma_T,
\end{align}
which implies the following decomposition:
\begin{align}
	\label{decom}
	{\bm{A}}_1={\bm{A}}_1^{(1)}+{\bm{A}}_1^{(0)}
	\quad \textrm{with } {\bm{A}}_1^{(0)}\big|_{x_1=0}=0,\
	{\bm{A}}_1^{(1)}:=
	\begin{pmatrix}
		0 & 1 &0 \\
		1 & 0 & 0\\
		0 & 0 & O_6
	\end{pmatrix}.
\end{align}
According to the kernel of the matrix ${\bm{A}}_1$ on the boundary $\varSigma_T$, 
we denote by $W_{\rm nc}:=(W_1,W_2)^{\mathsf{T}}$ the noncharacteristic variables.
The boundary matrix for the problem \eqref{ELP3}, namely $-{\bm{A}}_1$,
has one negative, one positive, and six zero eigenvalues on the boundary $\varSigma_T$.
As discussed in Section \ref{sec:nonlinear}, 
the correct number of boundary conditions is two, which is just the case in \eqref{ELP3b}--\eqref{ELP3c}.
Therefore, for the hyperbolic problem \eqref{ELP3}, 
the boundary is {\it characteristic of constant multiplicity} and {\it the maximality condition} is fulfilled ({\it cf.}~\cite[Definition 2 and (11)]{R85MR0797053}).

Let us turn to derive the $L^2$ {\it a priori} estimate for solutions of \eqref{ELP3}. Take the scalar product of \eqref{ELP3a} with $W$ to get
\begin{align} \label{est:0a}
	\mathcal{E}_0(t)-\int_{\varSigma_t} \bm{A}_1 W \cdot W\d x'\d{\tau}
	\leq C(K)\left(\|\mathring{J}^{\mathsf{T}}\tilde{f}\|_{L^2(\varOmega_T)}^2+\int_{0}^{t}\mathcal{E}_0({\tau})\d{\tau}  \right),
\end{align}
where $\mathcal{E}_0(t):=\int_{\varOmega} {\bm{A}}_0W\cdot W\d x.$
By virtue of \eqref{decom} and \eqref{ELP3b}--\eqref{ELP3c}, we obtain
\begin{align*}
-\bm{A}_1 W \cdot W
	=\,&-2W_1W_2=\p_t\left(\p_1\mathring{q} \psi^2\right) +\sum_{i=2,3}\p_i\left(\mathring{v}_i \p_1\mathring{q} \psi^2 \right)\\
	&-\bigg(\p_t\p_1\mathring{q} +\sum_{i=2,3}\p_i\left(\mathring{v}_i \p_1\mathring{q} \right) +{2}\p_1\mathring{q} \p_1\mathring{v}_N  \bigg)\psi^2
	\quad \textrm{on }\varSigma_T,
\end{align*}
which together with \eqref{est:0a} yields
\begin{align} \nonumber
	&\mathcal{E}_0(t)+\int_{\varSigma} \p_1\mathring{q} \psi^2 \d x'\\
\nonumber	
&\qquad \leq C(K)\bigg\{\|\tilde{f}\|_{L^2(\varOmega_T)}^2
	+\int_{0}^{t}\left(\mathcal{E}_0({\tau})+\|\psi({\tau})\|_{L^2(\varSigma)}^2 \right)\d{\tau}  \bigg\}.
\end{align}
Since the sign condition \eqref{bas1c} and ${\bm{A}}_0\geq \kappa_1 I_{8}$ are satisfied
for some positive constant $\kappa_1$ (independent of the light speed in the relativistic case),
we apply Gr\"{o}nwall's inequality to deduce the $L^2$ estimate
\begin{align}
	\label{est:0b}
	\|W\|_{L^2(\varOmega_T)}+\|\psi\|_{L^2(\varSigma_T)} \leq C(K) \|\tilde{f}\|_{L^2(\varOmega_T)}.
\end{align}
The last estimate exhibits no loss of derivatives from the source term $\tilde{f}$ to the solution $W$, so one can apply the classical argument in \cite{LP60MR0118949,R85MR0797053} and \cite[Chapter 7]{CP82MR0678605} to construct solutions of the problem \eqref{ELP3}.
We only need to show that the dual problem of \eqref{ELP3} satisfies an {\it a priori} estimate without loss of derivatives similar to \eqref{est:0b}. Let us define the following dual problem for \eqref{ELP3}:
\begin{subequations}
	\label{dual}
	\begin{alignat}{3}
		\label{dual.a}
		&{\bm{L}}^*U^*=f^*
		&\quad &\textnormal{in }\varOmega_T,\\
		\label{dual.b}
		&\p_t U_1^*+\p_2(\mathring{v}_2 U_1^*)+\p_3(\mathring{v}_3 U_1^*)+\p_1\mathring{q} U_2^*+\p_1\mathring{v}_N U_1^*=0
		&\quad &\textnormal{on }\varSigma_T,\\
		&U^*=0
		&\quad &\textnormal{if }t<0,
	\end{alignat}
\end{subequations}
where ${\bm{L}}^*:=-{\bm{L}}+{\bm{A}}_4+{\bm{A}}_4^{\mathsf{T}}-\p_t{\bm{A}}_0-\p_i{\bm{A}}_i$ denotes the formal adjoint of operator ${\bm{L}}$.
Taking the scalar product of \eqref{dual.a} with $U^*$ implies
\begin{align}\nonumber  
	\mathcal{E}_0^*(t)-2\int_{\varSigma_t} U_1^* U_2^*\d x'\d{\tau}
	\leq C(K)\left(\|f ^*\|_{L^2(\varOmega_T)}^2+\int_{0}^{t}\mathcal{E}_0^*({\tau})\d{\tau}  \right),
\end{align}
where $\mathcal{E}_0^*(t):=\int_{\varOmega} {\bm{A}}_0U^*\cdot U^*\d x.$
Then we utilize \eqref{dual.b} to obtain
\begin{align*}
	&\mathcal{E}_0^*(t)+\int_{\varSigma} \frac{(U_1^*)^2}{\p_1\mathring{q}} \d x'\\
	&\qquad \leq C(K)\bigg\{\|{f}^*\|_{L^2(\varOmega_T)}^2
	+\int_{0}^{t}\left(\mathcal{E}_0^*({\tau})+\|U_1^*({\tau})\|_{L^2(\varSigma)}^2 \right)\d{\tau}  \bigg\},
\end{align*}
which combined with the condition \eqref{bas1c} and Gr\"{o}nwall's inequality 
leads to
\begin{align}
	\nonumber
	\|U^*\|_{L^2(\varOmega_T)}\leq C(K)\|f^*\|_{L^2(\varOmega_T)}.
\end{align}
With the aid of the last estimate and \eqref{est:0b}, one can deduce the following well-posedness result in $L^2$ for the reformulated problem \eqref{ELP3}. We omit  further details that are standard and can be found in \cite{LP60MR0118949,R85MR0797053} and \cite[Chapter 7]{CP82MR0678605}.

\begin{theorem}\label{thm:4}
	Assume that the basic state $(\mathring{U},\mathring{\varphi})$ satisfies \eqref{bas1a}--\eqref{bas1d} and the source terms $f\in L^2(\varOmega_T)$, $g\in H^1(\varSigma_T)$ vanish in the past.
	Then the problem \eqref{ELP3} has a unique solution $(W,\psi)\in L^2(\varOmega_T)\times L^2(\varSigma_T)$ satisfying \eqref{est:0b}.
\end{theorem}

\subsection{Higher-order Energy Estimates}

{We} now derive the higher-order energy estimates for solutions of the problem \eqref{ELP3}.
Let $\alpha=(\alpha_0,\alpha_1,\alpha_2,\alpha_3,\alpha_4)\in\mathbb{N}^5$ with $\langle \alpha \rangle:=\sum_{i=0}^{3}\alpha_i+2\alpha_4\leq m$. Applying the operator $\mathrm{D}_*^{\alpha}:=\p_t^{\alpha_0}(\sigma \p_1)^{\alpha_1}\p_2^{\alpha_2}\p_3^{\alpha_3}\p_1^{\alpha_4}$ to \eqref{ELP3a} yields
\begin{align} \label{iden1}
	{\bm{A}}_0\p_t\mathrm{D}_*^{\alpha}W+{\bm{A}}_i\p_i \mathrm{D}_*^{\alpha} W   =R_{\alpha},
\end{align}
where
\begin{align} \nonumber 
	R_{\alpha}:=\mathrm{D}_*^{\alpha}(\mathring{J}^{\mathsf{T}}\tilde{f})-\mathrm{D}_*^{\alpha}({\bm{A}}_4 W)
	-[\mathrm{D}_*^{\alpha},{\bm{A}}_i\p_i]W-[\mathrm{D}_*^{\alpha},{\bm{A}}_0\p_t] W.
\end{align}
Since it follows from from \eqref{decom} that $\bm{A}_1=\bm{A}_1^{(1)}$ on $\varSigma_T$, we take the scalar product of \eqref{iden1} with $\mathrm{D}_*^{\alpha} W$ to obtain
\begin{align}\label{est:alpha}
	\mathcal{E}_{\alpha}(t)=\mathcal{R}_{\alpha}(t)+\mathcal{Q}_{\alpha}(t),
\end{align}
where
\begin{align}
	\label{E.alpha}
	\mathcal{E}_{\alpha}(t):=\;&\int_{\varOmega}{\bm{A}}_0\mathrm{D}_*^{\alpha}W\cdot \mathrm{D}_*^{\alpha}W\d x\gtrsim \|\mathrm{D}_*^{\alpha}W(t)\|_{L^{2}(\varOmega)}^2 
	,\\
	\label{Q.alpha}
	\mathcal{Q}_{\alpha}(t):=\;&
	\int_{\varSigma_t}\bm{A}_1\mathrm{D}_*^{\alpha}W\cdot\mathrm{D}_*^{\alpha}W \d x'\d{\tau}
	= 2\int_{\varSigma_t}\mathrm{D}_*^{\alpha}W_1\mathrm{D}_*^{\alpha}W_2\d x'\d{\tau},\\
	\label{R.alpha}
	\mathcal{R}_{\alpha}(t):=\;&\int_{\varOmega_t}\mathrm{D}_*^{\alpha}W\cdot  \Big\{
	2R_{\alpha}+(\p_t{\bm{A}}_0 +\p_i {\bm{A}}_i)\mathrm{D}_*^{\alpha}W \Big\}\d x\d{\tau}.
\end{align}

{Hereafter we denote $\p_0:=\p_t$ for simplifying the presentation.}
First we make the estimate of $\mathcal{R}_{\alpha}(t)$ defined by \eqref{R.alpha}.
\begin{lemma}
	\label{lem:R.alpha}
If the assumptions in Theorem \ref{thm:3} are fulfilled, then
\begin{align}
\label{est:R.alpha}
\sum_{\langle\alpha\rangle \leq m}
\mathcal{R}_{\alpha}(t)\leq C(K)
\mathcal{M}(t),
\end{align}
where $\mathcal{M}(t)$ is defined by
\begin{align} 
\mathcal{M}(t):= {\|}(\tilde{f},W){\|}_{m,*,t}^2
+\mathring{\rm C}_{m+4}	 \|(\tilde{f},W)\|_{W_*^{2,\infty}(\varOmega_t)}^2,
\label{M.cal}
\end{align}
with $$\mathring{\rm C}_{s}:=1+{\|}(\mathring{V},\mathring{\varPsi}){\|}_{s,*,T}^2.$$
\end{lemma}
\begin{proof}
It follows from Cauchy's inequality that
\begin{align}
\label{R.e1}
\mathcal{R}_{\alpha}(t) \leq
{\|}(\mathring{J}^{\mathsf{T}}\tilde{f},{\bm{A}}_4 W){\|}_{m,*,t}^2+\sum_{i=0}^3C(K) \|(\mathrm{D}_*^{\alpha} W, [\mathrm{D}_*^{\alpha},{\bm{A}}_i\p_i]W)\|_{L^2(\varOmega_t)}^2.
\end{align}
Since ${\bm{A}}_4$ is a $C^{\infty}$--function of
$(\mathring{U},\p_t\mathring{U},\nabla\mathring{U},\nabla\mathring{\varPhi},
\p_t\nabla\mathring{\varPhi},\nabla^2\mathring{\varPhi})$
and $\mathring{J}$ is a $C^{\infty}$--function of $\nabla\mathring{\varPhi}$ ({\it cf.}~\eqref{J.ring}),
we apply the Moser-type calculus inequalities \eqref{Moser2}--\eqref{Moser3} to obtain
\begin{align}   \label{R.e1a}
{\|}(\mathring{J}^{\mathsf{T}}\tilde{f},{\bm{A}}_4 W){\|}_{m,*,t}^2
\leq C(K)
\left(
{\|}(\tilde{f},W){\|}_{m,*,t}^2+
\mathring{\rm C}_{m+4}\|(\tilde{f},W)\|_{W_*^{1,\infty}(\varOmega_t)}^2
\right).
\end{align}
Regarding the commutator terms in \eqref{R.e1}, for $i=0,\ldots,3$,  we have
\begin{align} \label{R.e1b}
\|[\mathrm{D}_*^{\alpha},{\bm{A}}_i]\p_iW\|_{L^2(\varOmega_t)}^2
\lesssim
\sum_{0<\alpha'\leq \alpha} \underbrace{\|\mathrm{D}_*^{\alpha'}{\bm{A}}_i \mathrm{D}_*^{\alpha-\alpha'}\p_iW\|_{L^2(\varOmega_t)}^2}_{\mathcal{J}_{\alpha'}^{(i)}}.
\end{align}
It follows from \eqref{decom} that $\p_i {\bm{A}}_1|_{x_1=0}=0$ for $i=0,2,3$.
Since $\langle \alpha'\rangle=1$ implies $\alpha_4'=0$, we get
\begin{align} \nonumber 
\|\mathrm{D}_*^{\alpha'} {\bm{A}}_1(x_1)\|_{L^{\infty}((-\infty,T)\times\mathbb{T}^{2})}
\leq C(K)\sigma(x_1)
\ \ \textrm{for all } x_1\geq 0,\ \langle \alpha'\rangle=1.
\end{align}
By virtue of the last inequality, we infer
\begin{align} 
\sum_{i=0}^3\mathcal{J}_{\alpha'}^{(i)}
\leq \, &
C(K)\sum_{i=0,2,3}\|\mathrm{D}_*^{\alpha-\alpha'}\p_iW\|_{L^2(\varOmega_t)}^2
+C(K) \|\sigma \mathrm{D}_*^{\alpha-\alpha'}\p_1W\|_{L^2(\varOmega_t)}^2  
\nonumber \\
\leq \, &
C(K) {\|}W{\|}_{m,*,t}^2
\qquad \textrm{for }  \langle \alpha'\rangle=1 \textrm{ and } \alpha'\leq \alpha.
\label{R.e1b1}
\end{align}
Since ${\|}\p_i W{\|}_{m-2,*,t}\lesssim {\|}W{\|}_{m,*,t}$
and ${\bm{A}}_i$ are $C^{\infty}$--functions of $(\mathring{U},\nabla\mathring{\varPhi})$
for $i=0,\ldots,3$,
we deduce from \eqref{Moser1} and \eqref{Moser3} that
\begin{align} 
\sum_{i=0}^3\mathcal{J}_{\alpha'}^{(i)}
\lesssim \;&
\sum_{i=0}^3
\sum_{\langle \beta\rangle=2,\,\beta\leq \alpha'}\|\mathrm{D}_*^{\alpha'-\beta }(\mathrm{D}_*^{\beta }{\bm{A}}_i) \mathrm{D}_*^{\alpha-\alpha'}(\p_iW)\|_{L^2(\varOmega_t)}^2 
\nonumber \\
\leq \, &
C(K)\left({\|}W{\|}_{m,*,t}^2+ \mathring{\rm C}_{m+4}\|W\|_{W_*^{2,\infty}(\varOmega_t)}^2 \right)
\quad \textrm{for } \langle \alpha'\rangle\geq 2.
\label{R.e1b2}
\end{align}

If  $\alpha_1=0$, then $[\mathrm{D}_*^{\alpha},{\bm{A}}_i\p_i]W=[\mathrm{D}_*^{\alpha},{\bm{A}}_i]\p_i W$ for $i=0,\ldots,3$. Hence we plug \eqref{R.e1b1}--\eqref{R.e1b2} into \eqref{R.e1b} and
then substitute the resulting estimate and \eqref{R.e1a} into \eqref{R.e1} to obtain
\eqref{est:R.alpha} for the case $\alpha_1=0$.

Next let $\alpha_1>0$.
It suffices to make the estimate of differences $[\mathrm{D}_*^{\alpha},{\bm{A}}_i\p_i]W-[\mathrm{D}_*^{\alpha},{\bm{A}}_i]\p_i W$ for $i=0,\ldots,3$.
In the following computations, we can  replace the operator $\mathrm{D}_*^{\alpha}$ by $\sigma^{\alpha_1}\p_1^{\alpha_1+\alpha_4}\p_t^{\alpha_0}\p_2^{\alpha_2}\p_3^{\alpha_3}$,
because the corresponding norms are equivalent; see, {\it e.g.}, \cite[p.\;394]{OSY94MR1289186}.
In view of the decomposition \eqref{decom}, we compute
\begin{align}
\nonumber 
& \|[\mathrm{D}_*^{\alpha},{\bm{A}}_i\p_i]W-[\mathrm{D}_*^{\alpha},{\bm{A}}_i]\p_i W\|_{L^2(\varOmega_t)}^2\\
&  
\lesssim
\| {\bm{A}}_1 \sigma^{\alpha_1-1}\p_t^{\alpha_0}\p_1^{\alpha_1+\alpha_4}\p_2^{\alpha_2}\p_3^{\alpha_3} W\|_{L^2(\varOmega_t)}^2 \nonumber \\
&  
\lesssim
\|  \sigma^{\alpha_1-1}\p_t^{\alpha_0}\p_1^{\alpha_1+\alpha_4}\p_2^{\alpha_2}\p_3^{\alpha_3}  W_{\rm nc}\|_{L^2(\varOmega_t)}^2+\|  \sigma^{\alpha_1}\p_t^{\alpha_0}\p_1^{\alpha_1+\alpha_4}\p_2^{\alpha_2}\p_3^{\alpha_3} W\|_{L^2(\varOmega_t)}^2\nonumber  \\
&  
\lesssim
{\|}\p_1 W_{\rm nc}{\|}_{m-1,*,t}^2+{\|}W{\|}_{m,*,t}^2
\qquad \textrm{for }i=0,\ldots 3,
\label{R.e3a}
\end{align}
where $W_{\rm nc}:=(W_1,W_2)^{\mathsf{T}}$.
For estimating the right-hand side of \eqref{R.e3a},
we utilize the equations \eqref{ELP3a} and the decomposition \eqref{decom} to get
\begin{align} \label{iden2}
(\p_1 W_2,\p_1 W_1,0)^{\mathsf{T}}
=\mathring{J}^{\mathsf{T}}\tilde{f}-{\bm{A}}_4 W-\sum_{i=0,2,3}{\bm{A}}_i\p_i W-{\bm{A}}_1^{(0)}\p_1W,
\end{align}
which implies
\begin{align}
{\|}\p_1 W_{\rm nc}{\|}_{m-1,*,t} \leq
\sum_{i=0,2,3}{\|}(\mathring{J}^{\mathsf{T}}\tilde{f}, {\bm{A}}_4 W, {\bm{A}}_i\p_i W, {\bm{A}}_1^{(0)}\p_1W){\|}_{m-1,*,t}.
\label{R.e3b}
\end{align}
By definition, it follows that
\begin{align}
{\|}{\bm{A}}_1^{(0)}\p_1W{\|}_{m-1,*,t}^2
\lesssim
\sum_{\substack{\beta'\leq \beta,\,\langle\beta\rangle\leq m-1}}\|\mathrm{D}_*^{\beta'}{\bm{A}}_1^{(0)}\mathrm{D}_*^{\beta-\beta'}\p_1W\|_{L^2(\varOmega_t)}^2.
\label{R.e3c}
\end{align}
{To estimate the right-hand side of \eqref{R.e3c}, we first use} \eqref{decom} to get
\begin{align}
\|{\bm{A}}_1^{(0)}\mathrm{D}_*^{\beta }\p_1W\|_{L^2(\varOmega_t)}^2
\lesssim
\|\sigma \mathrm{D}_*^{\beta}\p_1W\|_{L^2(\varOmega_t)}^2\lesssim {\|}W{\|}_{m,*,t}^2.
\label{R.e3d}
\end{align}
Applying the Moser-type calculus inequalities \eqref{Moser1}--\eqref{Moser3} leads to
\begin{align} \nonumber
&\sum_{0<\beta'\leq \beta}
\|\mathrm{D}_*^{\beta'}{\bm{A}}_1^{(0)}\mathrm{D}_*^{\beta-\beta'}\p_1W\|_{L^2(\varOmega_t)}^2
\\
& \qquad \nonumber\lesssim
\sum_{0<\beta'\leq \beta}\sum_{\langle\gamma\rangle=1,\,\gamma\leq \beta'}\|\mathrm{D}_*^{\beta'-\gamma}(\mathrm{D}_*^{\gamma}{\bm{A}}_1^{(0)})\mathrm{D}_*^{\beta-\beta'}\p_1W\|_{L^2(\varOmega_t)}^2\\
&\qquad \leq C(K){\left({\|}W{\|}_{m,*,t}^2+ \mathring{\rm C}_{m+4}\|W\|_{W_*^{2,\infty}(\varOmega_t)}^2 \right)},
\label{R.e3e}
\end{align}
owing to $\langle \beta'-\gamma\rangle+\langle \beta-\beta'\rangle+2\leq m$.
Employing \eqref{Moser1}--\eqref{Moser3}
to estimate ${\|}{\bm{A}}_i\p_i W{\|}_{m-1,*,t}$ in \eqref{R.e3b}, for $i=0,2,3$,
and combining  \eqref{R.e1}--\eqref{R.e3e}, we  derive
\eqref{est:R.alpha} for the case $\alpha_1>0$ and so finish up the proof.
\end{proof}
The following lemma provides the estimate of the non-weighted tangential derivatives $\mathrm{D}_*^{\alpha}W$ for $\alpha_1=\alpha_4=0$ and $\langle \alpha \rangle\leq m$.
\begin{lemma} \label{lem:tan}
	If the assumptions in Theorem \ref{thm:3} are fulfilled, then
	\begin{align}
		\sum_{\langle\alpha\rangle \leq m,\, \alpha_1=\alpha_4=0}
		\left(\|\mathrm{D}_*^{\alpha} W(t)\|_{L^2(\varOmega)}^2+\|\mathrm{D}_*^{\alpha} \psi(t)\|_{L^2(\varSigma)}^2\right)
		\label{est:tan}
		\leq  C(K) \widetilde{\mathcal{M}}(t)
	\end{align}
	for all $t\in [0,T]$ and $m\in\mathbb{N}_+$, where
	\begin{align}\nonumber
		\widetilde{\mathcal{M}}(t):=\;&{\|}(\tilde{f},W){\|}_{m,*,t}^2
		+{\mathring{\rm C}_{m+4}
		\|(\tilde{f},W)\|_{W_*^{2,\infty}(\varOmega_t)}^2}
		\\
		\label{M.tilde.cal}&
		+\|\psi\|_{H^m(\varSigma_t)}^2+\mathring{\rm C}_{{m+4}}\|\psi\|_{L^{\infty}(\varSigma_t)}^2.
	\end{align}
\end{lemma}
\begin{proof}
	Let $\alpha_1=\alpha_4=0$ and $\langle \alpha \rangle\leq m$.
	It follows from definition that $ \mathring{\rm C}_s\geq \|{\mathring{\varphi}}\|_{H^s(\varSigma_T)}^2$.
	Let us estimate the boundary term $\mathcal{Q}_{\alpha}(t)$ defined by \eqref{Q.alpha}.
	In view of \eqref{ELP3b}--\eqref{ELP3c}, we find
	\begin{align} \nonumber
		\mathcal{Q}_{\alpha}(t)=\mathcal{Q}_{\alpha}^{(1)}(t)+\mathcal{Q}_{\alpha}^{(2)}(t),
	\end{align}
	where
	\begin{align*}
		\mathcal{Q}_{\alpha}^{(1)}(t):=\;&-2\int_{\varSigma_t} \mathrm{D}_*^{\alpha}(\p_1 \mathring{q} \psi) (\p_t+\mathring{v}_2\p_2+\mathring{v}_3\p_3)\mathrm{D}_*^{\alpha} \psi \\
		=\;&-2\int_{\varSigma_t} \Big(\p_1 \mathring{q}\mathrm{D}_*^{\alpha} \psi+
		[\mathrm{D}_*^{\alpha},\p_1 \mathring{q}]\psi \Big) (\p_t+\mathring{v}_2\p_2+\mathring{v}_3\p_3)\mathrm{D}_*^{\alpha} \psi, \\
		\mathcal{Q}_{\alpha}^{(2)}(t):=
		\;&-2\int_{\varSigma_t} \mathrm{D}_*^{\alpha}(\p_1 \mathring{q} \psi)
		\Big([\mathrm{D}_*^{\alpha}, \mathring{v}_2\p_2+\mathring{v}_3\p_3]\psi -\mathrm{D}_*^{\alpha}(\p_1 \mathring{v}_N \psi) \Big).
	\end{align*}
	Since $\mathrm{D}_*^{\alpha}=\p_t^{\alpha_0}\p_2^{\alpha_2}\p_3^{\alpha_3}$
	{for $\alpha_1=\alpha_4=0$
	and  $\|\p_1 \mathring{q}\|_{H^{m+1}(\varSigma_t)}\lesssim
	\|\p_1 \mathring{q}\|_{m+2,*,t}\lesssim {\|}\mathring{V}{\|}_{{m+4},*,T}$},
	we use the sign condition \eqref{bas1c},
	integration by parts, Cauchy's inequality,
	and Lemma~\ref{lem:Moser1} to get
	\begin{align*}
		\mathcal{Q}_{\alpha}^{(1)}(t)\leq \;& -\int_{\varSigma} \p_1 \mathring{q} (\mathrm{D}_*^{\alpha} \psi)^2
		-2\int_{\varSigma}[\mathrm{D}_*^{\alpha},\p_1 \mathring{q}]\psi \mathrm{D}_*^{\alpha} \psi
		+C(K) \|\mathrm{D}_*^{\alpha} \psi\|_{L^2(\varSigma_t)}^2 \\[1mm]
		\;& +\| \p_t [\mathrm{D}_*^{\alpha},\p_1 \mathring{q}] \psi \|_{L^2(\varSigma_t)}^2
		+\sum_{i=2,3}\| \p_i (\mathring{v}_i [\mathrm{D}_*^{\alpha},\p_1 \mathring{q}] \psi )\|_{L^2(\varSigma_t)}^2\\
		\leq \;&
		-\frac{\kappa_{0}}{4}\|\mathrm{D}_*^{\alpha} \psi\|_{L^2(\varSigma)}^2
		+C(K) \left(\|\mathrm{D}_*^{\alpha} \psi\|_{L^2(\varSigma_t)}^2
		+ \| [\mathrm{D}_*^{\alpha},\p_1 \mathring{q}] \psi \|_{H^1(\varSigma_t)}^2  \right)\\[1mm]
		\leq \;&
		-\frac{\kappa_{0}}{4}\|\mathrm{D}_*^{\alpha} \psi\|_{L^2(\varSigma)}^2
		+C(K) \left(\|\psi\|_{H^m(\varSigma_t)}^2
		+\mathring{\mathrm{C}}_{{m+4}} \|\psi\|_{L^{\infty}(\varSigma_t)}^2\right).
	\end{align*}
	Apply Lemma~\ref{lem:Moser1} to derive
	\begin{align*}
		\mathcal{Q}_{\alpha}^{(2)}(t)\leq
		C(K) \left(\|\psi\|_{H^m(\varSigma_t)}^2
		+ \mathring{\mathrm{C}}_{{m+4}} \|\psi\|_{L^{\infty}(\varSigma_t)}^2 \right),
	\end{align*}
	which together with the estimate for $\mathcal{Q}_{\alpha}^{(1)}(t)$ implies
	\begin{align} \label{Q.e1}
		\mathcal{Q}_{\alpha}(t) 
		+\frac{\kappa_{0}}{4}\|\mathrm{D}_*^{\alpha} \psi(t)\|_{L^2(\varSigma)}^2
		\leq  C(K) \left(\|\psi\|_{H^m(\varSigma_t)}^2
		+\mathring{\mathrm{C}}_{{m+4}} \|\psi\|_{L^{\infty}(\varSigma_t)}^2\right).
	\end{align}
	
Plugging \eqref{E.alpha}, \eqref{est:R.alpha}, and \eqref{Q.e1}
into \eqref{est:alpha} yields the desired estimate
\eqref{est:tan} and completes the proof of the lemma.
\end{proof}

In the next lemma, we have the estimate of the non-weighted normal derivatives $\mathrm{D}_*^{\alpha}W$ for $\alpha_1=0$, $\alpha_4>0$, and $\langle \alpha \rangle\leq m$.
\begin{lemma} \label{lem:non}
	If the assumptions in Theorem \ref{thm:3} are fulfilled, then
	\begin{align}
		&\sum_{\substack{\langle\alpha\rangle \leq m,\, \alpha_1=0,\,\alpha_4>0}} \|\mathrm{D}_*^{\alpha} W(t)\|_{L^2(\varOmega)}^2\nonumber \\
		\label{est:non}
		&\qquad \qquad \lesssim  \varepsilon \sum_{\langle \beta\rangle\leq m} \| \mathrm{D}_*^{\beta} W(t)\|_{L^2(\varOmega)}^2
		+ C(\varepsilon,K) \mathcal{M}(t)
	\end{align}
    for all $t\in [0,T]$, $m\in \mathbb{N}_+$, and $\varepsilon>0$,
	where $\mathcal{M}(t)$ is defined by \eqref{M.cal}.
\end{lemma}
\begin{proof}
	Let $\alpha_1=0$, $\alpha_4>0$, and $\langle \alpha \rangle\leq m$.
	Then $|\alpha|=\langle \alpha \rangle-\alpha_4<m$.
	Thanks to \eqref{E.alpha} and \eqref{est:R.alpha},
	it remains to control the term $\mathcal{Q}_{\alpha}(t)$ ({\it cf.}~definition \eqref{Q.alpha}).
We obtain from the identity \eqref{iden2} that
	\begin{align}
		\mathcal{Q}_{\alpha}(t)\lesssim
		\sum_{i=0,2,3}\|\mathrm{D}_*^{\alpha-{\bm{e}}} (\mathring{J}^{\mathsf{T}}\tilde{f},  {\bm{A}}_4 W,{\bm{A}}_i\p_i W,{\bm{A}}_1^{(0)}\p_1W)\|_{L^2(\varSigma_t)}^2,
		\label{Q.e2}
	\end{align}
	where ${\bm{e}}:=(0,0,0,0,1)$ and $\p_0:=\p_t$.
	Use Lemma~\ref{lem:trace} and \eqref{R.e1a} to infer
	\begin{align}   \label{Q.e2a}
		\|\mathrm{D}_*^{\alpha-{\bm{e}}} (\mathring{J}^{\mathsf{T}}\tilde{f},  {\bm{A}}_4 W)\|_{L^2(\varSigma_t)}^2
		\leq C(K) \mathcal{M}(t).
	\end{align}
	For $i=0,2,3$, we have
	\begin{align} \label{Q.e2b}
		\|\mathrm{D}_*^{\alpha-{\bm{e}}} ({\bm{A}}_i\p_i W)\|_{L^2(\varSigma_t)}^2
		\lesssim  C(K)\|\mathrm{D}_*^{\alpha-{\bm{e}}} \p_i W\|_{L^2(\varSigma_t)}^2
		+\sum_{0<\alpha'\leq \alpha-{\bm{e}}} \mathcal{K}_{\alpha'}^{(i)},
	\end{align}
	where
	$$
	\mathcal{K}_{\alpha'}^{(i)}:=\|\mathrm{D}_*^{\alpha'}{\bm{A}}_i \mathrm{D}_*^{\alpha-{\bm{e}}-\alpha'}\p_iW\|_{L^2(\varSigma_t)}^2
	\quad \textrm{for }\alpha'\leq \alpha-{\bm{e}}.
	$$
	To estimate the first term on the right-hand side of \eqref{Q.e2b}, one can employ the argument of passing to the volume integral ({\it cf.}~\cite[pp.\;273--274]{T09ARMAMR2481071}). More precisely, we use integration by parts to obtain
	\begin{align}
		\sum_{i=0,2,3}\|\mathrm{D}_*^{\alpha-{\bm{e}}} \p_i W\|_{L^2(\varSigma_t)}^2
		&=-2\sum_{i=0,2,3}\int_{\varOmega_t} \mathrm{D}_*^{\alpha}\p_iW\cdot \mathrm{D}_*^{\alpha-\bm{e}}\p_iW
		\nonumber\\
		& 
		=\mathcal{K}_1 +\mathcal{K}_2,
 \label{Q.e2b1}
	\end{align}
	where
	\begin{align} \nonumber
		&\mathcal{K}_1:=-2\int_{\varOmega}\mathrm{D}_*^{\alpha } W \cdot \mathrm{D}_*^{\alpha-{\bm{e}}}\p_t W,\ \ 
		\mathcal{K}_2:=2\sum_{i=0,2,3}\int_{\varOmega_t}\mathrm{D}_*^{\alpha} W \cdot \mathrm{D}_*^{\alpha-{\bm{e}}}\p_i^2 W.
		\nonumber
	\end{align}
	Using integration by parts yields
	\begin{align}
		\mathcal{K}_1
		\lesssim\;& \varepsilon \|\mathrm{D}_*^{\alpha}   W(t)\|_{L^2(\varOmega)}^2
		+C(\varepsilon) \|\mathrm{D}_*^{\alpha-{\bm{e}}}\p_t  W\|_{L^2(\varOmega)}^2 \nonumber  \\[1mm]
		\lesssim\;&\varepsilon \|\mathrm{D}_*^{\alpha}   W(t)\|_{L^2(\varOmega)}^2
		+C(\varepsilon) \|\mathrm{D}_*^{\alpha-{\bm{e}}}\p_t  W\|_{1,*,t}^2 \nonumber  \\[1mm]
		\lesssim\;& \varepsilon \sum_{\langle \beta\rangle\leq m} \|\mathrm{D}_*^{\beta}  W(t)\|_{L^2(\varOmega)}^2
		+C(\varepsilon) {\|}W{\|}_{m,*,t}^2.
		\label{K1.e}
	\end{align}
Since $\langle\alpha-\bm{e}\rangle+2\leq m$ and $i\in\{0,2,3\}$, we use Cauchy's inequality to get
	\begin{align}
		\mathcal{K}_2
		\lesssim \|(\mathrm{D}_*^{\alpha} W , \mathrm{D}_*^{\alpha-{\bm{e}}}\p_i^2 W)\|_{L^2(\varOmega_t)}^2
		\lesssim {\|}W{\|}_{m,*,t}^2.
		\label{K2.e}
	\end{align}
	For $i=0,2,3$, we apply Lemma~\ref{lem:trace} and \eqref{Moser1}--\eqref{Moser3} to infer
	\begin{align}
		\nonumber
		\sum_{0<\alpha'\leq \alpha-{\bm{e}}} \mathcal{K}_{\alpha'}^{(i)}
		\lesssim  \;&
		\sum_{\langle\beta\rangle=1,\,
			\beta\leq \alpha'\leq \alpha-{\bm{e}}}  {\|}\mathrm{D}_*^{\alpha'-\beta}(\mathrm{D}_*^{\beta}{\bm{A}}_i)
		\mathrm{D}_*^{\alpha-{\bm{e}}-\alpha'}\p_i W {\|}_{2,*,t}^2 \\
		\leq  \;&
		C(K)\left({\|}W{\|}_{m,*,t}^2+\mathring{\rm C}_{m+4} \|W\|_{W_*^{2,\infty}(\varOmega_t)}^2 \right)
		\label{Q.e2b2}
	\end{align}
	due to $\langle \alpha'-\beta\rangle +\langle \alpha-{\bm{e}}-\alpha'\rangle+1\leq m-2$.
	Since $(\mathrm{D}_*^{\beta}{\bm{A}}_1^{(0)})|_{\varSigma_T}=0$ for all $\beta=(\beta_0,\beta_1,\beta_2,\beta_3,\beta_4)$ with $\beta_4=0$, we get
	\begin{align}
\nonumber		&\|\mathrm{D}_*^{\alpha-{\bm{e}}} ({\bm{A}}_1^{(0)}\p_1W)\|_{L^2(\varSigma_t)}^2
		\lesssim
		\sum_{{\bm{e}}\leq \beta\leq \alpha-{\bm{e}} } \|\mathrm{D}_*^{\beta} {\bm{A}}_1^{(0)}\mathrm{D}_*^{\alpha-{\bm{e}}-\beta} \p_1W\|_{L^2(\varSigma_t)}^2   \\
		&\qquad \lesssim
		\sum_{{\bm{e}}\leq \beta\leq \alpha-{\bm{e}} } {\|}\mathrm{D}_*^{\beta-{\bm{e}}}(\p_1 {\bm{A}}_1^{(0)})\mathrm{D}_*^{\alpha-\beta} W{\|}_{2,*,t}^2\leq C(K)\mathcal{M}(t).
		\label{Q.e2c}
	\end{align}
	Plug \eqref{Q.e2a}--\eqref{Q.e2c} into \eqref{Q.e2} and
	combine the resulting estimate with \eqref{est:alpha}--\eqref{E.alpha} and \eqref{est:R.alpha}
	to derive \eqref{est:non} and finish the proof of the lemma.
\end{proof}

The estimate of the weighted derivatives $\mathrm{D}_*^{\alpha}W$ for $\alpha_1>0$ and $\langle \alpha \rangle\leq m$ is provided in the following lemma.

\begin{lemma}
	\label{lem:weighted}
	If the assumptions in Theorem \ref{thm:3} are fulfilled, then
	\begin{align}
		\sum_{\langle\alpha\rangle \leq m,\, \alpha_1>0} \|\mathrm{D}_*^{\alpha} W(t)\|_{L^2(\varOmega)}^2
		\label{est:weighed}
		\leq C(K) \mathcal{M}(t)
	\end{align}
	holds for all $t\in [0,T]$ and $m\in \mathbb{N}_+$,
	where $\mathcal{M}(t)$ is defined by \eqref{M.cal}.
\end{lemma}
\begin{proof}
	Let $\alpha_1>0$ and $\langle \alpha \rangle\leq m$.
	Then it is easy to check that $\mathcal{Q}_{\alpha}(t)=0$. So we combine Lemma~\ref{lem:R.alpha} and \eqref{est:alpha}--\eqref{E.alpha}
	to obtain the estimate \eqref{est:weighed}.
\end{proof}

\subsection{Proof of Theorem \ref{thm:3}}

We first show the tame {\it a priori} estimates for the problem \eqref{ELP3} in $H_*^m(\varOmega_T)$.
By definition, for $s\in\mathbb{N}$, we get
\begin{align}   \label{est:psi}
	\|{\mathring{\varphi}}\|_{H^s(\varSigma_t)}  \leq  {\|}\mathring{\varPsi}{\|}_{s,*,t}\lesssim \|{\mathring{\varphi}}\|_{H^s(\varSigma_t)} ,
	\
	\|\psi\|_{H^s(\varSigma_t)}\leq  {\|}\varPsi{\|}_{s,*,t} \lesssim \|\psi\|_{H^s(\varSigma_t)}.
\end{align}
Combine \eqref{est:tan}, \eqref{est:non}, and \eqref{est:weighed},
and take $\varepsilon>0$ small enough to derive
\begin{align}  \nonumber
	\mathcal{I}(t):= 
	\;&\sum_{\langle\alpha\rangle \leq m}
	\|\mathrm{D}_*^{\alpha} W(t)\|_{L^2(\varOmega)}^2+
	\sum_{\langle\alpha\rangle \leq m,{\alpha_1}=\alpha_4=0} \|\mathrm{D}_*^{\alpha} \psi(t)\|_{L^2(\varSigma)}^2 \\
	\leq  	\;&  C(K) \widetilde{\mathcal{M}}(t),
	\nonumber
\end{align}
where $\widetilde{\mathcal{M}}(t)$ is defined by \eqref{M.tilde.cal}.
Noting from \eqref{D*}--\eqref{norm1b} and \eqref{ELP3d} that
\begin{align} \label{I.cal1}
	\int_{0}^t\mathcal{I}({\tau})\d{\tau}={\|}W{\|}_{m,*,t}^2+\|\psi\|_{H^m(\varSigma_t)}^2,
\end{align}
we apply Gr\"{o}nwall's inequality to obtain
\begin{align}
	\label{I.cal2}
	\mathcal{I}(t)\leq C(K) \mathrm{e}^{C(K)T} \mathcal{N}(T)
	\quad \textrm{for } 0\leq t \leq T,
\end{align}
where
\begin{align} \nonumber
	\mathcal{N}(t):={\|}\tilde{f}{\|}_{m,*,t}^2+\mathring{\rm C}_{{m+4}} \left(\|(\tilde{f},W)\|_{W_*^{2,\infty}(\varOmega_t)}^2 +\|\psi\|_{L^{\infty}(\varSigma_t)}^2\right).
\end{align}
By virtue of the embedding inequalities \eqref{embed1}--\eqref{embed2}, we have
\begin{align}\label{M.cal1}
	\mathcal{N}(T)\lesssim {\|}\tilde{f}{\|}_{m,*,T}^2+\mathring{\rm C}_{{m+4}} \left({\|}(\tilde{f},W){\|}_{6,*,T}^2 +\|\psi\|_{H^{2}(\varSigma_T)}^2\right).
\end{align}
Integrate \eqref{I.cal2} over $[0,T]$, use \eqref{I.cal1} and \eqref{M.cal1}, and take $T>0$ sufficiently small to discover
\begin{align}
 \nonumber
 {\|}W{\|}_{m,*,T}^2+\|\psi\|_{H^m(\varSigma_T)}^2&\\[1mm]
\nonumber \leq   C(K) T\mathrm{e}^{C(K)T}  \Big\{ &
	{\|}\tilde{f}{\|}_{m,*,T}^2  	+ {\|}(\mathring{V},\mathring{\varPsi}){\|}_{{m+4},*,T}^2
	 \\[1mm]
& \times \big({{\|}(\tilde{f},W){\|}_{6,*,T}^2} +\|\psi\|_{H^{2}(\varSigma_T)}^2\big)
	\Big\}
\ \ \textrm{for }  m\geq 6.   \label{est:t1}
\end{align}
In view of \eqref{est:t1} with $m=6$, the assumption \eqref{thm3.H1}, the estimate \eqref{est:psi}, and the embedding $ H_*^{9}(\varOmega_T)\hookrightarrow W^{3,\infty}(\varOmega_T)$,
we can find a  sufficiently small constant $T_0>0$, depending on $K_0$, such that if $0<T\leq T_0$, then
\begin{align}\label{est:t2}
	{\|}W{\|}_{6,*,T}^2+\|\psi\|_{H^6(\varSigma_T)}^2
	\leq C(K_0)
	{\|}\tilde{f}{\|}_{6,*,T}^2 .
\end{align}
Plugging \eqref{est:t2} into \eqref{est:t1} yields
\begin{align}\nonumber
&	{\|}W{\|}_{m,*,T}^2+\|\psi\|_{H^m(\varSigma_T)}^2 \\
&\quad 	\leq C(K_0)
 \label{est:t3}	\left({\|}\tilde{f}{\|}_{m,*,T}^2
	+  {\|}(\mathring{V},\mathring{\varPsi}){\|}_{{m+4},*,T}^2{\|}\tilde{f}{\|}_{6,*,T}^2\right)
\ \ \textrm{for }m\geq 6.
\end{align}

According to Theorem \ref{thm:4}, for $(f,g)\in L^2(\varOmega_T)\times H^1(\varSigma_T)$ vanishing in the past, the problem \eqref{ELP3} has a unique solution $(W,\psi)\in L^2(\varOmega_T)\times L^2(\varSigma_T)$.
Applying the arguments in \cite[Chapter 7]{CP82MR0678605} and using the tame estimate \eqref{est:t3},
one can establish the existence and uniqueness of solutions of the problem \eqref{ELP3} in $H_*^m(\varOmega_T)\times H^{m}(\varSigma_T)$ for $m\geq 6$.

It suffices to show the tame estimate \eqref{tame} for the problem \eqref{ELP1}. To this end, we utilize \eqref{V.natural} and Lemma~\ref{lem:Moser2} to get the estimates
\begin{align}
	\nonumber
	{\|}\dot{V}{\|}_{m,*,T}^2 \leq\;& C(K_0)\left\{
	{\|}W{\|}_{m,*,T}^2+\|W\|_{W_*^{1,\infty}(\varOmega_T) }^2\mathring{\rm C}_{m+2}
	+\|g\|_{H^{m}(\varSigma_T)}^2
	\right\},\\
	\nonumber
	{\|}\tilde{f}{\|}_{m,*,T}^2\leq\;&  C(K_0)\left\{
	{\|} {f}{\|}_{m,*,T}^2+
	\|V_{\natural}\|_{W_*^{2,\infty}(\varOmega_T) }^2\mathring{\rm C}_{m+2}
	+\|g\|_{H^{m+1}(\varSigma_T)}^2
	\right\},
\end{align}
which together with \eqref{embed1}--\eqref{embed2},
\eqref{est:t3}, and \eqref{est:psi} imply the desired tame estimate \eqref{tame}.
This completes the proof of Theorem \ref{thm:3}.

\section{Nash--Moser Iteration}\label{sec:Nash}

This section is devoted to solving the nonlinear problem \eqref{NP1} by an appropriate Nash--Moser iteration scheme. See {\sc Alinhac--G{{\'e}}rard} \cite[Chapter 3]{AG07MR2304160} and {\sc Secchi} \cite{S16MR3524197} for a general description of the method.

\subsection{Approximate Solution}

To apply Theorem \ref{thm:3}, that is a well-posedness result in the space of functions vanishing in the past, we should reduce the nonlinear problem \eqref{NP1} to the case with zero initial data. For this purpose, we introduce the so-called approximate solutions to absorb the initial data into the interior equations.


Let $m\geq 3$ be an integer. Assume that the initial data $U_0$ and $\varphi_0$
satisfy
$\widetilde{U}_0:=U_0-\widebar{U} \in H^{m+3/2}(\varOmega)$ and
$\varphi_0\in H^{m+2}(\mathbb{T}^{d-1})$.
Define
\begin{align}
\label{varPsi_0}
{\varPsi}_0:=\chi(x_1) \varphi_0,\quad
\varPhi_0:= x_1+{\varPsi}_0,
\end{align}
where $\chi\in C_0^{\infty}(\mathbb{R})$ satisfies \eqref{chi}.
Without loss of generality, we assume that $\|\varphi_0\|_{L^{\infty}(\mathbb{T}^{d-1})}\leq 1/4$.
Then we have
\begin{align}\label{CA1}
 \p_1 \varPhi_0 \geq \frac{3}{4}\quad \textrm{in } \varOmega.
\end{align}

Let us denote $\widetilde{U} :=U -\widebar{U} $ and define
\begin{align*}
\widetilde{U}_{(j)}:= \p_t^{j}\widetilde{U}\big|_{t=0},\quad
\varphi_{(j)}:= \p_t^{j}\varphi\big|_{t=0},\quad
{\varPsi}_{(j)}:=\chi(x_1) \varphi_{(j)}
\quad\textrm{for  } j\in\mathbb{N}.
\end{align*}
Then we get
$\widetilde{U}_{(0)}=\widetilde{U}_0,$ $\varphi_{(0)}=\varphi_0,$ and ${\varPsi}_{(0)}={\varPsi}_{0}.$
To introduce the compatibility conditions,
we shall determine the traces  $\widetilde{U}_{(j)}$ and $\varphi_{(j)}$
in terms of the initial perturbations $\widetilde{U}_0$ and $\varphi_0$
through the equations \eqref{NP1a} and the first boundary condition in \eqref{NP1b}.
More precisely, applying the operator $\p_t^{j}$ to the first equation in \eqref{NP1b},
taking the traces at the initial time, and employing the Leibniz's rule yield
\begin{align}
\varphi_{(j+1)} = v_{1(j)}\big|_{x_1=0}-\sum_{k=0}^{j} \sum_{i=2}^d
\begin{pmatrix}
j \\ k
\end{pmatrix}  \p_i\varphi_{(j-k)}v_{i(k)}\big|_{x_1=0},
\label{trace.id1}
\end{align}
where $\left(\begin{smallmatrix}
j \\[0.3mm] k
\end{smallmatrix}\right)$ is the binomial coefficient.
Setting
$\mathcal{W}:=(\widetilde{U},\nabla \widetilde{U},\nabla{\varPsi})\in\mathbb{R}^{2d^2+5d+2}$
and assuming that the hyperbolicity condition \eqref{hyperbolicity} is satisfied,
we can rewrite the equations \eqref{NP1a} as
\begin{align}\label{tilde.U.Phi}
\p_t \widetilde{U}=\bm{G}\big(\mathcal{W}\big),
\end{align}
where $\bm{G}$ is a $C^{\infty}$--function that vanishes at the origin.
We employ the generalized Fa\`a di Bruno's formula (see \cite[Theorem 2.1]{M00MR1781515}) to obtain
\begin{align} \label{trace.id2}
&\widetilde{U}_{(j+1)}
=\sum_{\substack{\alpha_{k}\in\mathbb{N}^{2d^2+5d+2}\\ |\alpha_1|+\cdots+j |\alpha_{j}|=j}}
\mathrm{D}^{\alpha_1+\cdots+\alpha_j}\bm{G}\big(\mathcal{W}_{(0)}\big)\prod_{k=1}^j\frac{j!}{\alpha_{k}!}
\left(\frac{\mathcal{W}_{(k)}}{k!}\right)^{\alpha_{k}},
\end{align}
where $\mathcal{W}_{(k)}:=(\widetilde{U}_{(k)},\nabla \widetilde{U}_{(k)},\nabla {\varPsi}_{(k)})$.
From \eqref{trace.id1} and \eqref{trace.id2},
we can determine $\widetilde{U}_{(j)}$ and ${\varphi}_{(j)}$ for integers $j$ inductively
as functions of $\widetilde{U}_0$, $\varphi_0$, and their space derivatives up to order $j$.
Moreover, we have the following lemma (see \cite[Lemma 4.2.1]{M01MR1842775} for the proof).

\begin{lemma}\label{lem:CA1}
Let $m\in\mathbb{N}$ with $m\geq 3$ and $(\widetilde{U}_0,\varphi_0) \in H^{m+3/2}(\varOmega) \times H^{m+2}(\mathbb{T}^{d-1})$.
Then the equations \eqref{trace.id1} and \eqref{trace.id2} determine functions
$(\widetilde{U}_{(j)},\varphi_{(j)})$, for $j=1,\ldots,m$, belonging to $ H^{m+3/2-j}(\varOmega)\times H^{m+2-j}(\mathbb{T}^{d-1})$ and satisfying
\begin{align}
 & \sum_{j=0}^{m}\left(\big\|\widetilde{U}_{(j)}\big\|_{H^{m+3/2-j}(\varOmega)} +\big\|{\varphi}_{(j)}\big\|_{H^{m+2-j}(\mathbb{T}^{d-1})}\right)
\label{CA.est1}
\leq C M_0,
\end{align}
where
\begin{align}
\label{M0}
 M_0:=\big\|\widetilde{U}_0\big\|_{H^{m+3/2}(\varOmega)}
+\|\varphi_0\|_{H^{m+2}(\mathbb{T}^{d-1})},
\end{align}
and the constant $C>0$ depends only on $m$,
$\|\widetilde{U}_{0}\|_{W^{1,\infty}(\varOmega)}$,
and $\|{\varphi}_{0}\|_{W^{1,\infty}(\mathbb{T}^{d-1})}$.
\end{lemma}

To construct a smooth approximate solution,
it is necessary to impose certain compatibility conditions on the initial data.

\begin{definition}
\label{def:1}
Let $m\in\mathbb{N}$ with $m\geq 3$. Suppose that
$\widetilde{U}_0:=U_0-\widebar{U}\in H^{m+3/2}(\varOmega) $ and
$\varphi_0 \in H^{m+2}(\mathbb{T}^{d-1})$
satisfy \eqref{varPsi_0}--\eqref{CA1}.
The initial data $(U_0,\varphi_0)$ are said to be compatible up to order $m$ if
the functions $\widetilde{U}_{(j)}$ and $\varphi_{(j)}$ determined by \eqref{trace.id1} and \eqref{trace.id2}
satisfy
\begin{align}\label{compa1}
\left.q_{(j)}\right|_{x_1=0}=0\quad \textrm{for } j=0,\ldots,m.
\end{align}
\end{definition}
It follows from Lemma~\ref{lem:CA1} that $\tilde{q}_{(j)}\in H^{1}(\varOmega)$ for $j=0,\ldots,m$.
Hence, it is permissible to consider the traces of these functions on the boundary $\{x_1=0\}$.
The zero-th order compatibility condition, namely \eqref{compa1} with $j=0$, comes from
the second component of \eqref{NP1b}.

Now we are in a position to construct the approximate solution.
\begin{lemma} \label{lem:app}
Let $m\in\mathbb{N}$ with $m\geq 3$.
Assume that the initial data \eqref{NP1c} satisfy \eqref{hyperbolicity}, \eqref{HN}, \eqref{RT}, the compatibility conditions up to order $m$,
and $(\widetilde{U}_0,\varphi_0) \in H^{m+3/2}(\varOmega) \times H^{m+2}(\mathbb{T}^{d-1})$ for $\widetilde{U}_0:=U_0-\widebar{U}$.
Then there exist positive constants $T_1(M_0)$ and $C(M_0)$ depending on $M_0$ ({\it cf.}~\eqref{M0}), such that if $0<T\leq T_1(M_0)$, then there are functions $U^{a}$ and $\varphi^a$ satisfying
\begin{alignat}{3}
&\big\|\widetilde{U}^{a}\big\|_{H^{m+1}(\varOmega_T)}
+\|\varphi^a\|_{H^{m+5/2}(\varSigma_T)}  \leq C (M_0)
\quad&& \textrm{for }\widetilde{U}^{a}:=U^{a}-\widebar{U},\label{app1a}\\
\label{app1b}
&\rho(U^a)\in(\rho_*,\rho^*),\quad
\p_1\varPhi^{a}\geq \frac{5}{8}\quad&& \textrm{in }  \varOmega_T,
\end{alignat}
where
$\varPhi^{a}:=x_1 +\varPsi^{a} $ with $\varPsi^{a} :=\chi(x_1)\varphi^{a}$. Moreover,
\begin{alignat}{3}
\label{app2a}
&\p_t^k\mathbb{L}(U^{a},\varPhi^{a})\big|_{t=0}=0\quad
\textrm{for all } k\in\{0,\ldots,m-1\},\quad  &&\textrm{in } \varOmega,\\
&\label{app2b}\mathbb{B}(U^a,\varphi^a)=0,
\quad H_{N}^{a}=0,\quad
\p_1 q^{a}\geq \frac{3\kappa_{0}}{4}>0 \quad
&&\textrm{on }  \varSigma_T,\\
&\label{app2c}  U^{a}\big|_{t=0}=U_0\ \ \quad \textrm{in } \varOmega,
\  \ \qquad \varphi^a\big|_{t=0}=\varphi_0   &&\textrm{on } \varSigma,\\[1mm]
\label{app3}
&\mathbb{L}_H(v^{a},H^{a},\varPhi^{a})=0
\quad &&\textrm{in } \varOmega_T,
\end{alignat}
where
$ H_{N}^{a}:=H_{1}^{a}-\sum_{i=2}^{d}\p_{i} \varPsi^{a}  H_{i }^{a}$, and
$\mathbb{L}_H$ denotes the component for the magnetic field of $\mathbb{L}$, that is,
\begin{align} \label{L.bb.H}
\mathbb{L}_H(v,H,\varPhi):=
\big(\p_t^{ { \varPhi}}+v_{i}\p_{i}^{ { \varPhi}}\big)
H-  {H}_{i }\p_{i}^{ {\varPhi}}  {v}+ {H}\p_{i}^{ {\varPhi}}  {v}_{i},
\end{align}
with $\p_t^{\varPhi}$ and $\p_i^{\varPhi}$ defined by \eqref{differential}.
\end{lemma}
\begin{proof}
We divide the proof into four steps.

\vspace*{2mm}
\noindent {\it Step 1}.\quad
We take $\varphi^a\in H^{m+5/2}(\mathbb{R}\times \mathbb{T}^{d-1})$
and $(\tilde{v}_2^{a},\ldots,\tilde{v}_d^{a},\widetilde{S}^{a})\in H^{m+2}(\mathbb{R}\times \varOmega)$
to satisfy
\begin{alignat*}{3}
\p_t^{k} \varphi^a\big|_{t=0}=\varphi_{(k)},\quad
\p_t^{k}(\tilde{v}_2^{a},\ldots,\tilde{v}_d^{a},\widetilde{S}^{a})\big|_{t=0}
=(\tilde{v}_{2(k)},\ldots,\tilde{v}_{d(k)},\widetilde{S}_{(k)}),
\end{alignat*}
for all $k\in\{0,\ldots,m\}.$
Set $\varPsi^{a}:=\chi( x_1) \varphi^{a}\in  H^{m+5/2}(\mathbb{R}\times \varOmega)$
and $\varPhi^{a }:= x_1 +\varPsi^{a } $.

\vspace*{2mm}
\noindent {\it Step 2}.\quad
In view of the compatibility conditions \eqref{compa1},
we can apply the lifting result in \cite[Theorem 2.3]{LM72MR0350178}
to find $\tilde{q}^{a} \in H^{m+2}(\mathbb{R}\times \varOmega)$ such that
\begin{align*}
{q}^{a}=0\quad \textrm{on } \varSigma,\quad
\p_t^{k}\tilde{q}^{a} \big|_{t=0}=\tilde{q}_{(k)}
\quad  \textrm{in }  \varOmega
\quad \textrm{for all } k\in\{0,\ldots,m\}.
\end{align*}
It follows from the trace theorem that
\begin{align*}
w^a:=\p_t\varphi^a+\sum_{i=2}^d \p_i\varphi^a v_i^{a}\big|_{x_1=0} \in H^{m+3/2}(\mathbb{R}\times \mathbb{T}^{d-1}).
\end{align*}
Thanks to \eqref{trace.id1},
we can choose $\tilde{v}_1^{a}\in H^{m+2}(\mathbb{R}\times\varOmega)$ such that
$$v_1^{a} =w^a \quad \textrm{on } \varSigma,\quad  \p_t^{k}v_1^{a}\big|_{t=0}=v_{1(k)}^{}\quad \textrm{in } \varOmega \quad
\textrm{for all } k\in\{0,\ldots,m\}.$$
As a direct consequence, we obtain the first identity in \eqref{app2b}.

\vspace*{2mm}
\noindent {\it Step 3}.\quad
Noting that $\tilde{v}^{a }\in H^{m+2}(\mathbb{R}\times \varOmega)$ and $\varPsi^{a}=\chi( x_1) \varphi^{a} \in H^{m+5/2}(\mathbb{R}\times \varOmega)$ have been already specified,
we take $\widetilde{H}^{a}\in H^{m+1}(\mathbb{R}\times \varOmega)$
as the unique solution of the equations \eqref{app3} supplemented with the initial data
$\widetilde{H}^{a}\big|_{t=0}=\widetilde{H}_{0}. $
Since the second identity in \eqref{app2b} is fulfilled at $t=0$,
similar to the proof of Proposition \ref{pro:1},
we can deduce that the second identity in \eqref{app2b} holds for all $t\in\mathbb{R}$
by considering the restriction of the equations \eqref{app3} to the boundary $\varSigma$.

\vspace*{2mm}
\noindent {\it Step 4}.\quad
We now have obtained \eqref{app2c}--\eqref{app3} and the first two identities in \eqref{app2b}.
Estimate \eqref{app1a} follows from \eqref{CA.est1} and the continuity of the lifting operators.
We deduce \eqref{app1b} and the third relation in \eqref{app2b} by taking $T>0$ small enough.
Equations \eqref{trace.id2} imply \eqref{app2a}.
The proof of the lemma is complete.
\end{proof}
The vector function $(U^a,\varphi^a)$ in Lemma~\ref{lem:app}
is called the approximate solution to the problem \eqref{NP1}.
Let us define
\begin{align}\label{f^a}
f^{a}:=\left\{\begin{aligned}
& -\mathbb{L}(U^{a},\varPhi^{a}) \quad &\textrm{if }t>0,\\
& 0 \quad &\textrm{if }t<0.
\end{aligned}\right.
\end{align}
Since $(\widetilde{U}^{a},{\varphi}^{a})\in  H^{m+1}(\varOmega_T)\times H^{m+5/2}(\varSigma_T)$,
we obtain from \eqref{app1a} and \eqref{app2a} that
$f^{a}\in H^{m}(\varOmega_T)$ and
\begin{align}\label{f^a:est}
\|f^{a}\|_{ H^{m}(\varOmega_T)}\leq
\delta_0\left(T\right),
\end{align}
where $\delta_0(T)\to 0$ as $T\to 0$.
The estimate \eqref{f^a:est} results from the Moser-type calculus and embedding inequalities.

Let $(U^a,\varPhi^a)$ be the approximate solution defined in Lemma~\ref{lem:app}.
By virtue of \eqref{app2a}--\eqref{app2c} and \eqref{f^a},
we find that  $(U,\varphi)=(U^a,\varphi^a)+(V,\psi)$ is a solution of
the nonlinear problem \eqref{NP1} on $[0,T]\times \varOmega$,
provided $V$, $\psi$, and $\varPsi:=\chi(x_1)\psi$ solve
\begin{align} \label{P.new}
\left\{
\begin{aligned}
&\mathcal{L}(V,\varPsi):=\mathbb{L}(U^a+V,\varPhi^a+\varPsi)-\mathbb{L}(U^a,\varPhi^a)=f^a
\ &&\textrm{in }\varOmega_T,\\
&\mathcal{B}(V,\psi):=\mathbb{B}(U^a+V,\varphi^a+\psi)=0
\ &&\textrm{on }\varSigma_T,\\
&(V,\psi)=0,\ &&\textrm{if }t< 0.
\end{aligned}\right.
\end{align}
Thanks to \eqref{app2b}, we find that $(V,\psi)=0$ satisfies \eqref{P.new} for $t<0$.
Therefore, the original problem on $[0,T]\times \varOmega$ is reformulated as a problem in $\varOmega_T$ whose solutions vanish in the past.

\subsection{Iteration Scheme}

We first quote the following result on the smoothing operators from \cite[Proposition 10]{T09ARMAMR2481071}.

\begin{proposition} \label{pro:smooth}
Let $T>0$ and $m\in \mathbb{N}$ with $m\geq 3$.
Let $\mathscr{F}_*^s(\varOmega_T):=\big\{u\in H_*^{s}(\varOmega_T):u=0\textrm{ for }t<0\big\}$.
Then there exists a family $\{\mathcal{S}_{\theta}\}_{\theta\geq 1}$ of smoothing operators from $\mathscr{F}_*^3(\varOmega_T) $ to $\bigcap_{s\geq 3}\mathscr{F}_*^s(\varOmega_T)$, such that
\begin{subequations}\label{smooth.p1}
\begin{alignat}{2}
&  \label{smooth.p1a}
\|\mathcal{S}_{\theta} u\|_{k,*,T}
\leq C \theta^{(k-j)_+}\|u\|_{j,*,T}
&& \textrm{for  \;}k,j\in\{1,\ldots,m\},\\[1.5mm]
&  \label{smooth.p1b}
\|\mathcal{S}_{\theta} u-u\|_{k,*,T}
\leq C  \theta^{k-j}\|u\|_{j,*,T}
&& \textrm{for  \;}1\leq k\leq j \leq m,\\
&  \label{smooth.p1c}
\left\|\frac{\d}{\d \theta}\mathcal{S}_{\theta} u\right\|_{k,*,T}
\leq C \theta^{k-j-1}\|u\|_{j,*,T}
&\quad&\textrm{for \;}k,j\in\{1,\ldots,m\},
\end{alignat}
\end{subequations}
where $k$ and $j$ are integers, $(k-j)_+:=\max\{0,\, k-j \}$,
and the constant $C$ depends only on $m$.
Furthermore, there is another family of smoothing operators (still denoted by $\mathcal{S}_{\theta}$) acting on the functions defined on the boundary $\varSigma_T$ and satisfying the properties in \eqref{smooth.p1} with norms $\|\cdot\|_{H^{j}(\varSigma_T)}$.
\end{proposition}

Now we follow \cite{CS08MR2423311,T09ARMAMR2481071} to describe the iteration scheme for problem \eqref{P.new}.

\vspace*{2mm}
\noindent{\bf Assumption\;(A-1)}: {\it  Set $ (V_0, \psi_0)=0$. Let
	$(V_k,\psi_k)$ be given and vanish in the past, and set $\varPsi_k:=\chi( x_1)\psi_k$,
	for all $k\in\{0,\ldots,{n}\}$.}

\vspace*{2mm}
We consider
\begin{align}\label{NM0}
V_{{n}+1}=V_{{n}}+\delta V_{{n}},
\quad \psi_{{n}+1}=\psi_{{n}}+\delta \psi_{{n}},
\quad  \delta\varPsi_{{n}}:=\chi(x_1)\delta \psi_{{n}},
\end{align}
where the differences $\delta V_{{n}}$ and $\delta \psi_{{n}}$ will be specified via the effective linear problem
\begin{align} \label{effective.NM}
\left\{\begin{aligned}
&\mathbb{L}_e'(U^a+V_{{n}+1/2},\varPhi^a+\varPsi_{{n}+1/2})\delta \dot{V}_{{n}}=f_{{n}}
\ \  &&\textrm{in }\varOmega_T,\\
& \mathbb{B}_e'(U^a+V_{{n}+1/2},\varphi^a+\psi_{{n}+1/2})(\delta \dot{V}_{{n}},\delta\psi_{{n}})=g_{{n}}
\ \ &&\textrm{on }\varSigma_T,\\
& (\delta \dot{V}_{{n}},\delta\psi_{{n}})=0\ \ &&\textrm{for }t<0,
\end{aligned}\right.
\end{align}
with $\varPsi_{{n}+1/2}:=\chi(x_1)\psi_{{n}+1/2}$.
Here $\delta \dot{V}_{{n}}$ is the good unknown ({\it cf.}~\eqref{good}), {\it i.e.},
\begin{align} \label{good.NM}
\delta \dot{V}_{{n}}:=\delta V_{{n}}-\frac{\p_1 (U^a+V_{{n}+1/2})}{\p_1 (\varPhi^a+\varPsi_{{n}+1/2})}\delta\varPsi_{{n}},
\end{align}
and $(V_{{n}+1/2},\psi_{{n}+1/2})$ is a smooth modified state
such that $(U^a+V_{{n}+1/2},\varphi^a+\psi_{{n}+1/2})$ satisfies
the constraints \eqref{bas1a}--\eqref{bas1d};
see Proposition \ref{pro:modified}  for the detailed construction and estimate.
The source terms $f_n$ and $g_n$ will be defined through the accumulated error terms
at Step ${n}$ later on.

\vspace*{2mm}
\noindent{\bf Assumption (A-2)}: {\it Set $f_0:=\mathcal{S}_{\theta_0}f^a$ and $(e_0,\tilde{e}_0,g_0):=0$ for $\theta_0\geq 1$ sufficiently large, and let $(f_k,g_k,e_k,\tilde{e}_k)$ be given and vanish in the past for all $k\in\{1,\ldots,{n}-1\}$.}

\vspace*{2mm}

Under Assumptions (A-1)--(A-2), we can describe our iteration scheme as follows.
First we compute the accumulated error terms at Step $n$ for $n\geq 1$ by
\begin{align}  \label{E.E.tilde}
E_{{n}}:=\sum_{k=0}^{{n}-1}e_{k},\quad \widetilde{E}_{{n}}:=\sum_{k=0}^{{n}-1}\tilde{e}_{k}.
\end{align}
Then we calculate the terms $f_{{n}}$ and $g_{{n}}$ from
\begin{align} \label{source}
\sum_{k=0}^{{n}} f_k+\mathcal{S}_{\theta_{{n}}}E_{{n}}=\mathcal{S}_{\theta_{{n}}}f^a,
\quad \sum_{k=0}^{{n}}g_k+\mathcal{S}_{\theta_{{n}}}\widetilde{E}_{{n}}=0,
\end{align}
where $\mathcal{S}_{\theta_{{n}}}$ are the smoothing operators given in Proposition \ref{pro:smooth} with the sequence $\{\theta_{{n}}\}$ defined by
\begin{align} \label{theta}
\theta_0\geq 1,\quad \theta_{{n}}=\sqrt{\theta^2_0+{n}}.
\end{align}
Once $f_n$ and $g_n$ are specified, we can apply Theorem \ref{thm:3} to
solve $(\delta \dot{V}_{{n}},\delta \psi_{{n}})$ from the problem \eqref{effective.NM}.
Then we obtain the function $\delta V_{{n}}$
and $(V_{{n}+1},\psi_{{n}+1})$ from \eqref{good.NM} and \eqref{NM0}.

The error terms at Step ${n}$ are defined through the following decompositions:
\begin{align}
\nonumber&\mathcal{L}(V_{{n}+1},\varPsi_{{n}+1})-\mathcal{L}(V_{{n}},\varPsi_{{n}})\\
\nonumber&\  = \mathbb{L}'(U^a+V_{{n}},\varPhi^a+\varPsi_{{n}})(\delta V_{{n}},\delta\varPsi_{{n}})+e_{{n}}'\\
\nonumber&\  = \mathbb{L}'(U^a+\mathcal{S}_{\theta_{{n}}}V_{{n}},\varPhi^a+\mathcal{S}_{\theta_{{n}}}\varPsi_{{n}})(\delta V_{{n}},\delta\varPsi_{{n}})+e_{{n}}'+e_{{n}}''\\
\nonumber&\ = \mathbb{L}'(U^a+V_{{n}+1/2},\varPhi^a+\varPsi_{{n}+1/2})(\delta V_{{n}},\delta\varPsi_{{n}})+e_{{n}}'+e_{{n}}''+e_{{n}}'''\\
\label{decom1}&\  = \mathbb{L}_e'(U^a+V_{{n}+1/2},\varPhi^a+\varPsi_{{n}+1/2})\delta \dot{V}_{{n}}+e_{{n}}'+e_{{n}}''+e_{{n}}'''+D_{{n}+1/2} \delta\varPsi_{{n}},
\end{align} and \begin{align}
\nonumber&\mathcal{B}(V_{{n}+1},\psi_{{n}+1})-\mathcal{B}(V_{{n}},\psi_{{n}})\\
\nonumber&\  = \mathbb{B}'(U^a+V_{{n}},\varphi^a+\psi_{{n}})(\delta V_{{n}},\delta\psi_{{n}})+\tilde{e}_{{n}}'\\
\nonumber&\ = \mathbb{B}'(U^a+\mathcal{S}_{\theta_{{n}}}V_{{n}},\varphi^a+\mathcal{S}_{\theta_{{n}}}\psi_{{n}})(\delta V_{{n}},\delta\psi_{{n}})+\tilde{e}_{{n}}'+\tilde{e}_{{n}}''\\
\label{decom2}&\  =\mathbb{B}_e'(U^a+V_{{n}+1/2},
\varphi^a+\psi_{{n}+1/2})(\delta \dot{V}_{{n}} ,\delta\psi_{{n}})+\tilde{e}_{{n}}'+\tilde{e}_{{n}}''+\tilde{e}_{{n}}''',
\end{align}
with
\begin{align}\label{error.D}
D_{{n}+1/2}:=\frac{1}{\p_1(\varPhi^a+\varPsi_{{n}+1/2})}\p_1\mathbb{L}(U^a+V_{{n}+1/2},\varPhi^a+\varPsi_{{n}+1/2}),
\end{align}
where we utilize \eqref{Alinhac} to obtain the last identity in \eqref{decom1}.
Setting
\begin{align} \label{e.e.tilde}
e_{{n}}:=e_{{n}}'+e_{{n}}''+e_{{n}}'''+D_{{n}+1/2} \delta\varPsi_{{n}},\quad
\tilde{e}_{{n}}:=\tilde{e}_{{n}}'+\tilde{e}_{{n}}''+\tilde{e}_{{n}}''',
\end{align}
we complete the description of the iteration scheme.

Summing \eqref{decom1} and \eqref{decom2} from ${n}=0$ to $N$, respectively,
we use \eqref{effective.NM} and \eqref{E.E.tilde}--\eqref{source}  to find
\begin{align}
\nonumber 
&\mathcal{L}(V_{N+1},\varPsi_{N+1})=\sum_{{n}=0}^{N}f_{{n}}+E_{N+1}=\mathcal{S}_{\theta_{N}}f^a+({I}-\mathcal{S}_{\theta_{N}})E_{N}+e_{N},\\
\nonumber 
&\mathcal{B}(V_{N+1},\psi_{N+1})=\sum_{{n}=0}^{N}g_{{n}}+\widetilde{E}_{N+1}=({I}-\mathcal{S}_{\theta_{N}})\widetilde{E}_{N}+\tilde{e}_{N}.
\end{align}
Since $\mathcal{S}_{\theta_{N}}\to {I}$ as $N\to \infty$, we can formally obtain the solution to the problem \eqref{P.new}
from $\mathcal{L}(V_{N+1},\varPsi_{N+1})\to f^a$ and  $\mathcal{B}(V_{N+1},\psi_{N+1})\to 0,$
provided $(e_{N},\tilde{e}_{N})\to 0$  as $N\to \infty$.

\subsection{Inductive Hypothesis}

Let $m\in \mathbb{N}$ with $m\geq {13}$ and $\widetilde{\alpha}:=m-{5}$.
Suppose that the initial data \eqref{NP1c} satisfy $(\widetilde{U}_0,\varphi_0) \in H^{m+3/2}(\varOmega) \times H^{m+2}(\mathbb{T}^{d-1})$ for $\widetilde{U}_0:=U_0-\widebar{U}$.
Thanks to Lemma~\ref{lem:app}, the following estimates hold:
\begin{align} \label{small}
\big\|\widetilde{U}^a\big\|_{H^{\widetilde{\alpha}+{6}}(\varOmega_T)}
+\big\|\varphi^a\big\|_{H^{\widetilde{\alpha}+{15/2}}(\varSigma_T)}
\leq C(M_0),\  
\big\|f^a\big\|_{H^{\widetilde{\alpha}+{5}}(\varOmega_T)}\leq \delta_0(T),
\end{align}
where $M_0$ is defined by \eqref{M0} and $\delta_0(T)\to 0$ as $T\to 0$.
Suppose further that Assumptions (A-1)--(A-2) are fulfilled.
For an integer $ {\alpha }\geq 7$ and a real number $\varepsilon>0$,
our inductive hypothesis reads
\begin{align*}
(\mathbf{H}_{{n}-1})\ \left\{\begin{aligned}
\textrm{(a)}\,\,  &\|(\delta V_k,\delta \varPsi_k)\|_{s,*,T}+\|\delta\psi_k\|_{H^{s}(\varSigma_T)}\leq \varepsilon \theta_k^{s-{\alpha }-1}\varDelta_k\\
&\quad \textrm{for all } k\in\{0,\ldots,{n}-1\}\textrm{ and }s\in \{6,\ldots,\widetilde{\alpha} \};\\
\textrm{(b)}\,\, &\|\mathcal{L}( V_k,  \varPsi_k)-f^a\|_{s,*,T}\leq 2 \varepsilon \theta_k^{s-{\alpha }-1}\\
&\quad \textrm{for all } k\in\{0,\ldots,{n}-1\}\textrm{ and } s\in\{ 6,\ldots,\widetilde{\alpha}-2\};\\
\textrm{(c)}\,\,  &\|\mathcal{B}( V_k,  \psi_k)\|_{H^{s}(\varSigma_T)}\leq  \varepsilon \theta_k^{s-{\alpha }-1}\\
&\quad \textrm{for all } k\in\{0,\ldots,{n}-1\}\textrm{ and } s\in\{7,\ldots,{\alpha }\},
\end{aligned}\right.
\end{align*}
where $\theta_k$ is defined by \eqref{theta} and $\varDelta_{k}:=\theta_{k+1}-\theta_k$.
Note that $ {1}/{3} \leq \theta_k\varDelta_{k}\leq  {1}/{2}  $ for all $k\in\mathbb{N}$.
We will choose ${\alpha }$, $\widetilde{\alpha}>{\alpha }$, and $\varepsilon$ later on.

We aim to prove that hypothesis ($\mathbf{H}_{{n}-1}$) implies ($\mathbf{H}_{{n}}$) and that ($\mathbf{H}_0$) holds,
provided $T>0$ and $\varepsilon>0$ are sufficiently small and
$\theta_0\geq 1$ is large enough.
Supposing that hypothesis ($\mathbf{H}_{{n}-1}$) holds, we have the following result.

\begin{lemma}[{\cite[Lemma 7]{T09ARMAMR2481071}}] \label{lem:triangle}
	If $\theta_0$ is large enough, then
	\begin{align}
	\label{tri1}&\|( V_k, \varPsi_k)\|_{s,*,T}+\|\psi_k\|_{H^{s}(\varSigma_T)}
	\leq
	\left\{\begin{aligned}
	&\varepsilon \theta_k^{(s-{\alpha })_+}   &&\textrm{if }s\neq {\alpha },\\
	&\varepsilon \log \theta_k   &&\textrm{if }s= {\alpha },
	\end{aligned}\right.\\
	&\| ({I}-\mathcal{S}_{\theta_k})(V_k,   \varPsi_k)\|_{s,*,T}
	\label{tri2}  +\|({I}-\mathcal{S}_{\theta_k})\psi_k\|_{H^{s}(\varSigma_T)}
	\leq C\varepsilon \theta_k^{s-{\alpha }},
	\end{align}
	for all $k\in\{0,\ldots,{n}-1\}$ and  $s\in\{6,\ldots,\widetilde{\alpha}\}$.
	Furthermore,
	\begin{align}
	\label{tri3}&\|( \mathcal{S}_{\theta_k}V_k, \mathcal{S}_{\theta_k}\varPsi_k)\|_{s,*,T}
	+\|\mathcal{S}_{\theta_k}\psi_k\|_{H^{s}(\varSigma_T)}\leq
	\left\{\begin{aligned}
	&C\varepsilon \theta_k^{(s-{\alpha })_+}  &&\textrm{if }s\neq {\alpha },\\
	&C\varepsilon \log \theta_k  &&\textrm{if }s= {\alpha },
	\end{aligned}\right.
	\end{align}
	for all $k\in\{0,\ldots,{n}-1\}$ and $s\in\{ 6,\ldots,\widetilde{\alpha}+6\}$.
\end{lemma}

\subsection{Estimates of the Error Terms}

For deriving ($\mathbf{H}_{{n}}$) from ($\mathbf{H}_{{n}-1}$), in this subsection,
we estimate the quadratic error terms $e'_{k}$ and $\tilde{e}_{k}'$,
the first substitution error terms $e_{k}''$ and $\tilde{e}_{k}''$,
the second substitution error terms $e_{k}''' $  and  $\tilde{e}_{k}'''$,
and the last error term $D_{k+1/2} \delta\varPsi_{k}$
({\it cf.}~\eqref{decom1}--\eqref{error.D}).
First we find
\begin{align*}
e_k' =\;&\int_{0}^{1}
\mathbb{L}''\big(U^a+V_{k}+\tau \delta  V_{k},
\varPhi^a+\varPsi_{k}\\
&\quad \quad\ +\tau \delta \varPsi_{k}\big)\big((\delta V_{k},\delta\varPsi_{k}),(\delta V_{k},\delta\varPsi_{k})\big)
(1-\tau)\d\tau,\\
\tilde{e}_k'
=\;&
\frac{1}{2}\,\mathbb{B}''\big((\delta V_{k},\delta\psi_{k}),(\delta V_{k},\delta\psi_{k})\big),
\end{align*}
where  $\mathbb{L}''$ and $\mathbb{B}''$ are, respectively, the second derivatives of the operators
$\mathbb{L}$ and $\mathbb{B}$.
More precisely, we define
\begin{align*}
&\mathbb{L}''\big(\mathring{U},\mathring{\varPhi}\big)
\big((V,\varPsi),(\widetilde{V},\widetilde{\varPsi})\big)
:=\left.\frac{\d}{\d \theta}
\mathbb{L}'\big(\mathring{U}+\theta \widetilde{V},
\mathring{\varPhi}+\theta \widetilde{\varPsi}\big)
\big(V,\varPsi\big)\right|_{\theta=0},\\
&\mathbb{B}''
\big((V,\psi),(\widetilde{V},\tilde{\psi})\big)
:=\left.\frac{\d}{\d \theta}
\mathbb{B}'(\mathring{U}+\theta \widetilde{V},
\mathring{\varphi}+\theta \tilde{\psi}) (V,\psi)\right|_{\theta=0},
\end{align*}
where the operators $\mathbb{L}'$ and $\mathbb{B}'$ are given in \eqref{L'.bb}--\eqref{B'.bb}.
For our problem, we compute
\begin{align}
\label{B''.form}
&\mathbb{B}'' \big((V,\psi),(\widetilde{V},\tilde{\psi})\big)
=
\begin{pmatrix}
\sum_{i=2}^{d}\left(\tilde{v}_i   \p_i \psi + \p_i \tilde{\psi} v_i  \right)\\[1mm]
0
\end{pmatrix}.
\end{align}

To control the error terms, we need  estimates for the operators $\mathbb{L}''$ and $\mathbb{B}''$.
These estimates can be obtained from the explicit forms of $\mathbb{L}''$ and $\mathbb{B}''$
by applying the Moser-type calculus inequalities in Lemmas \ref{lem:Moser1}--\ref{lem:Moser2}.
Omitting detailed calculations, we have the following proposition.

\begin{proposition}  \label{pro:tame2}
	Let $T>0$ and $s\in\mathbb{N}$ with $s\geq 6$.
	Assume that $(\widetilde{V},\widetilde\varPsi)\in H_*^{s+2}(\varOmega_T)$ satisfies
	$$\|(\widetilde{V},\widetilde{\varPsi} )\|_{W_*^{2,\infty}(\varOmega_T)}
\leq \widetilde{K}$$
	for some constant $\widetilde{K}>0$.
	Then there exists a positive constant $C$, depending on $\widetilde{K}$ but not on $T$,
	such that, if $(V_i,\varPsi_i)\in H_*^{s+2}(\varOmega_T)$ and $(W_i,\psi_i)\in H^{s}(\varSigma_T)\times H^{s+1}(\varSigma_T)$ for $i=1,2$, then
	\begin{align}
	\nonumber
&\big\|\mathbb{L}''\big(\widebar{U}+\widetilde{V},x_1+\widetilde{\varPsi} \big)\big((V_1,\varPsi_1),(V_2,\varPsi_2) \big)\big\|_{s,*,T}\\[1mm]
\nonumber &\quad 	\leq C\big\|\big({\widetilde{V}},\widetilde{\varPsi}\big)\big\|_{s+2,*,T}
\|(V_1,\varPsi_1)\|_{W_*^{2,\infty}(\varOmega_T)}  \|(V_2,\varPsi_2)\|_{W_*^{2,\infty}(\varOmega_T)}
\\[1mm]
&\qquad
\nonumber + C  \sum_{i\neq j}\|(V_i,\varPsi_i)\|_{s+2,*,T} \|(V_j,\varPsi_j)\|_{W_*^{2,\infty}(\varOmega_T)}
 ,
	\end{align}
and
	\begin{align}
\nonumber
&\big\|\mathbb{B}''\big((W_1,\psi_1),(W_2,\psi_2) \big)\big\|_{H^{s}(\varSigma_T)}\\[1mm]
	\nonumber&\quad   \leq C \sum_{i\neq j}
	\Big\{\big\|W_i\big\|_{H^{s}(\varSigma_T)}\|\psi_j\|_{W^{1,\infty}(\varSigma_T)}
	+ \|W_i\|_{L^{\infty}(\varSigma_T)} \big\|\psi_j\big\|_{H^{s+1}(\varSigma_T)}
	\Big\}.
	\end{align}
\end{proposition}

We first apply Proposition \ref{pro:tame2} to deduce
the following estimate for the quadratic error terms $e'_{k}$ and $\tilde{e}_{k}'$.

\begin{lemma}\label{lem:quad}
	Let ${\alpha }\geq 7$.
	If $\varepsilon>0$ is sufficiently small and $\theta_0\geq 1$ is suitably large, then
	\begin{align}\label{quad.est}
	\|e_k'\|_{s,*,T}+\|\tilde{e}_k'\|_{H^{s}(\varSigma_T)}
	\lesssim \varepsilon^2 \theta_k^{\varsigma_1(s)-1}\varDelta_k,
	\end{align}
	for all $k\in \{0,\ldots,{n}-1\}$ and $s\in \{ 6,\ldots,\widetilde{\alpha}-2\}$,
	where  $\varsigma_1(s):=\max\{(s+2-{\alpha })_++10-2{\alpha },s+6-2{\alpha } \}$.
\end{lemma}
\begin{proof}
In view of the assumption \eqref{small}, the hypothesis $(\mathbf{H}_{{n}-1})$, and the estimate \eqref{tri1},
we utilize \eqref{embed2} to obtain
\begin{align*}
\|(\widetilde{ U}^a,\,V_{k},\,\delta  V_{k},\varPsi^a,\,\varPsi_{k},\, \delta \varPsi_{k})\|_{W_*^{2,\infty}(\varOmega_T)}
\lesssim 1.
\end{align*}
Then we can apply Proposition \ref{pro:tame2} and use the embedding inequalities
\eqref{embed2}, the assumption \eqref{small}, and the hypothesis $(\mathbf{H}_{{n}-1})$
to get
\begin{align*}
\|e_k'\|_{s,*,T}\lesssim \;& \varepsilon^2\theta_k^{10-2{\alpha }}\varDelta_k^2\big(1+\|(V_k,\varPsi_k)\|_{s+2,*,T}
+\varepsilon \theta_k^{s+1-{\alpha }}\varDelta_{k} \big)\\
&+ \|(\delta V_{k},\delta\varPsi_{k})\|_{s+2,*,T}\|(\delta V_{k},\delta\varPsi_{k})\|_{6,*,T}\\
\lesssim \;& \varepsilon^2\theta_k^{10-2{\alpha }}\varDelta_k^2\big(1+\|(V_k,\varPsi_k)\|_{s+2,*,T}\big)
+\varepsilon^2\theta_{k}^{s+6-2{\alpha }}\varDelta_k^2
\end{align*}
for all $s\in  \{ 6,\ldots,\widetilde{\alpha}-2\}$.

If $s+2\neq {\alpha }$, then it follows from \eqref{tri1} and $2\theta_k\varDelta_k\leq 1$ that
\begin{align*}
\|e_k'\|_{s,*,T}\lesssim  \varepsilon^2 \varDelta_k^2\big(\theta_k^{(s+2-{\alpha })_+
	+10-2{\alpha }}+\theta_k^{s+6-2{\alpha }} \big)\lesssim  \varepsilon^2 \theta_k^{\varsigma_1(s)-1 }\varDelta_k.
\end{align*}

If $s+2= {\alpha }$, then we use \eqref{tri1} and ${\alpha }\geq 7$ to find
\begin{align*}
\|e_k'\|_{{\alpha }-2,*,T}
\lesssim  \varepsilon^2 \varDelta_k^2\left\{ \theta_k^{11-2{\alpha }} +\theta_k^{4-{\alpha }}\right\}
\lesssim \varepsilon^2 \theta_k^{\varsigma_1({\alpha }-2)-1 }\varDelta_k.
\end{align*}

Employing Proposition \ref{pro:tame2} and Lemma~\ref{lem:trace}
yields the estimate for $\tilde{e}_k'$ and completes the proof of the lemma.
\end{proof}

For the first substitution error terms $e_{k}''$ and $\tilde{e}_{k}''$ defined in \eqref{decom1}--\eqref{decom2}, we have the following result.
\begin{lemma} \label{lem:1st}
	Let ${\alpha }\geq 7$.
	If $\varepsilon>0$ is sufficiently small and $\theta_0\geq 1$ is suitably large, then
	\begin{align}\label{1st.sub}
	\|e_k'' \|_{s,*,T}
	+\|\tilde{e}_k''\|_{H^{s}(\varSigma_T)}\lesssim \varepsilon^2 \theta_k^{\varsigma_2(s)-1}\varDelta_k,
	\end{align}
	for all $k\in \{0,\ldots,{n}-1\}$ and $s\in \{ 6,\ldots,\widetilde{\alpha}-2\}$,
	where
\begin{align}\label{varsigma2.def}
\varsigma_2(s):=\max\{(s+2-{\alpha })_++12-2{\alpha },s+8-2{\alpha } \}.
\end{align}
\end{lemma}
\begin{proof}
First we have
\begin{align}
\nonumber {e}_k''=\;&\int_{0}^{1}\mathbb{L}''\Big(U^a+\mathcal{S}_{\theta_k}V_k+\tau({I}-\mathcal{S}_{\theta_k})V_k,\,
\varPhi^a+\mathcal{S}_{\theta_k}\varPsi_k \\
&\quad\quad\ +\tau({I}-\mathcal{S}_{\theta_k})\varPsi_k\Big)\Big(\big(\delta V_k ,\delta \varPsi_k\big),\, \big(({I}-\mathcal{S}_{\theta_k})V_k,({I}-\mathcal{S}_{\theta_k})\varPsi_k\big) \Big)\d\tau,\nonumber\\
\nonumber\tilde{e}_k''=\, &\mathbb{B}''\Big(\big(\delta V_k,\delta \psi_k\big),\, \big(({I}-\mathcal{S}_{\theta_k})V_k ,({I}-\mathcal{S}_{\theta_k})\psi_k  \big) \Big) .
\nonumber
\end{align}
Thanks to  \eqref{tri2}--\eqref{tri3}, we have
\begin{align*}
\|(  \mathcal{S}_{\theta_k}V_k,  V_k,\mathcal{S}_{\theta_k}\varPsi_k, \varPsi_k)\|_{W_*^{2,\infty}(\varOmega_T)}
\lesssim
\|(  \mathcal{S}_{\theta_k}V_k,  V_k,\mathcal{S}_{\theta_k}\varPsi_k, \varPsi_k)\|_{6,*,T}
 \lesssim 1.
\end{align*}
Then we apply Proposition \ref{pro:tame2} and
use \eqref{embed2}, \eqref{small}, the hypothesis $(\mathbf{H}_{{n}-1})$, and \eqref{tri2}  to get
\begin{align*}
\|{e}_k''\|_{s,*,T}
\lesssim  \varepsilon^2 \theta_k^{11-2{\alpha }}\varDelta_{k} \left(1 + \|(\mathcal{S}_{\theta_k} V_k, \mathcal{S}_{\theta_k} \varPsi_k)\|_{s+2,*,T}   \right)+  \varepsilon^2 \theta_k^{s+7-2{\alpha }}\varDelta_{k}
\end{align*}
for all $s\in \{ 6,\ldots,\widetilde{\alpha}-2\}$.
Similar to the proof of Lemma~\ref{lem:quad},
analyzing the cases $s+2\neq {\alpha }$ and $s+2= {\alpha }$ separately,
we use \eqref{tri3} to deduce \eqref{1st.sub} and finish the proof.
\end{proof}

Let us construct the smooth modified state $(V_{{n}+1/2},\psi_{{n}+1/2})$ so that $(U^a+V_{{n}+1/2},\varphi^a+\psi_{{n}+1/2})$ satisfies the constraints \eqref{bas1a}--\eqref{bas1d}.
Since the smooth modified state will be chosen to vanish in the past and
the approximate solution satisfies \eqref{app1a}--\eqref{app1b} and \eqref{app2b},
the state $(U^a+V_{{n}+1/2},\varphi^a+\psi_{{n}+1/2})$ will satisfy \eqref{bas1a}, \eqref{bas1c}, and \eqref{bas1d} for $T>0$ small enough. Consequently, we only need to focus on the constraints \eqref{bas1b}.

\begin{proposition}\label{pro:modified}
Let ${\alpha }\geq 8$.
Then there exist functions $V_{n+1/2}$ and $\psi_{n+1/2}$ vanishing in the past, such that
$(U^a+V_{n+1/2},\varphi^a+\psi_{n+1/2})$ satisfies \eqref{bas1b}, where $(U^a, \varphi^a)$ is the approximate solution constructed in Lemma~\ref{lem:app}. Furthermore,
\begin{alignat}{3}\label{MS.id1}
&\psi_{n+1/2}=\mathcal{S}_{\theta_n}\psi_{{n}},\quad
v_{i,n+1/2}=\mathcal{S}_{\theta_n} v_{i,n} \quad &&
\textrm{for } i=2,\ldots,d, \\
\label{MS.est1}
&\|\mathcal{S}_{\theta_n}\varPsi_n-\varPsi_{n+1/2}\|_{s,*,T}\lesssim \varepsilon \theta_n^{s-{\alpha }}
\quad&& \textrm{for } s\in\{ 6,\ldots,\widetilde{\alpha}+6\},\\
\label{MS.est2}
&\|\mathcal{S}_{\theta_n}V_n-V_{n+1/2}\|_{s,*,T}\lesssim \varepsilon \theta_n^{s+2-{\alpha }}
\quad&& \textrm{for } s\in\{6,\ldots,\widetilde{\alpha}+{4}\}.
\end{alignat}
\end{proposition}
\begin{proof}
The proof is divided into three steps.

\vspace*{2mm}
\noindent {\it Step 1}.\quad
We define $\psi_{{n}+1/2}$ and $v_{i,n+1/2}$ for $i=2,\ldots,d$ by \eqref{MS.id1}.
Let $\varPsi_{{n}+1/2}:=\chi(x_1)\psi_{n+1/2}$.
Let us define $q_{{n}+1/2}:=\mathcal{S}_{\theta_n} q_n$ and $S_{{n}+1/2}:=\mathcal{S}_{\theta_n} S_n$.

If $6\leq s\leq \widetilde{\alpha}$, then we use \eqref{smooth.p1b} and \eqref{tri1} to get
\begin{align}
\nonumber&\| \mathcal{S}_{\theta_n}\varPsi_n-\varPsi_{n+1/2}\|_{s,*,T}\\
\nonumber &\quad \lesssim \|(\mathcal{S}_{\theta_{{n}}}-{I})(\chi( x_1) \psi_{{n}})\|_{s,*,T}
+\|\chi( x_1)({I}-\mathcal{S}_{\theta_{{n}}})\psi_{{n}}\|_{s,*,T}\\
&\quad  \lesssim \theta_{{n}}^{s-\widetilde{\alpha}} \left(
\|\chi( x_1) \psi_{{n}}\|_{\widetilde{\alpha},*,T}
+\|\psi_n\|_{H^{\widetilde{\alpha}}(\varSigma_T)}\right)
\lesssim \varepsilon \theta_{{n}}^{s-{\alpha }}.
\nonumber
\end{align}
If $\widetilde{\alpha}<s\leq \widetilde{\alpha}+6$, then the estimate \eqref{tri3} gives
\begin{align}
\| \mathcal{S}_{\theta_n}\varPsi_n-\varPsi_{n+1/2}\|_{s,*,T}\nonumber
 \lesssim
 \| \mathcal{S}_{\theta_n}\varPsi_n\|_{s,*,T}
 +\|\mathcal{S}_{\theta_n}\psi_{{n}} \|_{H^{s}(\varSigma_T)}
\lesssim \varepsilon \theta_{{n}}^{s-{\alpha }}.
\nonumber
\end{align}
This finishes the proof of \eqref{MS.est1}.

\vspace*{2mm}
\noindent {\it Step 2}.\quad
Next we define and estimate $v_{1,n+1/2}$ such that the first identity in \eqref{bas1b} holds for
$(U^a+V_{n+1/2},\varphi^a+\psi_{n+1/2})$.
Let us define
\begin{gather*}
v_{1,n+1/2}:=
\mathcal{S}_{\theta_{{n}}}v_{1,n}+
\mathfrak{R}_T\left(\hat{w}_{n}- \left.\left(\mathcal{S}_{\theta_{{n}}}v_{1,n} \right)\right|_{x_1=0} \right),
\end{gather*}
where $\mathfrak{R}_T$ is the lifting operator given in Lemma~\ref{lem:trace} and
\begin{gather*} 
\hat{w}_{n}:=\p_t \psi_{n+1/2} +
\sum_{i=2}^{d} \left.\Big\{ (v_i^{a}+v_{i,n+1/2}) \p_i \psi_{n+1/2} +v_{i,n+1/2}
\p_i \varphi^a \Big\}\right|_{x_1=0}.
\end{gather*}
By virtue of \eqref{app2b},
we infer that $(v^a+v_{n+1/2},\varphi^a+\psi_{n+1/2})$ satisfies the first constraint in \eqref{bas1b}.

Use \eqref{P.new} and \eqref{MS.id1} to get
\begin{align*}
\mathcal{B}\big( \mathcal{S}_{\theta_{{n}}}V_{n} ,\mathcal{S}_{\theta_{{n}}}\psi_n\big) _1
=\;& \hat{w}_{n}-\big(\mathcal{S}_{\theta_{{n}}}v_{1,n}\big)\big|_{x_1=0},
\end{align*}
which together with Lemma~\ref{lem:trace} implies
\begin{align*}
\|v_{1,n+1/2} -\mathcal{S}_{\theta_{{n}}}v_{1,n} \|_{s,*,T}
\lesssim \|\mathcal{B}\big( \mathcal{S}_{\theta_{{n}}}V_{n} ,\mathcal{S}_{\theta_{{n}}}\psi_n\big) _1\|_{H^{s-1}(\varSigma_T)}
\quad \textrm{for }s\geq 6.
\end{align*}
To estimate the right-hand side, we use the decomposition
\begin{align*}
\mathcal{B}\big( \mathcal{S}_{\theta_{{n}}}V_{n} ,\mathcal{S}_{\theta_{{n}}}\psi_n\big) _1
=\mathcal{I}_1+\mathcal{I}_2+
\mathcal{S}_{\theta_n}\mathcal{B}\big(V_{n-1} ,\psi_{n-1}\big)_1,
\end{align*}
where $\mathcal{I}_1:=\mathcal{B}\left( \mathcal{S}_{\theta_{{n}}}V_{n}  ,\mathcal{S}_{\theta_{{n}}}\psi_n\right) _1
-\mathcal{S}_{\theta_n}\mathcal{B}\left(V_n ,\psi_n\right)_1$ is decomposed further as
\begin{align*}
\mathcal{I}_1=\;&
\left\{ \p_t(\mathcal{S}_{\theta_n}\psi_n)-\mathcal{S}_{\theta_n}\p_t\psi_n\right\}
-\left\{ \big(\mathcal{S}_{\theta_n} v_{1,n}\big)\big|_{x_1=0}
-\mathcal{S}_{\theta_n} \big(v_{1,n}\big|_{x_1=0}\big)\right\}\\
&+\sum_{i=2}^d \left\{ \big(\mathcal{S}_{\theta_n} v_{i,n} +v_{i}^{a} \big)\big|_{x_1=0}
\p_i \mathcal{S}_{\theta_n} \psi_n
-\mathcal{S}_{\theta_n}  \big( \big( v_{i,n} +v_{i}^{a} \big)\big|_{x_1=0}\p_i  \psi_n\big)\right\}\\
&+\sum_{i=2}^d \left\{ \big(\mathcal{S}_{\theta_n} v_{i,n}\big)\big|_{x_1=0}\p_i   \varphi^a
-\mathcal{S}_{\theta_n}  \big(  v_{i,n}  \big|_{x_1=0}\p_i  \varphi^a\big)\right\}.
\end{align*}
Noting that $(\widetilde{ U}^a,\varphi^a)\in H_*^{\widetilde{\alpha}+{6}}(\varOmega_T)\times  H^{\widetilde{\alpha}+{15/2}}(\varSigma_T)$,
we use Lemma~\ref{lem:Moser1},
the {Sobolev embedding and trace} theorems, and \eqref{tri3} to infer
\begin{align*}
 \|  (\mathcal{S}_{\theta_n} v_{i,n} +v_{i}^{a}  )
\p_i \mathcal{S}_{\theta_n} \psi_n  \|_{H^{s-1}(\varSigma_T)}
\lesssim \;&
(1+\|\mathcal{S}_{\theta_n} v_{i,n} +\tilde{v}_{i}^{a}\|_{6,*,T}) \|  \mathcal{S}_{\theta_n} \psi_n  \|_{H^{s}(\varSigma_T)}\\
&+\|\mathcal{S}_{\theta_n} v_{i,n} +\tilde{v}_{i}^{a}\|_{s,*,T} \|  \mathcal{S}_{\theta_n} \psi_n  \|_{H^{6}(\varSigma_T)}\\
\lesssim \;& \varepsilon \theta_n^{s-{\alpha }}
\qquad \textrm{for } s\in\{{\alpha }+1,\ldots,\widetilde{\alpha}+{6}\}.
\end{align*}
Since the other terms in $\mathcal{I}_1$ can be estimated as in the proof of \cite[Proposition 12]{T09ARMAMR2481071}, we omit the details and obtain
\begin{align*}
\| \mathcal{I}_1\|_{H^{s-1}(\varSigma_T)}
\lesssim \varepsilon \theta_n^{s-{\alpha }}\quad \textrm{for }\ s =6,\ldots, \widetilde{\alpha}+{6}.
\end{align*}
For the term $\mathcal{I}_2:=\mathcal{S}_{\theta_{{n}}}\! \left( \mathcal{B}\big(V_n ,\psi_n\big)_1
-\mathcal{B}\big(V_{n-1}, \psi_{n-1}\big)_1  \right)$,
we have
\begin{align*}
\mathcal{I}_2=\;&
\mathcal{S}_{\theta_{{n}}} (\p_t \delta\psi_{n-1}) -\mathcal{S}_{\theta_{{n}}}(\delta v_{1,n-1} |_{x_1=0} )
\\
&\!+\sum_{i=2}^d \mathcal{S}_{\theta_{{n}}}  \Big(
 (v_{i,n} +v_{i}^{a}  ) |_{x_1=0}\p_i \delta\psi_{n-1}+
\delta v_{i,n-1}   |_{x_1=0}\p_i (\psi_{n-1}  +\varphi^a )\Big).
\end{align*}
Use  \eqref{smooth.p1a}, hypothesis ($\mathbf{H}_{n-1}$),
and Lemma~\ref{lem:Moser1} to get
\begin{align*}
\| \mathcal{I}_2\|_{H^{s-1}(\varSigma_T)}
\lesssim \varepsilon \theta_n^{s-{\alpha }}\quad \textrm{for }\ s =6,\ldots, \widetilde{\alpha}+{6}.
\end{align*}
Thanks to \eqref{smooth.p1} and hypothesis ($\mathbf{H}_{n-1}$), we have
\begin{align*}
\| \mathcal{S}_{\theta_n}\mathcal{B}(V_{n-1} ,\psi_{n-1})\|_{H^{s-1}(\varSigma_T)}
\lesssim \theta_n^{s-6} \|\mathcal{B}(V_{n-1},\psi_{n-1})\|_{H^7(\varSigma_T)}
\lesssim \varepsilon \theta_n^{s-{\alpha }}
\end{align*}
for $s =6,\ldots, \widetilde{\alpha}+{6}$.
In conclusion, we infer
\begin{align}
\|v_{n+1/2}-\mathcal{S}_{\theta_{{n}}}v_{n}\|_{s,*,T}
\lesssim  \varepsilon \theta_n^{s-{\alpha }}
\quad \textrm{for }\ s =6,\ldots, \widetilde{\alpha}+{6}.
\label{MS.e2}
\end{align}

\vspace*{2mm}
\noindent {\it Step 3}.\quad
Finally we define and estimate $H_{n+1/2}$ such that the second identity in \eqref{bas1b} holds for
$(U^a+V_{n+1/2},\varphi^a+\psi_{n+1/2})$.
Noting that ${v}_{n+1/2}  $ and $\varPsi_{n+1/2} $ have been specified,
we take ${H}_{n+1/2} $ as the unique solution of the linear transport equations
\begin{align} \label{MS.eq1}
\mathbb{L}_H(v^{a}+{v}_{n+1/2},H^{a}+{H}_{n+1/2},\varPhi^{a}+\varPsi_{n+1/2})=0,
\end{align}
supplemented with zero initial data ${H}_{n+1/2} |_{t=0}=0,$
where the operator $\mathbb{L}_H$ is defined by \eqref{NP1a}.
Since $(v^a+v_{n+1/2},\varphi^a+\psi_{n+1/2})$ satisfies the first boundary condition in \eqref{bas1b},
equations \eqref{MS.eq1} do not need to be supplemented with any boundary condition.

Noting that ${H}_{n+1/2}$ and $\psi_{{n}+1/2}$ vanish at the initial time,
by virtue of the second identity in \eqref{app2b},
one can show as the proof of Proposition \ref{pro:1} that
$(H^a+H_{n+1/2},\varphi^a+\psi_{n+1/2})$ satisfies the second constraint {in} \eqref{bas1b}.

We now estimate $H_{n+1/2}-\mathcal{S}_{\theta_n}H_{n}$.
We first utilize \eqref{MS.eq1} to find
\begin{align}
&\mathbb{L}_H({v}^{a }+{v}_{n+1/2} ,\,
{H}_{n+1/2} -\mathcal{S}_{\theta_n}H_{n},\,
{ \varPhi}^{a }+\varPsi_{n+1/2} )
=  \mathcal{I}_3+\mathcal{I}_4+\mathcal{I}_5, \label{MS.e3}
\end{align}
where $\mathcal{I}_{3}:=
-\mathcal{S}_{\theta_{n}}\mathbb{L}_{H}(v^a+  v_{n},
H^a+ H_{n},  \varPhi^a+ \varPsi_{{n}}),$ and
\begin{align*}
\mathcal{I}_{4}:=\;&
-\mathbb{L}_{H}(v^a+v_{n+1/2},
H^a+\mathcal{S}_{\theta_{n}} H_{n},
\varPhi^a+\varPsi_{{n}+1/2})\\ &
+\mathbb{L}_{H}(v^a+\mathcal{S}_{\theta_{n}}  v_{n},
H^a+\mathcal{S}_{\theta_{n}} H_{n},
\varPhi^a+\mathcal{S}_{\theta_{n}}  \varPsi_{{n}}),\\
\mathcal{I}_{5}:=\;&
-\mathbb{L}_{H}(v^a+\mathcal{S}_{\theta_{n}}  v_{n},
H^a+\mathcal{S}_{\theta_{n}} H_{n},
\varPhi^a+\mathcal{S}_{\theta_{n}}  \varPsi_{{n}})\\&
+\mathcal{S}_{\theta_{n}}\mathbb{L}_{H}(v^a+  v_{n},
H^a+ H_{n},  \varPhi^a+ \varPsi_{{n}}).
\end{align*}
Noting  $(\widetilde{ U}^a,\varphi^a)\in  H^{\widetilde{\alpha}+{6}}(\varOmega_T)\times  H^{\widetilde{\alpha}+{15/2}}(\varSigma_T)$ and using the Moser-type calculus inequalities \eqref{Moser1}--\eqref{Moser3},
we get from \eqref{smooth.p1}, \eqref{small}--\eqref{tri3}, \eqref{MS.est1}, and \eqref{MS.e2} that
\begin{align}
\| \mathcal{I}_4\|_{s,*,T}+\| \mathcal{I}_5\|_{s,*,T}\lesssim \varepsilon \theta_n^{s-{\alpha }+2}
\quad \textrm{for }\ s =6,\ldots, \widetilde{\alpha}+{4}.
\label{MS.p3}
\end{align}
Thanks to  \eqref{NM0},
we utilize Lemma~\ref{lem:Moser2}, \eqref{smooth.p1},
\eqref{small}--\eqref{tri3}, and hypothesis  ($\mathbf{H}_{n-1}$) to get
\begin{align}
\big\|\mathcal{I}_3-
\mathcal{S}_{\theta_n}\mathbb{L}_{H}(v^a+  v_{n-1},   H^a+ H_{n-1},  \varPhi^a+ \varPsi_{{n-1}})
\big\|_{s,*,T}\lesssim \varepsilon \theta_n^{s-{\alpha }+2}
\label{MS.p4}
\end{align}
for $s =6,\ldots, \widetilde{\alpha}+{4}.$
By virtue of \eqref{smooth.p1}, \eqref{small}--\eqref{tri3}, and hypothesis ($\mathbf{H}_{n-1}$),
we infer
\begin{align}
\nonumber & \|\mathcal{S}_{\theta_n}  \mathbb{L}_{H}(v^a+  v_{n-1},   H^a+ H_{n-1},  \varPhi^a+ \varPsi_{{n-1}})\|_{s,*,T}   \\
&\quad   \lesssim  \theta_{{n}}^{s-6}
\|\mathcal{L}(V_{n-1},\varPsi_{n-1}) \|_{6,*,T}
 \lesssim \varepsilon \theta_n^{s-{\alpha }-1}
\quad \textrm{for }\ s =6,\ldots, \widetilde{\alpha}+{4}.
\label{MS.p5}
\end{align}
Plugging the estimates \eqref{MS.p3}--\eqref{MS.p5} into \eqref{MS.e3} yields
\begin{align*}
\big\|\mathbb{L}_{H}({v}^{a }+{v}_{n+1/2} ,\,
{H}_{n+1/2} -\mathcal{S}_{\theta_n}H_{n},\,
{ \varPhi}^{a }+\varPsi_{n+1/2} )\big\|_{s,*,T}\lesssim \varepsilon \theta_n^{s-{\alpha }+2}
\end{align*}
for $s =6,\ldots, \widetilde{\alpha}+{4}.$
Employing the standard argument of the energy method and the Moser-type calculus inequalities \eqref{Moser1}--\eqref{Moser3},
we deduce
\begin{align}
\big\| {H}_{n+1/2} -\mathcal{S}_{\theta_n}H_{n} \big\|_{s,*,T}
\lesssim \varepsilon \theta_n^{s-{\alpha }+1}
\nonumber
\end{align}
for $s =6,\ldots, \widetilde{\alpha}+{4}.$
This completes the proof of the lemma.
\end{proof}

In the next lemma, we have the estimate for the second substitution error terms $e_{k}'''$ and  $\tilde{e}_{k}'''$ given in \eqref{decom1}--\eqref{decom2}.

\begin{lemma} \label{lem:2nd}
	Let ${\alpha }\geq 8$.
	If $\varepsilon>0$ is sufficiently small and $\theta_0\geq 1$ is suitably large, then
	\begin{align}\label{2st.sub}
	\tilde{e}_k'''=0,\quad  \| e_k'''\|_{s,*,T}
	\lesssim \varepsilon^2 \theta_k^{\varsigma_3(s)-1}\varDelta_k,
	\end{align}
	for all $k\in\{0,\ldots,{n}-1\}$ and $s\in\{6,\ldots,\widetilde{\alpha}-2\}$,
	where  $\varsigma_3(s):=\max\{(s+2-{\alpha })_++14-2{\alpha },s+10-2{\alpha } \}$.
\end{lemma}
\begin{proof}
It follows from \eqref{MS.id1} and \eqref{B''.form} that $\tilde{e}_k'''=0$.
For term ${e}_k'''$, we have
\begin{align}
\nonumber {e}_k'''=\int_{0}^{1}&
\mathbb{L}''\Big(U^a+\tau(\mathcal{S}_{\theta_k} V_k-V_{k+1/2})+V_{k+1/2},\,
\varPhi^a+\tau(\mathcal{S}_{\theta_k}\varPsi_k-\varPsi_{k+1/2})
\\
&\quad\ \ +\varPsi_{k+1/2} \Big)
\Big((\delta V_k ,\delta \varPsi_k),\, (\mathcal{S}_{\theta_k} V_k-V_{k+1/2},\mathcal{S}_{\theta_k}\varPsi_k-\varPsi_{k+1/2})\Big) \d\tau.
\nonumber
\end{align}
In view of \eqref{MS.est1}--\eqref{MS.est2}, similar to the proof of Lemmas \ref{lem:quad}--\ref{lem:1st},
we can apply Proposition \ref{pro:tame2} to obtain the estimate for ${e}_k'''$ in \eqref{2st.sub}.
\end{proof}

The following lemma gives the estimate of $D_{k+1/2}\delta\varPsi_k$ defined by \eqref{error.D}.
\begin{lemma}   \label{lem:last}
	Let ${\alpha }\geq 8$ and $\widetilde{\alpha}\geq {\alpha }+2$.
	If $\varepsilon>0$ is sufficiently small and $\theta_0\geq 1$ is suitably large, then
	\begin{align} \label{last.e0}
	\|D_{k+1/2}\delta\varPsi_k\|_{s,*,T}\lesssim \varepsilon^2 \theta_k^{\varsigma_4 (s)-1}\varDelta_k,
	\end{align}
	for all $k\in\{0,\ldots,{n}-1\}$ and $s\in\{6,\ldots,\widetilde{\alpha}-2\}$,
	where
	\begin{align} \label{varsigma4.def}
	\varsigma_4(s):=\max\{(s-{\alpha })_++18-2{\alpha },s+12-2{\alpha }\}.
	\end{align}
\end{lemma}
\begin{proof}
We proceed as in \cite{A89MR976971,CS08MR2423311}.
Denote $\varOmega_T^+:=(0,T)\times\varOmega$ and
$$R_k:=\p_1\mathbb{L}(U^a+V_{k+1/2},\varPhi^a+\varPsi_{k+1/2}).$$
Since
$\|{\varPsi}^a+\varPsi_{k+1/2}\|_{H_*^{s+2}(\varSigma_T)}\lesssim
\|{\varphi}^a+\psi_{k+1/2}\|_{H^{s}(\varSigma_T)}$
and $\delta\varPsi_k$ vanishes in the past,
we apply \eqref{Moser2} and \eqref{embed1} to get
\begin{align}
\nonumber \|D_{k+1/2}\delta\varPsi_k\|_{s,*,T}
=\;&\|D_{k+1/2}\delta\varPsi_k\|_{H_*^{s}(\varOmega_T^+)}\\
 \lesssim\;& \|\delta\varPsi_k\|_{4,*,T}\|R_k\|_{H_*^{s}(\varOmega_T^+)}
+ \|\delta\varPsi_k\|_{s,*,T}\|R_k\|_{4,*,T}
 \nonumber\\
& + \|\delta\varPsi_k\|_{4,*,T}\|R_k\|_{4,*,T}
\|{\varphi}^a+\psi_{k+1/2}\|_{H^{s}(\varSigma_T)}.
\label{err4.p1}
\end{align}
Regarding $R_k$, we use \eqref{f^a} and \eqref{P.new} to decompose
\begin{align}
R_k
=&\;\p_1\left(\mathcal{L}(V_k,\varPsi_k)-f^a+\mathcal{I}_6\right)
\quad \textrm{if }\ t>0,
\label{decom4}
\end{align}
where
\begin{align}
\nonumber
\mathcal{I}_6:=&
\int_{0}^1\mathbb{L}'\big(  U^a+V_{k}+\tau(V_{k+1/2}-V_k),
\varPhi^a+\varPsi_{k}  \\
\nonumber &  \quad\quad +\tau(\varPsi_{k+1/2}-\varPsi_{k})\big)
(V_{k+1/2}-V_k,\varPsi_{k+1/2}-\varPsi_{k})\d\tau.
\end{align}
If $s\in\{4,\ldots,\widetilde{\alpha}-4\}$, then hypothesis $(\mathbf{H}_{{n}-1})$ leads to
\begin{align} \label{last.e1}
\|\mathcal{L}(V_k,\varPsi_k)-f^a\|_{s+2,*,T}\leq 2\varepsilon\theta_k^{s+1-{\alpha }}.
\end{align}
Estimating the term $\mathcal{I}_6$ through
\begin{align*}
\|\mathbb{L}'(\widebar{U}+V_1,\widebar{\varPhi}+\varPsi_1) (V_2,\varPsi_2) \|_{s,*,T}
\lesssim \sum_{i\neq j}\|({V}_i,{\varPsi}_i) \|_{6,*,T}\|(V_j,\varPsi_j) \|_{s+2,*,T},
\end{align*}
we use \eqref{tri1}--\eqref{tri2}, \eqref{MS.id1}--\eqref{MS.est2}, and
\eqref{decom4}--\eqref{last.e1}  to derive
\begin{align} \label{last.e2}
\|R_k\|_{H^{s}_{*}(\varOmega_T^+)}
\lesssim \varepsilon
\left(\theta_{k}^{s+6-{\alpha }}
+\theta_{k}^{(s+4-{\alpha })_+ +8-{\alpha }}\right)
\quad \textrm{for }  s\in\{4,\ldots,\widetilde{\alpha}-4\}.
\end{align}
If $s\in\{\widetilde{\alpha}-3, \widetilde{\alpha}-2\}$,
then we use \eqref{tri3} and \eqref{MS.id1}--\eqref{MS.est2} to obtain
\begin{align*}
\|R_{k}\|_{s,*,T}
\lesssim \|(\tilde{U}^a+V_{k+1/2},\widetilde{\varPhi}^a+\varPsi_{k+1/2})\|_{s+4,*,T}
\lesssim \varepsilon \theta_k^{s+6-{\alpha }}.
\end{align*}
Consequently, the estimate \eqref{last.e2} holds for $s\in\{4,\ldots,\widetilde{\alpha}-2\}$.
Substituting hypothesis $(\mathbf{H}_{{n}-1})$, \eqref{last.e2}, \eqref{tri3}, and \eqref{MS.id1}--\eqref{MS.est2}
into \eqref{err4.p1} implies \eqref{last.e0}  and completes the proof.
\end{proof}

As a direct corollary to Lemmas \ref{lem:quad}--\ref{lem:last},
we have the estimate for $e_k$ and $\tilde{e}_k$ defined by \eqref{e.e.tilde} as follows.
\begin{corollary}  \label{cor:sum1}
Let ${\alpha }\geq 8$ and $\widetilde{\alpha}\geq {\alpha }+2$.
If $\varepsilon>0$ is sufficiently small and $\theta_0\geq 1$ is suitably large, then
\begin{align} \label{es.sum1a}
&	\|e_k\|_{s,*,T}
\lesssim \varepsilon^2 \theta_k^{\varsigma_4(s)-1}\varDelta_k,\\
\label{es.sum1b}
&	\|\tilde{e}_k\|_{H^{s}(\varSigma_T)}
\lesssim \varepsilon^2 \theta_k^{\varsigma_2(s)-1}\varDelta_k,
\end{align}
for all $k\in\{0,\ldots,{n}-1\}$ and $s\in\{6,\ldots,\widetilde{\alpha}-2\}$,
where  $\varsigma_4(s)$ and  $\varsigma_2(s)$ are defined by \eqref{varsigma4.def} and \eqref{varsigma2.def} respectively.
	
\end{corollary}

Corollary \ref{cor:sum1} implies
the following estimate for the accumulated error terms $E_n$ and $\widetilde{E}_n$
defined by \eqref{E.E.tilde}.

\begin{lemma}\ \label{lem:sum2}
Let ${\alpha }\geq 12$ and $\widetilde{\alpha}={\alpha }+3$.
If $\varepsilon>0$ is sufficiently small and $\theta_0\geq 1$ is suitably large, then
	\begin{align}\label{es.sum2}
	\|E_{{n}}\|_{{\alpha }+1,*,T}\lesssim \varepsilon^2 \theta_{{n}},\quad
	\|\widetilde{E}_{{n}}\|_{H^{{\alpha }+1}(\varSigma_T)}\lesssim \varepsilon^2.
	\end{align}
\end{lemma}
\begin{proof}
Notice that $\varsigma_4({\alpha }+1)\leq 1$ if ${\alpha }\geq 12$.
It follows from \eqref{es.sum1a} that
\begin{align*}
&\|E_{{n}}\|_{{\alpha }+1,*,T}
\lesssim \sum_{k=0}^{{n}-1}\|e_{k} \|_{{\alpha }+1,*,T}
\lesssim \sum_{k=0}^{{n}-1} \varepsilon^2 \varDelta_k
\lesssim \varepsilon^2\theta_{{n}},
\end{align*}
provided ${\alpha }\geq 12$ and ${\alpha }+1\leq \widetilde{\alpha}-2$.
Since $\varsigma_2({\alpha }+1)=9-{\alpha }\leq -3$ for ${\alpha }\geq 12$
and ${\alpha }+1\leq \widetilde{\alpha}-2$,
the estimate \eqref{es.sum1b} implies
\begin{align*}
\|\widetilde{E}_{{n}}\|_{H^{{\alpha }+1}(\varSigma_T)}
\lesssim \sum_{k=0}^{{n}-1} \|\tilde{e}_{k} \|_{H^{{\alpha }+1}(\varSigma_T)}
\lesssim \sum_{k=0}^{{n}-1} \varepsilon^2 \theta_{{k}}^{-4}\varDelta_k
\lesssim \varepsilon^2.
\end{align*}
The minimal possible $\widetilde{\alpha}$ is ${\alpha }+3$.
This completes the proof of the lemma.
\end{proof}

\subsection{Proof of Theorem \ref{thm:1}}

We first derive hypothesis $(\mathbf{H}_{{n}})$ from $(\mathbf{H}_{{n}-1})$.
For this purpose,  we derive the estimates of the source terms $f_{{n}}$ and $g_{{n}}$ given in \eqref{source}.

\begin{lemma}\  \label{lem:source}
	Let ${\alpha }\geq 12$ and $\widetilde{\alpha}={\alpha }+3$.
	If $\varepsilon>0$ is sufficiently small and $\theta_0\geq 1$ is suitably large, then
	\begin{align}
	\label{es.fl}&\|f_{{n}}\|_{s,*,T}
	\lesssim \varDelta_{{n}}\left(\theta_{{n}}^{s-{\alpha }-1} \|f^a\|_{{\alpha },*,T}
	+\varepsilon^2 \theta_{{n}}^{s-{\alpha }-1} +\varepsilon^2\theta_{{n}}^{\varsigma_4(s)-1}\right),\\
	\label{es.gl}&\|g_{{n}}\|_{H^{s+1}(\varSigma_T)}
	\lesssim  \varepsilon^2 \varDelta_{{n}}\left(\theta_{{n}}^{s-{\alpha }-1}+\theta_{{n}}^{\varsigma_2(s+1)-1}\right),
	\end{align}
	for $s\in\{6,\ldots,\widetilde{\alpha}\}$,
where  $\varsigma_4(s)$ and  $\varsigma_2(s)$ are defined by \eqref{varsigma4.def} and \eqref{varsigma2.def} respectively.
\end{lemma}
\begin{proof}
From \eqref{smooth.p1}, \eqref{source}, \eqref{es.sum1a}, and \eqref{es.sum2},
we obtain
\begin{align*}
\|f_n\|_{s,*,T}
&\leq
\|(\mathcal{S}_{\theta_{{n}}}-\mathcal{S}_{\theta_{{n}-1}})f^a
-(\mathcal{S}_{\theta_{{n}}}-\mathcal{S}_{\theta_{{n}-1}})E_{{n}-1}
-\mathcal{S}_{\theta_{{n}}}e_{{n}-1}\|_{s,*,T} \\
&
\lesssim \varDelta_{{n}}
\theta_{{n}}^{s-{\alpha }-1}\big(\|f^a\|_{{\alpha },*,T}
+\theta_{{n}}^{-1}\|E_{{n}-1}\|_{{\alpha }+1,*,T} \big)
+\|\mathcal{S}_{\theta_{{n}}} e_{{n}-1}\|_{s,*,T} \\
&
\lesssim  \varDelta_{{n}}\big\{\theta_{{n}}^{s-{\alpha }-1}(\|f^a\|_{{\alpha },*,T}
+\varepsilon^2)+\varepsilon^2\theta_{{n}}^{\varsigma_4(s)-1}\big\}.
\end{align*}
Thanks to \eqref{es.sum1b} and \eqref{es.sum2}, we obtain
\begin{align*}
\|g_n\|_{H^{s+1}(\varSigma_T)}
&\leq
\| (\mathcal{S}_{\theta_{{n}}}-\mathcal{S}_{\theta_{{n}-1}})\widetilde{E}_{{n}-1}
-\mathcal{S}_{\theta_{{n}}}\tilde{e}_{{n}-1}\|_{H^{s+1}(\varSigma_T)} \\
& \lesssim \varDelta_{{n}}
\theta_{{n}}^{s-{\alpha }-1} \|\widetilde{E}_{{n}-1}\|_{H^{{\alpha }+1}(\varSigma_T)}
+\|\mathcal{S}_{\theta_{{n}}} \tilde{e}_{{n}-1}\|_{H^{s+1}(\varSigma_T)} \\
& \lesssim \varepsilon^2 \varDelta_{{n}}\big(\theta_{{n}}^{s-{\alpha }-1}+\theta_{{n}}^{\varsigma_2(s+1)-1}\big),
\end{align*}
which completes the proof of the lemma.
\end{proof}

The next lemma gives the estimate of solutions to the problem \eqref{effective.NM}
by means of \eqref{MS.est1}--\eqref{MS.est2} and the tame estimate \eqref{tame}.
We omit the proof for brevity, since it is similar to the proof of \cite[Lemma 15]{T09ARMAMR2481071}.

\begin{lemma}\  \label{lem:Hl1}
	Let ${\alpha }\geq 12$ and $\widetilde{\alpha}={\alpha }+3$.
	If $\varepsilon>0$ and $\|f^a\|_{{\alpha },*,T}/\varepsilon$ are sufficiently small,
	and if $\theta_0\geq1$ is suitably large, then
	\begin{align} \label{Hl.a}
	\|(\delta V_{{n}},\delta\varPsi_{{n}})\|_{s,*,T}+\|\delta\psi_{{n}}\|_{H^{s}(\varSigma_T)}
	\leq \varepsilon \theta_{{n}}^{s-{\alpha }-1}\varDelta_{{n}}
	\end{align}
	for $s\in\{6,\ldots,\widetilde{\alpha}\}$.
\end{lemma}

The estimate \eqref{Hl.a} is the inequality (a) in hypothesis $(\mathbf{H}_{{n}})$. The other inequalities in $(\mathbf{H}_{{n}})$ are provided in the following lemma, whose proof can be found in \cite[Lemma 16]{T09ARMAMR2481071}.

\begin{lemma}\ \label{lem:Hl2}
	Let ${\alpha }\geq 12$ and $\widetilde{\alpha}={\alpha }+3$.
	If $\varepsilon>0$ and $\|f^a\|_{{\alpha },*,T}/\varepsilon$ are sufficiently small,
	and if $\theta_0\geq1$ is suitably large,
	then
	\begin{alignat}{3}\label{Hl.b}
&	\|\mathcal{L}( V_{{n}},  \varPsi_{{n}})-f^a\|_{s,*,T}\leq 2 \varepsilon \theta_{{n}}^{s-{\alpha }-1}
\quad &&\textrm{for }	s\in\{6,\ldots,\widetilde{\alpha}-1\},\\
\label{Hl.c}
&	\|\mathcal{B}( V_{{n}} ,  \psi_{{n}})\|_{H^{s}(\varSigma_T)}\leq  \varepsilon \theta_{{n}}^{s-{\alpha }-1}
\quad &&\textrm{for }	s\in\{7,\ldots,{\alpha }\}.
	\end{alignat}
\end{lemma}

According to Lemmas \ref{lem:Hl1}--\ref{lem:Hl2}, we have derived  $(\mathbf{H}_{{n}})$ from $(\mathbf{H}_{{n}-1})$,
provided that ${\alpha }\geq 12$, $\widetilde{\alpha}={\alpha }+3$,
$\varepsilon>0$, and $\|f^a\|_{{\alpha },*,T}/\varepsilon$ are sufficiently small,
and $\theta_0 \geq 1$ is large enough.
Fixing the constants ${\alpha }$, $\widetilde{\alpha}$, $\varepsilon>0$, and $\theta_0\geq1$,
we can show as in \cite[Lemma 17]{T09ARMAMR2481071} that $(\mathbf{H}_{0})$ is true for a sufficiently short time.

\begin{lemma}\ \label{lem:H0}
	If time $T>0$ is small enough,
	then $(\mathbf{H}_0)$ holds.
\end{lemma}

We are now in a position to conclude the proof of Theorem \ref{thm:1}.

\vspace*{3mm}
\noindent  {\bf Proof of Theorem {\rm\ref{thm:1}}.}\quad
Assume that the initial data $(U_0,\varphi_0)$ satisfy all the conditions of Theorem {\rm\ref{thm:1}}.
Let $\widetilde{\alpha}=m-{5}$ and ${\alpha }=\widetilde{\alpha}-3\geq 12$.
Then the initial data $U_0$ and $\varphi_0$ are compatible up to order $m=\widetilde{\alpha}+{5}$.
Thanks to \eqref{app1a} and \eqref{f^a:est}, we obtain \eqref{small} and all the requirements of Lemmas \ref{lem:Hl1}--\ref{lem:H0}, provided {that} $\varepsilon>0$ and $T>0$ {are} sufficiently small and $\theta_0\geq 1$ is large enough.
Hence, for sufficiently short time $T>0$, hypothesis $(\mathbf{H}_{{n}})$ holds for all ${n}\in\mathbb{N}$.
In particular, recalling \eqref{theta}, we have
\begin{align*}
\sum_{k=0}^{\infty}\left(\|(\delta V_k,\delta \varPsi_k)\|_{s,*,T}+\|\delta\psi_k\|_{H^{s}(\varSigma_T)} \right)
\lesssim \sum_{k=0}^{\infty}\theta_k^{s-{\alpha }-2} <\infty
\ \ \textrm{for } 6\leq s\leq {\alpha }-1.
\end{align*}
As a result, the sequence $(V_{k},\psi_k)$ converges to some limit $(V,\psi)$ in $H_*^{{\alpha }-1}(\varOmega_T)\times H^{{\alpha }-1}({\varSigma_T})\subset 
H^{\lfloor ({\alpha }-1)/2\rfloor}(\varOmega_T)\times H^{{\alpha }-1}({\varSigma_T})$.
Passing to the limit in \eqref{Hl.b}--\eqref{Hl.c} for $s={\alpha }-1=m-{9}$, we obtain \eqref{P.new}.
Therefore, $(U, \varphi)=(U^a+V, \varphi^a+\psi)$ is a solution of the original problem \eqref{NP1} on the time interval $[0,T]$.
The uniqueness of solutions to the problem \eqref{NP1} can be obtained by a standard argument; see, for instance, \cite[\S 13]{ST14MR3151094}.
This completes the proof.
\qed

\section{Relativistic Case}\label{sec:rela}

Let us first reduce the RMHD equations \eqref{RMHD1} to an equivalent symmetric hyperbolic system.
To this end, we introduce the coordinate velocity
$v:=(v_1,\ldots,v_d)^{\mathsf{T}}$, where
\begin{align} \label{v.def}
v_i:=\epsilon^{-1}\varGamma^{-1} u^{i}\quad
\textrm{for } i=1,\ldots,d,
\end{align}
with the Lorentz factor $\varGamma:=u^0>0$ thanks to \eqref{u.relation}.
Moreover, from the first identity in \eqref{u.relation}, we infer
\begin{align} \label{varGamma.def}
|v|<\epsilon^{-1},\quad
\varGamma=\varGamma(v):=(1-\epsilon^2|v|^2)^{-1/2}.
\end{align}
Let us define $H:=(H_1,\ldots,H_d)^{\mathsf{T}}$ with
\begin{align} \label{H.def}
H_i:=\epsilon^{-1}(u^0 b^i-u^i b^0)\quad \textrm{for } i=1,\ldots,d.
\end{align}
Then the relations \eqref{u.relation} imply
\begin{align}
\label{bi.def}
&b^i= \epsilon \varGamma^{-1} H_i +\epsilon^3 \varGamma (v\cdot H) v_i
\quad \textrm{for } i=1,\ldots,d,
\\
&b^0=\epsilon^2 \varGamma (v\cdot H),\quad
|b|^2:= g_{\alpha \beta} b^{\alpha} b^{\beta}
= \epsilon^2 \varGamma^{-2}|H|^2+\epsilon^4 (v\cdot H)^2
. \label{|b|}
\end{align}
In the inertial coordinates $(x^0,\ldots,x^d)$, the covariant derivative $\nabla_{\alpha}$ coincides with the partial derivative $\p/\partial{x^{\alpha}}$.
We introduce the spacetime coordinates $(t,x)$, where $t=\epsilon x^0$ and $x=(x_1,\ldots,x_d)$ with $x_i:=x^i$ for $i=1,\ldots,d$.
Then a straightforward computation leads to the following equivalent system of \eqref{RMHD1} ({\it cf.}~\cite[(2)--(6)]{FT13MR3044369}):
\begin{subequations} \label{RMHD}
	\begin{alignat}{2}
	\label{RMHD.a}&\p_t(\rho \varGamma)+\nabla\cdot (\rho \varGamma v)=0,&\\
	\label{RMHD.b}&\p_t(\rho h \varGamma^2+\epsilon^2|H|^2-\epsilon^2q)
	+\nabla \cdot \big( \rho h \varGamma^2 v+\epsilon^2|H|^2 v -\epsilon^2( v\cdot H)H \big)=0,&\\
	\nonumber &\p_t\big( \rho h \varGamma^2 v+\epsilon^2|H|^2 v -\epsilon^2( v\cdot H)H \big)
	-\nabla\cdot (\varGamma^{-2}H\otimes H)
	+\nabla q\\
	\label{RMHD.c}&\quad
	+\nabla\cdot \big( (\rho h\varGamma^2+\epsilon^2 |H|^2)v\otimes v
	-\epsilon^2 (v\cdot H)(H\otimes v+v\otimes H) \big)=0,&\\
	\label{RMHD.d}&\p_t H-\nabla\times (v\times H)=0,  &
	\end{alignat}
\end{subequations}
and
\begin{align} \label{divH=0}
\nabla\cdot H=0,
\end{align}
where $q:=p+\frac{1}{2}\epsilon^{-2}|b|^2$ is the total pressure.
As in \cite[\S 34]{Lichnerowicz67}, it follows from \eqref{Gibbs} that
$u^{\alpha} \nabla_{\alpha} S=0$ for smooth solutions to \eqref{RMHD1},
and hence
\begin{align}
\label{S.eq}
(\p_t +v\cdot \nabla) S=0.
\end{align}
Let us impose the physical assumption \eqref{cs.def}
and define $V:=(p,w,H,S)^{\mathsf{T}}$ with
\begin{align} \label{w.def}
w:=\varGamma(v)v=(1-\epsilon^2 |v|^2)^{-1/2} v.
\end{align}
By properly applying the Lorentz transformation as in \cite{FT13MR3044369},
we can obtain the following symmetric hyperbolic system
in the smooth non-vacuum region $\{ \rho_* <\rho <\rho^*\}$,
which is equivalent to \eqref{RMHD} (and also to \eqref{RMHD1}):
\begin{align}\label{RMHD.vec}
B_0(V)\p_t V+B_i(V)\p_i V=0,
\end{align}
where
\begin{align} \label{B.def}
B_0(V):=
\begin{pmatrix}
\dfrac{\varGamma}{\rho a^2}& \epsilon^2 v^{\mathsf{T}}& 0& 0\\[1mm]
\epsilon^2 v &\mathcal{A}_0&O_d& 0\\[1.5mm]
0& O_d&\mathcal{M}_0&0\\[1.5mm]
\w0&\w0&\w0&\w1
\end{pmatrix},\ 
B_i(V)=\begin{pmatrix}
\dfrac{\varGamma v_i}{\rho a^2}& \bm{e}_i^{\mathsf{T}}& 0& 0\\[1mm]
\bm{e}_i&\mathcal{A}_i&\mathcal{N}_i^{\mathsf{T}}& 0\\[1.5mm]
0& \mathcal{N}_i&v_i\mathcal{M}_0&0\\[1.5mm]
\w0&\w0&\w0&\w{v_i}
\end{pmatrix},
\end{align}
with
\begin{align*}
\mathcal{M}_0:= \;&\varGamma^{-1} (I_d+\epsilon^2\varGamma^2 v\otimes v) ,\\
\mathcal{A}_0:=\;&
\left(\rho h\varGamma+{\epsilon^2 \varGamma^{-1}|H|^2} \right)(I_d-\epsilon^2v\otimes v )
- \epsilon^2 \varGamma^{-1} |b|^2  v\otimes v\\
&- \epsilon^2 \varGamma^{-1}  H\otimes H
+ \epsilon^4 \varGamma^{-1}   (v\cdot H )(H\otimes v+v\otimes H),\\[1mm]
\mathcal{A}_i:=\;&
v_i\left\{
\left(\rho h\varGamma+{\epsilon^2\varGamma^{-1} |H|^2}\right)(I_d-\epsilon^2v\otimes v )
+\epsilon^2{\varGamma^{-1}} (|b|^2 v\otimes v  -H\otimes H)\right\}\\
&
+{\epsilon^2 {\varGamma^{-1}} H_i}
\left\{\varGamma^{-2} {(H\otimes v+v\otimes H)}
-2(v\cdot H )(I_d-\epsilon^2 v\otimes v) \right\}\\
&+{\varGamma^{-1}} (\epsilon^2 (v\cdot H )H-|b|^2v)\otimes  {\bm{e}_i}
+{\varGamma^{-1}} {\bm{e}_i} \otimes(\epsilon^2 (v\cdot H )H-|b|^2v)  ,\\[1mm]
\mathcal{N}_i:= \;&
\left(\varGamma^{-2} {H}+\epsilon^2(v\cdot H )v \right)\otimes
\left(\bm{e}_i- \epsilon^2 v_i v \right)- {\varGamma^{-2}} {H_i} I_d,
\quad  \textrm{for }i=1,\ldots,d.
\end{align*}
Here, $a$ and $|b|^2$ are given in \eqref{a.def} and \eqref{|b|} respectively.
See Appendix \ref{sec:appB} for a direct verification of \eqref{RMHD.vec}
and the positive-definiteness of $B_0(V)$.

We consider the free boundary problem in ideal RMHD, that is, to solve the equations \eqref{RMHD} (or equivalently \eqref{RMHD1}) in $\varOmega(t)$ supplemented with the initial and boundary conditions
\eqref{IC1}--\eqref{BC1}.
As in the non-relativistic case, we take $U:=(q,v,H,S)^{\mathsf{T}}$ as the primary unknowns and deduce from \eqref{RMHD.vec} the hyperbolic symmetric system \eqref{MHD.vec},
with coefficient matrices $A_{\alpha}(U)$ being defined by
\setlength{\arraycolsep}{4.5pt}
\begin{align}
\label{Ai.def.r}
A_{\alpha}(U):= {\bm{J}}^{\mathsf{T}} B_{\alpha}(V) {\bm{J}}
\ \ \textrm{with }{\bm{J}}:=\dfrac{\p V}{\p U}
=\begin{pmatrix}
1 &  \bm{a}^{\mathsf{T}} & -\bm{b}^{\mathsf{T}} &0\\
0 &   \varGamma^2\mathcal{M}_0 & O_d &0\\
0 & O_d & I_d &0\\
\w{0} & \w{0} & \w{0}& \w{1}
\end{pmatrix},
\end{align}
for ${\alpha}=0,\ldots,d,$
where  $B_{{\alpha}}(V)$ are given in \eqref{B.def},
$\bm{a}:=\epsilon^2 (|H|^2 v -(v\cdot H) H ),$
and $\bm{b}:=\varGamma^{-2}{H}+\epsilon^2 (v\cdot H)v.$
We can compute $\det {\bm{J}}=\varGamma^5>0$,
meaning that ${\bm{J}}$ is invertible.
In the new variables, from \eqref{NP1b} and \eqref{A1t.def}--\eqref{HN}, we have
\begin{align*}
\widetilde{A}_1(U, \varPhi):=
{\bm{J}}^{\mathsf{T}}\begin{pmatrix}
0 & \bm{c}^{\mathsf{T}} & 0 &0 \\
\bm{c}
&  \widetilde{\mathcal{A}}_1 & \bm{c}\otimes\bm{b}&0\\
0 & \bm{b}\otimes \bm{c} & O_d & 0\\
\w{0} & \w{0} & \w{0} & \w{0}
\end{pmatrix}{\bm{J}}\quad
\textrm{on } [0,T]\times\varSigma,
\end{align*}
where $\bm{c}:=N-\epsilon^2 v_N v$ and
\begin{align*}
\widetilde{\mathcal{A}}_1:=\;&{\epsilon^2 {\varGamma^{-1}} v_N}\left\{
2 |b|^2 v\otimes v -\epsilon^2 (v\cdot H) (v\otimes H+H\otimes v)
\right\}\\
&+ \varGamma^{-1} (\epsilon^2 (v\cdot H )H-|b|^2v)\otimes N+
\varGamma^{-1}N \otimes  (\epsilon^2 (v\cdot H )H-|b|^2v).
\end{align*}
From \eqref{varGamma.def}, we get
$\varGamma \mathcal{M}_0\bm{c}=   N$
and $\mathcal{M}_0\widetilde{\mathcal{A}}_1\mathcal{M}_0
=-\varGamma^{-3} (N\otimes \bm{a} +\bm{a} \otimes N)$,
so that
\begin{align} \label{A1t.rela}
\widetilde{A}_1(U, \varPhi)=
 \begin{pmatrix}
0 & \varGamma N^{\mathsf{T}} & 0 &0 \\
\varGamma N
&  O_d &O_d  &0\\
0 & O_d  & O_d & 0\\
\w{0} & \w{0} & \w{0} & \w{0}
\end{pmatrix} \quad
\textrm{on } [0,T]\times\varSigma.
\end{align}
Having the identity \eqref{A1t.rela}
and the symmetric form \eqref{MHD.vec}
with the matrices $A_{{\alpha}}(U)$ defined by \eqref{Ai.def.r} in hand,
we can prove Theorem \ref{thm:2}
in an entirely similar way as for the non-relativistic case in Sections \ref{sec:linear} and \ref{sec:Nash}.

\appendix
\titleformat{\section}{\large\bfseries\centering}{Appendix \thesection}{1em}{}
\section[Conventional Notation in the Vector Calculus]{Conventional Notation in the Vector Calculus} \label{sec:appA}

For readers' convenience, we collect the conventional notation in the vector calculus.
The spatial dimension is denoted by $d=2,3$. We abbreviate the partial differentials as
\begin{align*}
\p_t:=\frac{\p}{\p t},\quad \p_i:=\frac{\p }{\p x_i}
\quad \textrm{for }i=1,\ldots,d.
\end{align*}
We denote the gradient by $\nabla:=(\p_1,\ldots,\p_d)^{\mathsf{T}}$.
For any $d\times d$ matrix $F=(F_{ij})$,
vectors $u=(u_1,\ldots,u_d)^{\mathsf{T}}$ and
$v=(v_1,\ldots,v_d)^{\mathsf{T}}$,
and scalar $a$,
the symbol $u\otimes v$ denotes the $d\times d$ matrix with $(i,j)$-entry $u_i v_j$,
and
\begin{align*}
&\nabla\cdot F:=  (\p_j F_{1j},\ldots,\p_j F_{dj})^{\mathsf{T}},
\quad \nabla\cdot u:= \p_i u_{i},\\
&u\times v:=
\left\{
\begin{aligned}
&u_1 v_2 -u_2 v_1\quad & \textrm{if }d=2,\\
&(u_2 v_3 -u_3 v_2,u_3 v_1 -u_1 v_3,u_1 v_2 -u_2 v_1)^{\mathsf{T}}\quad & \textrm{if }d=3,
\end{aligned}
\right.\\
&\nabla\times v:=
\left\{
\begin{aligned}
&\p_1 v_2 -\p_2 v_1\quad & \textrm{if }d=2,\\
&(\p_2 v_3 -\p_3 v_2,\p_3 v_1 -\p_1 v_3,\p_1 v_2 -\p_2 v_1)^{\mathsf{T}}\quad & \textrm{if }d=3,
\end{aligned}
\right.\\[0.5mm]
&u\times a = -a\times u:=(u_2 a, -u_1 a)^{\mathsf{T}},
\quad
\nabla\times a :=(\p_2a , -\p_1 a )^{\mathsf{T}}
\quad \textrm{if }d=2.
\end{align*}
The notation above was employed by {\sc Kawashima} \cite[p.\;144]{Kawashima84}
to write down the electromagnetic fluid system in two spatial dimensions.
The compressible MHD equations \eqref{MHD1} with $d=2$ follow
from the assumption that all the quantities in \eqref{MHD1} are independent of $x_3$
and the components $v_3$ and $H_3$ are identically zero.

\section[Symmetrization for RMHD]{Symmetrization for RMHD} \label{sec:appB}

Let us deduce the symmetric system \eqref{RMHD.vec} from \eqref{RMHD}--\eqref{S.eq}.
First, the last equation for $S$ in \eqref{RMHD.vec} is exactly \eqref{S.eq}.
In view of \eqref{varGamma.def} and \eqref{w.def}, we have
$\varGamma=(1+\epsilon^2|w|^2)^{1/2}$ and $v=\varGamma^{-1}w$, so that
\begin{align} \label{appB.1}
\p_{\alpha} \varGamma=\epsilon^2 v_i\p_{\alpha} w_i,\ \
\p_{\alpha} v=\varGamma^{-1}\p_{\alpha}  w-\epsilon^2 \varGamma^{-1} v v_i \p_{\alpha}  w_i
\quad\textrm{for }\alpha=0,\ldots, d.
\end{align}
It follows from the identities \eqref{S.eq}, \eqref{RMHD.a}, and \eqref{appB.1} that
\begin{align*}
\varGamma (\p_t+v_i\p_i) p=-\rho a^2(\epsilon^2 v_i \p_t w_i +\p_i w_i ),
\end{align*}
which immediately gives the first equation for $p$ in \eqref{RMHD.vec}.
Using \eqref{RMHD.d} and \eqref{divH=0}, we have
$(\p_t+v\cdot \nabla) H-(H\cdot \nabla)v+H\nabla \cdot v=0$,
which together with \eqref{appB.1} yields
\begin{align} \label{appB.2}
(\p_t+v\cdot \nabla) H +\mathcal{M}_i \p_i w=0,
\end{align}
with $
\mathcal{M}_i:=\varGamma^{-1}
\{
H\otimes \bm{e}_i -H_i I_d -\epsilon^2 (v_i H- H_i v )\otimes v
\}.
$
Thanks to \eqref{RMHD.b} and \eqref{appB.1}, we infer from \eqref{RMHD.c} that
\begin{align} \nonumber
\left(\rho h\varGamma+ {\epsilon^2{\varGamma^{-1}} |H|^2} \right)
(I_d -\epsilon^2 v\otimes v)(\p_t +v\cdot \nabla) w+
\epsilon^2 v\p_t p+\nabla p+\sum_{i=1}^{4}\mathcal{T}_i=0,
\end{align}
where 
 \begin{align*}
 \mathcal{T}_1:=\;& \frac{1}{2} v\p_t|b|^2 -\epsilon^2\p_t ((v\cdot H)H ),\\
\mathcal{T}_2:=\;&  -\nabla\cdot (\varGamma^{-2}H\otimes H),\quad
\mathcal{T}_3:=\frac{1}{2}\epsilon^{-2}\nabla |b|^2,\\
\mathcal{T}_4:=
\;&-\epsilon^2 \nabla\cdot ( (v\cdot H)(H\otimes v+v\otimes H) )
+\epsilon^2 v\nabla\cdot ( (v\cdot H) H) \\
=\;&-\epsilon^2
(v\cdot H)( (v\cdot\nabla)H+(H\cdot\nabla)v+H\nabla\cdot v )
-\epsilon^2   H  (v\cdot\nabla) (v\cdot H)  .
\end{align*}
By virtue of \eqref{|b|}, \eqref{appB.1}, and \eqref{divH=0}, we deduce
\begin{align*}
&\mathcal{T}_1=
\epsilon^2 \varGamma^{-1}
\big\{\epsilon^2 (v\cdot H) (H\otimes v+v\otimes H)- |b|^2  v\otimes v-H\otimes H\big\}\p_t w
+\mathcal{T}_{1a},\\
&\mathcal{T}_2=
-\varGamma^{-2} (H\cdot \nabla)H_i +2\epsilon^2 \varGamma^{-3} H_i H\otimes v \p_i w, \\
&\mathcal{T}_3=
(\varGamma^{-2} H_i+ \epsilon^2 (v\cdot H) v_i )\nabla H_i
+\varGamma^{-1}\bm{e}_i\otimes (\epsilon^2(v\cdot H) H-|b|^2 v  )\p_i w,
\end{align*}
with
$\mathcal{T}_{1a}:=\epsilon^2 (v\varGamma^{-2} H_i
+\epsilon^2 (v\cdot H) v v_i -H v_i)\p_t H_i-\epsilon^2 (v\cdot H)\p_t H.$
Then we utilize \eqref{appB.1} and \eqref{appB.2} for
calculating the terms $\mathcal{T}_4$ and $\mathcal{T}_{1a}$ respectively
to derive the equations for $w$ in \eqref{RMHD.vec}.
Noticing that $\mathcal{M}_0\mathcal{M}_i=\mathcal{N}_i$,
we obtain the equations for $H$ in \eqref{RMHD.vec} from
the left-multiplication of \eqref{appB.2} by $\mathcal{M}_0$.

Next we show that the matrix $B_0(V)$ is positive definite in the non-vacuum region $\{\rho_*<\rho <\rho^* \}$.
For any $\bm{u}\in\mathbb{R}^d\setminus \{ 0\}$,
we get $\bm{u}^{\mathsf{T}}\mathcal{M}_0\bm{u}\geq \varGamma^{-1} |\bm{u}|^2$,
that is, $\mathcal{M}_0\geq \varGamma^{-1}I_d$.
Since
\setlength{\arraycolsep}{3.5pt}
\begin{align*}
\begin{pmatrix}
\dfrac{\varGamma}{\rho a^2}& \epsilon^2 v^{\mathsf{T}} \\[4mm]
\epsilon^2 v &\mathcal{A}_0
\end{pmatrix}
=
\bm{P}^{\mathsf{T}}\begin{pmatrix}
\dfrac{\varGamma}{\rho a^2}& 0 \\[1.5mm]
0 &\mathcal{A}_0-\epsilon^4\dfrac{\rho a^2}{\varGamma} v v^{\mathsf{T}}
\end{pmatrix}\bm{P}
\ \ \textrm{with }
\bm{P}:=\begin{pmatrix}
1 & \dfrac{\rho a^2}{\varGamma}\epsilon^2 v^{\mathsf{T}} \\[4mm]
0 &I_d
\end{pmatrix},
\end{align*}
it suffices to show that the matrix
$\mathcal{A}_0-\epsilon^4 \rho a^2 \varGamma^{-1} v v^{\mathsf{T}}$ is positive definite.
For any $\bm{u}\in\mathbb{R}^d\setminus \{ 0\}$, we have
\begin{align*}
\bm{u}^{\mathsf{T}}\big(\mathcal{A}_0-\epsilon^4 \rho a^2 \varGamma^{-1} v v^{\mathsf{T}}\big)\bm{u}
=\mathcal{T}_5+\mathcal{T}_6,
\end{align*}
where $\mathcal{T}_5:=\rho h \varGamma |\bm{u}|^2   -(\epsilon^2\rho h \varGamma +\epsilon^4 \rho a^2 \varGamma^{-1}) (v\cdot \bm{u})^2$ and
\begin{align*}
\mathcal{T}_6:=\;& {\epsilon^2}{\varGamma^{-1}} \big\{  |\bm{u}|^2|H|^2  -\epsilon^2   (1+\varGamma^{-2})|H|^2 (v\cdot \bm{u})^2\\
&\,\quad \qquad -\epsilon^4 (v\cdot \bm{u})^2 (v\cdot H)^2   -(\bm{u}\cdot H)^2+2\epsilon^2 (v\cdot H)(v\cdot \bm{u})(H\cdot \bm{u}) \big\}\\
=\;&
{\epsilon^2}{\varGamma^{-1}}\big\{  |\bm{u}|^2|H|^2
-\epsilon^2 (v\cdot \bm{u})^2  (1+\varGamma^{-2})|H|^2
-(H\cdot (\epsilon^2(v\cdot\bm{u})v-\bm{u} ) )^2\big\}\\
\geq\;&
{\epsilon^2}{\varGamma^{-1}}|H|^2
\big\{  |\bm{u}|^2-\epsilon^2 (2-\epsilon^2 |v|^2)(v\cdot \bm{u})^2 -
|\epsilon^2(v\cdot\bm{u})v-\bm{u}|^2\big\}= 0,
\end{align*}
owing to \eqref{varGamma.def}.
By virtue of \eqref{cs.def} and \eqref{varGamma.def}, we infer
\begin{align*}
\mathcal{T}_5\geq \;& \rho h \varGamma |\bm{u}|^2
-(\epsilon^2\rho h \varGamma +\epsilon^4 \rho a^2 \varGamma^{-1}) |v|^2|\bm{u}|^2\\
=\;&\rho h \varGamma^{-1} |\bm{u}|^2
\left(1 -c_{\rm s}^2 \epsilon^4 |v|^2\right)
\geq \rho h \varGamma^{-1} |\bm{u}|^2 (1-\epsilon^2 |v|^2)
=\rho h \varGamma^{-3} |\bm{u}|^2.
\end{align*}
Therefore, we obtain that
$\mathcal{A}_0-\epsilon^4 \rho a^2 \varGamma^{-1} v v^{\mathsf{T}}
$
and $B_0(V)$ are positive definite.


\bigskip

\noindent{\it Acknowledgements.} \   The authors would like to thank the referees for helpful comments and suggestions that contributed to improving the quality
of redaction.

\bigskip
\noindent{\bf Compliance with Ethical Standards}

\vspace*{2mm}
\noindent {\bf Conflict of interest} \  The authors declare that they have no conflict of interest.




\end{document}